\pdfoutput=1
\documentclass[a4paper,reqno]{amsart}
\usepackage[margin=3cm]{geometry}

\newcommand{\ifarticle}[2]{
    \csname@ifclassloaded\endcsname{beamer}{#2}{#1}
}

    \usepackage{xparse} 
    \usepackage{xspace} 
    \usepackage{etoolbox} 
    \usepackage{xpatch} 
    \usepackage{pgffor} 

    \usepackage{amsmath,amssymb,amsthm} 
    \allowdisplaybreaks[1]

    \usepackage{mathtools} 
    \usepackage{mathrsfs} 
    \usepackage{stmaryrd} 
    \usepackage{bbm} 
    \usepackage{quiver} 
    \usepackage{xcolor} 
    \usepackage{scalerel} 
    \usepackage{booktabs} 
    \usepackage{nameref} 
    \usepackage[british]{babel} 
    \usepackage{csquotes} 
    \usepackage[UKenglish]{isodate} 
    \usepackage{microtype} 
    \usepackage[splitrule]{footmisc} 

    \ifarticle{
        \PassOptionsToPackage{hyphens}{url}
        \usepackage[pdfusetitle,linktoc=all,colorlinks,citecolor=cyan,linkcolor=purple,urlcolor=blue]{hyperref} 
        \urlstyle{rm}

        \usepackage[nameinlink,noabbrev,capitalize]{cleveref} 
        \usepackage{footnotebackref} 

        \usepackage[shortlabels]{enumitem} 
        \setlist{topsep=2pt,itemsep=2pt,partopsep=2pt,parsep=2pt} 
    }{}
    \usepackage[style=alphabetic,backref=true,backrefstyle=three,maxnames=9,minalphanames=3]{biblatex} 
    \DeclareUnicodeCharacter{0301}{\TODO[Invalid symbol.]} 

    \DeclareFieldFormat*{citetitle}{\emph{#1}} 
    \DeclareCiteCommand{\citetitle}
        {\boolfalse{citetracker}%
        \boolfalse{pagetracker}%
        \usebibmacro{prenote}}
        {\ifciteindex
            {\indexfield{indextitle}}
            {}%
        \printtext[bibhyperref]{\printfield[citetitle]{labeltitle}}}
        {\multicitedelim}
        {\usebibmacro{postnote}}
    \DeclareCiteCommand*{\citeauthor}
        {\defcounter{maxnames}{99}%
        \defcounter{minnames}{99}%
        \defcounter{uniquename}{2}%
        \boolfalse{citetracker}%
        \boolfalse{pagetracker}%
        \usebibmacro{prenote}}
        {\ifciteindex{\indexnames{labelname}}{}%
        \printnames{labelname}}
        {\multicitedelim}
        {\usebibmacro{postnote}}

    \usepackage{calc} 
    \usepackage{tabularray} 
    \usepackage{ebproof} 
    \usepackage{forest} 
    \usepackage{adjustbox} 
    \usepackage{float} 
    \usepackage{stfloats}

    \makeatletter
    \ifarticle{
        \xpretocmd{\@adminfootnotes}{\let\@makefntext\BHFN@OldMakefntext}{}{}
        \renewcommand\@makefntext[1]{%
        \@ifundefined{@makefnmark}
            {}
            {%
            \renewcommand\@makefnmark{%
            \mbox{%
                \textsuperscript{%
                \normalfont
                \hyperref[\BackrefFootnoteTag]{\@thefnmark}%
                }%
            }\,%
            }%
            \BHFN@OldMakefntext{#1}%
        }%
        }

        \LetLtxMacro{\BHFN@Old@footnotemark}{\@footnotemark}
        \renewcommand*{\@footnotemark}{%
            \refstepcounter{BackrefHyperFootnoteCounter}%
            \xdef\BackrefFootnoteTag{bhfn:\theBackrefHyperFootnoteCounter}%
            \label{\BackrefFootnoteTag}%
            \BHFN@Old@footnotemark
        }

        \makeatletter
        \def\paragraph{\@startsection{paragraph}{4}%
          \z@\z@{-\fontdimen2\font}%
          {\normalfont\bfseries}}
        \makeatother
    }{}
    \makeatother

    \ifarticle{
        \theoremstyle{plain}
        \newtheorem{theorem}{Theorem}[section]
        \newtheorem{proposition}[theorem]{Proposition}
        \newtheorem{lemma}[theorem]{Lemma}
        \newtheorem{corollary}[theorem]{Corollary}

        \newtheorem*{theorem*}{Theorem}

        \theoremstyle{definition}
        \newtheorem{definition}[theorem]{Definition}
        
        \newtheorem{example}[theorem]{Example}

        \newtheorem{remark}[theorem]{Remark}

        \newenvironment{sketch}{\proof}{\endproof}

        \Crefname{theoremenumi}{Theorem}{Theorems}
        \AtBeginEnvironment{theorem}{%
            \crefalias{enumi}{theoremenumi}%
            \setlist[enumerate,1]{
                ref={\csname thetheorem\endcsname.(\arabic*)}
            }%
            \crefalias{enumii}{theoremenumi}%
            \setlist[enumerate,2]{
                ref={\thetheorem.(\arabic*).(\alph*)}
            }%
        }

        \Crefname{propositionenumi}{Proposition}{Propositions}
        \AtBeginEnvironment{proposition}{%
            \crefalias{enumi}{propositionenumi}%
            \setlist[enumerate,1]{
                ref={\csname theproposition\endcsname.(\arabic*)}
            }%
            \crefalias{enumii}{propositionenumi}%
            \setlist[enumerate,2]{
                ref={\theproposition.(\arabic*).(\alph*)}
            }%
        }

        \Crefname{lemmaenumi}{Lemma}{Lemmas}
        \AtBeginEnvironment{lemma}{%
            \crefalias{enumi}{lemmaenumi}%
            \setlist[enumerate,1]{
                ref={\csname thelemma\endcsname.(\arabic*)}
            }%
            \crefalias{enumii}{lemmaenumi}%
            \setlist[enumerate,2]{
                ref={\thelemma.(\arabic*).(\alph*)}
            }%
        }

        \Crefname{corollaryenumi}{Corollary}{Corollaries}
        \AtBeginEnvironment{corollary}{%
            \crefalias{enumi}{corollaryenumi}%
            \setlist[enumerate,1]{
                ref={\csname thecorollary\endcsname.(\arabic*)}
            }%
            \crefalias{enumii}{corollaryenumi}%
            \setlist[enumerate,2]{
                ref={\thecorollary.(\arabic*).(\alph*)}
            }%
        }

        \Crefname{definitionenumi}{Definition}{Definitions}
        \AtBeginEnvironment{definition}{%
            \crefalias{enumi}{definitionenumi}%
            \setlist[enumerate,1]{
                ref={\csname thedefinition\endcsname.(\arabic*)}
            }%
            \crefalias{enumii}{definitionenumi}%
            \setlist[enumerate,2]{
                ref={\thedefinition.(\arabic*).(\alph*)}
            }%
        }

        \Crefname{exampleenumi}{Example}{Examples}
        \AtBeginEnvironment{example}{%
            \crefalias{enumi}{exampleenumi}%
            \setlist[enumerate,1]{
                ref={\csname theexample\endcsname.(\arabic*)}
            }%
            \crefalias{enumii}{exampleenumi}%
            \setlist[enumerate,2]{
                ref={\theexample.(\arabic*).(\alph*)}
            }%
        }

        \foreach \env in {definition,remark,example}{
        \AtBeginEnvironment{\env}{%
          \pushQED{\qed}%
        }
        \AtEndEnvironment{\env}{\popQED\endexample}
        }
    }{}

    \DeclareDocumentCommand{\mathcommand}{mO{0}m}{%
    \expandafter\let\csname old\string#1\endcsname=#1
    \expandafter\newcommand\csname new\string#1\endcsname[#2]{#3}
    \DeclareRobustCommand#1{%
        \ifmmode
        \expandafter\let\expandafter\next\csname new\string#1\endcsname
        \else
        \expandafter\let\expandafter\next\csname old\string#1\endcsname
        \fi
        \next
    }%
    }

    \mathcommand{\h}{\textup{-}}

    \newcommand{\tx}{\mathrm}
    \mathcommand{\b}{\mathbf}
    
    \newcommand{\cl}{\mathcal}
    \mathcommand{\bb}{\mathbb}
    \DeclareMathAlphabet{\bbn}{U}{bbold}{m}{n}
    \newcommand{\dc}[1]{{\b{\bb#1}}}
    
    \mathcommand{\sf}{\mathsf}
    \mathcommand{\u}{\underline}

    \newcommand{\TODO}[1][TODO]{\textcolor{orange}{\textup{#1}}\xspace}

    \newcommand{\flip}[1]{\text{\rotatebox[origin=c]{-180}{$#1$}}}

    \newcommand{\datetoday}{\date{\cleanlookdateon\today}}

    \newcommand{\defeq}{\mathrel{:=}}
    \mathcommand{\d}{\mathbin{;}}
    \mathcommand{\c}{\circ}
    \newcommand{\ph}[1][]{{({-}_{#1})}}
    
    \newcommand{\iso}{\cong}
    \renewcommand{\equiv}{\simeq}

    \newcommand{\tto}{\Rightarrow}
    
    \newcommand{\xtto}{\xRightarrow}
    
    \newcommand{\ffto}{\hookrightarrow}
    
    \newcommand*\cocolon{%
            \nobreak
            \mskip6mu plus1mu
            \mathpunct{}%
            \nonscript
            \mkern-\thinmuskip
            {:}%
            \mskip2mu
            \relax
    }

    \makeatletter
    \def\slashedarrowfill@#1#2#3#4#5{%
    $\m@th\thickmuskip0mu\medmuskip\thickmuskip\thinmuskip\thickmuskip
    \relax#5#1\mkern-7mu%
    \cleaders\hbox{$#5\mkern-2mu#2\mkern-2mu$}\hfill
    \mathclap{#3}\mathclap{#2}%
    \cleaders\hbox{$#5\mkern-2mu#2\mkern-2mu$}\hfill
    \mkern-7mu#4$%
    }
    \def\rightslashedarrowfill@{%
    \slashedarrowfill@\relbar\relbar\mapstochar\rightarrow}
    \newcommand\xslashedrightarrow[2][]{%
    \ext@arrow 0055{\rightslashedarrowfill@}{#1}{#2}}
    \def\leftslashedarrowfill@{%
    \slashedarrowfill@\leftarrow\relbar\mapsfromchar\relbar}
    \newcommand\xslashedleftarrow[2][]{%
    \ext@arrow 0055{\leftslashedarrowfill@}{#1}{#2}}
    \makeatother

    \newcommand{\xlto}{\xslashedrightarrow}
    \newcommand{\lto}{\xlto{}}

    \newcommand{\op}{{}^\tx{op}}
    \newcommand{\co}{{}^\tx{co}}

    \newcommand{\tp}[1]{\langle#1\rangle}
    \newcommand{\unit}{{\tp{}}}
    
    \newcommand{\rx}{\mathbin{\blacktriangleright}}
    \newcommand{\plx}{\mathbin{{\rhd}\mathclap{\mspace{-17.5mu}\cdot}}}

    \newcommand{\adj}{\dashv}
    \newcommand{\radj}[1]{\mathrel{\adj_{#1}}}

    \newcommand{\ob}[1]{|#1|}

    \DeclareFontFamily{U}{min}{}
    \DeclareFontShape{U}{min}{m}{n}{<-> udmj30}{}
    \newcommand{\yo}{\!\text{\usefont{U}{min}{m}{n}\symbol{'210}}\!}

    \mathcommand{\comma}{\downarrow}

    \newcommand{\copi}{\flip\pi}
    
    \newsavebox{\whitecircstar}\sbox{\whitecircstar}{\kern.075em\tikz{\node[draw, circle,line width=.36pt, inner sep=0]{$*$};}\kern.075em}
    \newcommand{\ostar}{\mathbin{\scalerel*{\usebox{\whitecircstar}}{\odot}}}
    \newsavebox{\blackcircstar}\sbox{\blackcircstar}{\kern.075em\tikz{\node[fill, circle, line width=.36pt, inner sep=0, text=white]{$*$};}\kern.075em}
    \newcommand{\bulletstar}{\mathbin{\scalerel*{\usebox{\blackcircstar}}{\odot}}}

    \makeatletter
    \def\widebreve{\mathpalette\wide@breve}
    \def\wide@breve#1#2{\sbox\z@{$#1#2$}%
         \mathop{\vbox{\m@th\ialign{##\crcr
    \kern0.08em\brevefill#1{0.8\wd\z@}\crcr\noalign{\nointerlineskip}%
                        $\hss#1#2\hss$\crcr}}}\limits}
    \def\brevefill#1#2{$\m@th\sbox\tw@{$#1($}%
      \hss\resizebox{#2}{\wd\tw@}{\rotatebox[origin=c]{90}{\upshape(}}\hss$}
    \makeatother

    \NewDocumentCommand{\jrule}{om}{%
        \IfNoValueTF{#1}
            {\textsc{#2}}
            {$#1$-\textsc{#2}}%
    }

    \newcommand{\N}{{\bb N}}
    
    \newcommand{\Set}{{\b{Set}}}
    \newcommand{\FinSet}{{\b{FinSet}}}
    \newcommand{\V}{{\bb V}} 
    \newcommand{\Cat}{\b{Cat}}
    \newcommand{\CAT}{\b{CAT}}
    \newcommand{\VCat}{\V\h\Cat}

    \newcommand{\RMnd}{\b{RMnd}}
    
    \newcommand{\Alg}{\b{Alg}}
    
    \newcommand{\ff}{fully faithful}

    \newcommand{\ioo}{identity-on-objects}

    \newcommand{\EM}{Eilenberg--Moore}

    \newcommand{\eg}{e.g.\@\xspace}
    \newcommand{\ie}{i.e.\@\xspace}
    
    \newcommand{\cf}{cf.\@\xspace}

    \NewDocumentCommand{\etc}{t.}{etc.\@\xspace}
    \NewDocumentCommand{\ibid}{t.}{ibid.\@\xspace}
    \NewDocumentCommand{\loccit}{t.}{loc.\ cit.\@\xspace}

    \newcommand{\vd}{virtual double}
    \newcommand{\vdc}{\vd{} category}
    \newcommand{\vdcs}{\vd{} categories}
    
    \newcommand{\ve}{virtual equipment}

\makeatletter

\define@key{beamerframe}{c}[true]{
    \beamer@frametopskip=0pt plus 1fill\relax%
    \beamer@framebottomskip=0pt plus 1fill\relax%
}

\patchcmd{\beamer@sectionintoc}{\vfill}{\vskip\itemsep}{}{}

\makeatother

\ifarticle{}{

  \colorlet{colour-bg}{black!85} 
  \definecolor{colour-primary}{HTML}{cc80ff} 
  \colorlet{colour-text}{black!10} 
  \colorlet{colour-subtle}{black!40} 
  \colorlet{colour-block-bg}{black!80} 
  \definecolor{colour-warning-bg}{HTML}{ffea80} 
  \definecolor{colour-warning-primary}{HTML}{e08152} 

  \hypersetup{linktoc=all,colorlinks, citecolor={colour-primary}, linkcolor={colour-primary}, urlcolor={colour-primary}} 
  \urlstyle{sf}

  \usepackage{changepage} 

  \usepackage{ragged2e} 

  \setbeamertemplate{section in toc}[sections numbered]

  \apptocmd{\frame}{}{\justifying}{}

  \usefonttheme[onlymath]{serif}

  \beamertemplatenavigationsymbolsempty
  \setbeamertemplate{footline}{
    \hfill%
    \usebeamercolor[fg]{page number in head/foot}%
    \usebeamerfont{page number in head/foot}%
    \setbeamertemplate{page number in head/foot}[pagenumber]%
    \usebeamertemplate*{page number in head/foot}\kern.2cm\vskip.2cm%
  }

  \setbeamertemplate{frametitle}[default][center]
  \addtobeamertemplate{frametitle}{\vspace*{1ex}}{\vspace*{1ex}}

  \setbeamertemplate{itemize item}{$\bullet$}

  \setbeamertemplate{theorems}[numbered]
  \newtheorem{proposition}[theorem]{\translate{Proposition}}
  \newtheorem{sketch}[theorem]{\translate{Proof (sketch)}}

  \addtobeamertemplate{block begin}
  {}
  {\vspace{0ex} 
  \begin{adjustwidth}{1em}{1em} 
  \tikzcdset{background color=colour-block-bg}
  }
  \addtobeamertemplate{block end}{\end{adjustwidth}%
  \vspace{1ex}} 
  {}
  \renewenvironment<>{block}[1]{%
      \begin{actionenv}#2%
        \def\insertblocktitle{\begin{adjustwidth}{.8em}{.8em}#1\end{adjustwidth}}%
        \par%
        \usebeamertemplate{block begin}}
      {\par%
        \usebeamertemplate{block end}%
      \end{actionenv}}

  \addtobeamertemplate{block example begin}
  {}
  {\vspace{0ex} 
  \begin{adjustwidth}{1em}{1em} 
  \tikzcdset{background color=colour-block-bg}
  }
  \addtobeamertemplate{block example end}{\end{adjustwidth}%
  \vspace{1ex}} 
  {}
  \renewenvironment<>{exampleblock}[1]{%
      \begin{actionenv}#2%
          \def\insertblocktitle{\begin{adjustwidth}{.8em}{.8em}#1\end{adjustwidth}}%
          \par%
          \only<presentation>{
            \setbeamercolor{local structure}{parent=example text}}%
          \usebeamertemplate{block example begin}}
        {\par%
          \usebeamertemplate{block example end}%
        \end{actionenv}}

    \setbeamercolor{background canvas}{bg=colour-bg}
    \tikzcdset{background color=colour-bg}

    \setbeamercolor{block title}{fg=colour-primary,bg=colour-block-bg}
    \setbeamercolor{block title example}{fg=colour-primary,bg=colour-block-bg}
    \setbeamercolor{block body}{bg=colour-block-bg}
    \setbeamercolor{block body example}{bg=colour-block-bg}

    \setbeamercolor{title}{fg=colour-primary}
    \setbeamercolor{frametitle}{fg=colour-primary}
    \setbeamercolor{alerted text}{fg=colour-primary}

    \setbeamercolor{normal text}{fg=colour-text}
    \setbeamercolor{item}{fg=colour-text}
    \setbeamercolor{section in toc}{fg=colour-text}

    \setbeamercolor{page number in head/foot}{fg=colour-subtle}

    \setbeamercolor*{bibliography entry author}{fg=colour-text}
    \setbeamercolor*{bibliography entry note}{fg=colour-text}

}

\IfFileExists{tangle.sty}{\usepackage{tangle}}{}
\usepackage{bbm}
\usepackage{wasysym}
\usepackage{suffix}

\newcommand{\X}{\bb X}
\newcommand{\tX}{\u\X}
\renewcommand{\Cat}{\dc{Cat}}

\newcommand{\K}{\cl K}

\mathcommand{\L}{\textup{\sagittarius}}

\newcommand{\jAE}{j \colon A \to E}
\newcommand{\jEI}{j' \colon E \to I}
\newcommand{\jadj}{\radj j}
\newcommand{\ljr}{\ell \jadj r}

\newcommand{\ljrp}{\ell' \jadj r'}

\newcommand{\aop}{{\rtimes}}

\newcommand{\wc}{\ostar}
\newcommand{\colim}{\operatorname{colim}}
\newcommand{\wl}{\bulletstar}
\renewcommand{\lim}{\operatorname{lim}}

\renewcommand{\defeq}{\mathrel{:=}}

\newcommand{\odotl}{\mathbin{{\odot}_L}}

\renewcommand{\EM}{\relax\ifmmode\b{EM}\else{}Eilenberg\nobreakdash--Moore\fi}

\mathcommand{\P}{\cl P}

\makeatletter
\@ifpackageloaded{tangle}{
\tgColours{F5A3A3,F5CCA3,F5F5A3,CCF5A3,A3F5A3,A3F5CC,A3F5F5,A3CCF5,A3A3F5,CCA3F5,F5A3F5,F5A3CC}

}{}
\makeatother

\newcounter{eqstep}
\newcounter{eqsubstep}[eqstep]
\newcommand{\nexteqstep}{\stepcounter{eqstep}}
\newcommand{\eqstepref}[1]{(\hyperlink{eqstep:#1}{#1})}

\makeatletter
\newcommand{\raisedtarget}[1]{\Hy@raisedlink{\hypertarget{#1}{}}}
\makeatother

\newcommand{\tangleeq}{\refstepcounter{eqsubstep}\raisedtarget{eqstep:\theeqstep.\theeqsubstep}{\mathclap{\overset{(\theeqstep.\theeqsubstep)}{=}}}}

\WithSuffix\newcommand\tangleeq*{\mathclap{=}}

\newcommand{\tangleeql}{\refstepcounter{eqsubstep}\raisedtarget{eqstep:\theeqstep.\theeqsubstep}{\mathllap{\overset{(\theeqstep.\theeqsubstep)}{=}}}}

\WithSuffix\newcommand\tangleeql*{\mathllap{=}}

\newenvironment{tangleeqs}{%
    \mathcommand{\=}{\tangleeq}
    \global\let\externaldblbackslash\\
    \csname gather*\endcsname
    \ifundef{\internaldblbackslash}{%
        \global\let\internaldblbackslash\\%
        \gdef\\{\internaldblbackslash\tangleeql}%
    }{}
    \nexteqstep
}{%
    \csname endgather*\endcsname
    \global\undef\internaldblbackslash
    \global\let\\\externaldblbackslash
}

\newenvironment{tangleeqs*}{%
    \mathcommand{\=}{\tangleeq*}
    \global\let\externaldblbackslash\\
    \csname gather*\endcsname
    \ifundef{\internaldblbackslash}{%
        \global\let\internaldblbackslash\\%
        \gdef\\{\internaldblbackslash\tangleeql*}%
    }{}
}{%
    \csname endgather*\endcsname
    \global\undef\internaldblbackslash
    \global\let\\\externaldblbackslash
}

\bibliography{references}

\newcommand{\Cart}{\b{Cart}}
\newcommand{\F}{{\bb F}}

\newcommand{\Met}{\b{Met}}
\newcommand{\FinMet}{\b{FinMet}}
\newcommand{\Semigrp}{\b{Semigrp}}
\newcommand{\monprod}{\star}

\theoremstyle{plain}

\newtheorem{manualtheoreminner}{Theorem}
\newenvironment{manualtheorem}[1]{%
  \renewcommand\themanualtheoreminner{#1}%
  \manualtheoreminner
}{\endmanualtheoreminner}

\newtheorem{manualpropositioninner}{Proposition}
\newenvironment{manualproposition}[1]{%
  \renewcommand\themanualpropositioninner{#1}%
  \manualpropositioninner
}{\endmanualpropositioninner}

\newtheorem{manuallemmainner}{Lemma}
\newenvironment{manuallemma}[1]{%
  \renewcommand\themanuallemmainner{#1}%
  \manuallemmainner
}{\endmanuallemmainner}

\newtheorem{manualcorollaryinner}{Corollary}
\newenvironment{manualcorollary}[1]{%
  \renewcommand\themanualcorollaryinner{#1}%
  \manualcorollaryinner
}{\endmanualcorollaryinner}

\theoremstyle{definition}

\newtheorem{manualdefinitioninner}{Definition}
\newenvironment{manualdefinition}[1]{%
  \renewcommand\themanualdefinitioninner{#1}%
  \manualdefinitioninner
}{\endmanualdefinitioninner}

\foreach \env in {manualdefinition}{
\AtBeginEnvironment{\env}{%
    \pushQED{\qed}%
}
\AtEndEnvironment{\env}{\popQED\endexample}
}

\title{Relative monadicity}
\author{Nathanael Arkor}
\address{Department of Mathematics and Statistics, Faculty of Science, Masaryk University, Czech Republic}
\author{Dylan McDermott}
\address{Department of Computer Science, Reykjavik University, Iceland}
\datetoday

\keywords{Relative monad, relative adjunction, monadicity, virtual double category, virtual equipment, formal category theory, enriched category theory}
\subjclass[2020]{18D70,18D65,18C15,18C20,18A40,18C10,18D20,18N10}

\hypersetup{pdfauthor={Nathanael Arkor, Dylan McDermott}}

\begin{document}

\begin{abstract}
    We establish a relative monadicity theorem for relative monads with dense roots in a virtual equipment, specialising to a relative monadicity theorem for enriched relative monads. In particular, for a dense $\V$-functor $\jAE$, a $\V$-functor $r \colon D \to E$ is $j$-monadic if and only if $r$ admits a left $j$-relative adjoint and creates $j$-absolute colimits. This provides a refinement of the classical monadicity theorem -- characterising those categories whose objects are given by those of $E$ equipped with algebraic structure -- in which the arities of the algebraic operations are valued in $A$. In particular, when $j = 1$, we recover a formal monadicity theorem. Furthermore, we examine the interaction between the pasting law for relative adjunctions and relative monadicity. As a consequence, we derive necessary and sufficient conditions for the ($j$-relative) monadicity of the composite of a $\V$-functor with a ($j$-relatively) monadic $\V$-functor.
\end{abstract}

\maketitle

\setcounter{tocdepth}{1}
\tableofcontents

\section{Introduction}

The concept of \emph{monadicity} is fundamental in category theory. A functor $r \colon D \to E$ is said to be monadic if it exhibits $D$ as the category of algebras for a monad on $E$. Conceptually, this permits the identification of the objects of $D$ with objects of $E$ equipped with algebraic structure -- often expressible concretely in terms of operations and equations -- and the morphisms of $D$ as morphisms of $E$ that preserve the algebraic structure. Having such a description of $D$ is useful as it facilitates the application of monad theory, for instance in providing techniques to lift structure from $E$ to $D$.

Given that we are interested in the monadicity of functors, it is important to have methods for establishing monadicity. A \emph{monadicity theorem} provides sufficient (and often also necessary) conditions for a functor to be monadic~\cite{beck1966untitled,pare1971absolute}. Monadicity theorems are particularly convenient tools in two kinds of situations. First, in concrete examples, it can be simpler to verify a set of monadicity conditions for $r$ than to directly construct an equivalence between $D$ and the category of algebras for an induced monad, particularly when the categories in question are complex. Second, in abstract examples, there may not be an amenable description of the category of algebras for the induced monad. This is typically the case, for example, when the existence of a left adjoint to $r$ has been deduced abstractly via an adjoint functor theorem, rather than by giving a concrete description of the left adjoint.

In many situations we are concerned not only with whether the objects of $D$ may be identified with objects of $E$ equipped with algebraic structure, but also with the nature of that algebraic structure. An algebra for a monad may be viewed as an object equipped with operations of arbitrary arity, along with equations imposing laws that the operations must satisfy. In practice, we frequently wish to restrict attention not to arbitrary algebras, but only to those that may be axiomatised in terms of operations with particular arity. For instance, in universal algebra, it is typical to consider specifically those algebras that admit a description in terms of finitary operations~\cite{birkhoff1935structure}. From a concrete perspective, this can be desirable because such algebras admit a tractable syntactic description; from an abstract perspective, categories of such algebras typically have stronger categorical properties. In these situations, the classical notion of monadicity proves too coarse, as monads make no distinction between the possible arities of operations for algebras and the possible carriers for algebras.

A \emph{relative monad} is a refinement of the notion of monad that relaxes the requirement that the monad's underlying functor be an endofunctor~\cite{altenkirch2010monads,altenkirch2015monads}. A monad relative to a functor $\jAE$ may be viewed as a generalisation of the notion of monad in which the possible arities of operations for algebras are valued in $A$, while the possible carriers for algebras are valued in $E$. Just as for ordinary monads, every $j$-relative monad has a category of algebras, equipped with free and forgetful functors $f_T \colon A \to \Alg(T)$ and $u_T \colon \Alg(T) \to E$ that together form a $j$-relative adjunction $f_T \jadj u_T$, \ie a family of bijections $\Alg(T)(f_T x, y) \iso E(j x, u_T y)$ natural in $x$ and $y$~\cite{altenkirch2015monads,ulmer1968properties}.

A functor $r \colon D \to E$ is said to be \emph{$j$-relatively monadic} (or simply \emph{$j$-monadic}) if it exhibits $D$ as the category of algebras for a $j$-relative monad. Conceptually, this permits the identification of the objects of $D$ with objects of $E$ equipped with algebraic structure whose arities are valued, via $j$, in the category $A$. For instance, if $\jAE$ is the inclusion of a full subcategory of finite objects, a $j$-monadic functor $r$ exhibits $D$ as a category of algebras whose operations are finitary. Relative monadicity thus provides the sought-after refinement of monadicity; it is indeed a refinement since, when $j$ is the identity functor on $E$, relative monadicity is simply ordinary monadicity. It is consequently desirable to establish \emph{relative monadicity theorems}, extending the classical monadicity theorems. This is the motivation for the present paper.

In this paper, we establish two relative monadicity theorems. The first (\cref{relative-monadicity-theorem}) is a characterisation in the style of \textcite{beck1966untitled} and \textcite{pare1971absolute}, establishing that, for $r$ to be $j$-monadic, it is necessary and sufficient that $r$ admits a left $j$-relative adjoint and creates certain colimits. The second (\cref{relative-monadicity-of-composite}) is a pasting law for relatively monadic functors, which may be seen as analogous to the pasting law for pullbacks, characterising the relative monadicity of one functor in terms of the relative monadicity of another.

\subsection{Relative monadicity via creation of colimits}
\label{relative-monadicity-in-terms-of-colimits}

For $j$ the identity functor on a category $E$, a $j$-relative monad is a (non-relative) monad, and there is a well-known characterisation of $j$-monadicity: a functor $r \colon D \to E$ is monadic if and only if it admits a left adjoint and creates coequalisers of $r$-contractible pairs~\cite{beck1966untitled}; or, equivalently, if and only if it admits a left adjoint and creates absolute colimits~\cite{pare1971absolute}.

For more general functors $j$, there are also known characterisations of $j$-monadicity. However, perhaps surprisingly, these characterisations predate the modern notion of relative monad by some thirty-five years. \textcite{diers1975jmonades,diers1976foncteur} and \textcite{lee1977relative} independently established that a functor $r \colon D \to E$ is monadic relative to a dense and \ff{} functor $\jAE$ if and only if it admits a left $j$-relative adjoint and creates $j$-absolute colimits\footnotemark{}. Due to the relative obscurity of the sources, and because the notions of relative monad studied by \citeauthor{diers1975jmonades} and \citeauthor{lee1977relative} appear quite different to that of \textcite{altenkirch2010monads} (\cf{}~\cite[Example~8.14]{arkor2024formal}), these relative monadicity theorems have till now been overlooked.
\footnotetext{A colimit in $E$ is \emph{$j$-absolute} when it is preserved by the nerve functor $E(j {-}_2, {-}_1) \colon E \to [A\op, \Set]$.}

While the characterisations of \citeauthor{diers1975jmonades} and \citeauthor{lee1977relative} are useful, there are two respects in which one might hope for greater generality. First, we should like to drop the assumption that the root $j$ be \ff{}. Second, we should like a similar characterisation in the setting of enriched category theory, extending that for enriched monadicity~\cite{bunge1969relative,dubuc1970kan}.

The purpose of \cref{monadicity} is to establish such a characterisation. However, our approach is more general still. It is the philosophy of formal category theory that the various flavours of category theory -- such as ordinary category theory, enriched category theory, internal category theory, and so on -- should all be viewed as instances of a general theory. That is, rather than establish the fundamental theorems of category theory separately in each setting, it is valuable to instead work in a framework in which such theorems may be proven just once, and then specialised to each setting. This was the approach of \cite{arkor2024formal}, in which we laid the foundations for a formal theory of relative monads. Herein, we follow the same approach, working in the context of a virtual equipment in the sense of \textcite{cruttwell2010unified}. Thus, we shall first establish a relative monadicity theorem in the context of a virtual equipment, from which an enriched relative monadicity theorem will follow as a special case (\cf{}~\cite[\S8]{arkor2024formal}). However, we are able to make some simplifications for enrichment in well-behaved monoidal categories, which we describe in \cref{relative-monadicity-in-VCat}.

\subsection{Relative monadicity via composition}

Given a monadic functor $r' \colon D \to E$, it is often useful to establish when precomposing some functor $r \colon C \to D$ produces another monadic functor $(r \d r') \colon C \to E$. It is neither necessary nor sufficient that $r$ be monadic: $(r \d r')$ may admit a left adjoint even when $r$ does not; and $(r \d r')$ may not create the appropriate colimits even when $r$ does. To understand the relationship between $r$, $r'$, and $(r \d r')$, it is enlightening to consider the generalisation to the relative setting. If $r'$ is $j$-monadic for some functor $\jAE$, we may ask when $(r \d r')$ is also $j$-monadic. However, it makes no sense to ask whether $r$, too, is $j$-monadic, because $r$ and $j$ have different codomains. Instead, it is most natural to ask whether $r$ is monadic relative to $\ell' \colon A \to D$, the left $j$-relative adjoint of $r'$.
\[\begin{tikzcd}[sep=small]
	&& C \\
	&&& D \\
	A &&&& E
	\arrow[""{name=0, anchor=center, inner sep=0}, "r", from=1-3, to=2-4]
	\arrow["j"', from=3-1, to=3-5]
	\arrow[""{name=1, anchor=center, inner sep=0}, "{r'}", from=2-4, to=3-5]
	\arrow[""{name=2, anchor=center, inner sep=0}, "{\ell'}"{description}, from=3-1, to=2-4]
	\arrow[""{name=3, anchor=center, inner sep=0}, "\ell", from=3-1, to=1-3]
	\arrow["\dashv"{anchor=center}, shift right=1, draw=none, from=2, to=1]
	\arrow["\dashv"{anchor=center, rotate=18}, shift right=2, draw=none, from=3, to=0]
\end{tikzcd}\]
In fact, as we prove in \cref{composition-and-monadicity}, this is necessary and sufficient to ensure the $j$-monadicity of the composite $(r \d r')$. This observation may be viewed as a pasting law for relatively monadic adjunctions: supposing that the triangle on the right is relatively monadic, the triangle on the left is relatively monadic if and only if the outer triangle is relatively monadic. This characterisation has a couple of advantages over the characterisation in terms of colimit creation: for one, it does not require that the roots $j$ or $\ell'$ be dense, and the existence of a category of algebras can be derived rather than assumed\footnotemark{}. Furthermore, this characterisation appears to be new even for $j = 1$, though we note that a similar observation appears in the work of \textcite[Theorem~1.5.5]{walters1970categorical}, of which our result may be seen as a refinement.
\footnotetext{Recall that the existence of a $\V$-category of algebras for a $\V$-enriched (relative) monad requires the existence of certain limits in $\V$ (\cf{}~\cite[Corollary~8.20]{arkor2024formal}), and so is a nontrivial assumption for general $\V$.}

\subsection{Outline of the paper}

In \cref{prerequisites}, we briefly recall the formal categorical concepts from \cite{arkor2024formal} that will be fundamental to our study of relative monadicity. In \cref{colimits}, we introduce the notions of creation of limits and of colimits in a virtual equipment, and show that the forgetful tight-cell $u_T \colon \Alg(T) \to E$ from the algebra object for a $j$-relative monad creates limits and $j$-absolute colimits (\cref{u_T-creates-j-absolute-colimits,u_T-creates-limits,u_T-nonstrict-creation}). In \cref{monadicity}, we prove a formal relative monadicity theorem (\hyperref[relative-monadicity-theorem]{Theorems~\ref*{relative-monadicity-theorem}} and \hyperref[relative-monadicity-theorem']{\ref*{relative-monadicity-theorem'}}) and, as a consequence, show that the category of algebras for a relative monad $T$ is the category of algebras for a monad if and only if the forgetful tight-cell $u_T \colon \Alg(T) \to E$ admits a left adjoint (\cref{monadic-iff-left-adjoint}). In \cref{relative-monadicity-in-VCat}, we specialise the relative monadicity theorem to $\VCat$ and $\VCat\co$ for a well-behaved monoidal category $\V$, obtaining enriched relative monadicity and relative comonadicity theorems (\hyperref[enriched-relative-monadicity-theorem]{Theorems~\ref*{enriched-relative-monadicity-theorem}} and \hyperref[enriched-relative-comonadicity-theorem]{\ref*{enriched-relative-comonadicity-theorem}}). To demonstrate the application of the relative monadicity theorem, we use it to prove that (1) the category of algebras for a finitary algebraic theory is (relatively) monadic over $\Set$ (\cref{algebraic-theories-induce-monads}); (2) that the category of algebras for a colimit-preserving monad on a free cocompletion is relatively monadic over the cocompletion (\cref{cocontinuous-monads}); and (3) that the category of algebras for a quantitative equational theory in the sense of \textcite{mardare2016quantitative} is (relatively) monadic over $\Met$ (\cref{metric-spaces}). More generally, we show that the category of algebras for a $j$-theory in the sense of \textcite{lucyshyn2023enriched} is $j$-monadic (\cref{j-theory}). Finally, in \cref{composition-and-monadicity}, we consider the interaction between the pasting law for relative adjunctions (\cite[Proposition~5.30]{arkor2024formal}) and relative monadicity, and thereby deduce necessary and sufficient conditions for the composite of a functor with a relatively monadic functor to itself be relatively monadic (\cref{relative-monadicity-of-composite}).

\begin{remark}
    In addition to \citeauthor{diers1975jmonades} and \citeauthor{lee1977relative}, \citeauthor{walters1970categorical} also established a series of relative monadicity theorems for unenriched relative monads with \ff{} roots~\cite[\S2.3]{walters1970categorical}. However, \citeauthor{walters1970categorical} imposes a number of additional assumptions on the functor $r \colon D \to E$ that do not hold in general, and which are not appropriate for our purposes. We shall not pursue an explicit comparison to the theorems of \citeauthor{walters1970categorical} in this paper.
\end{remark}

\begin{remark}
    After presenting an early version of this work to the Masaryk University Algebra Seminar, the authors were informed that John Bourke, Marcelo Fiore, and Richard Garner have, in unpublished joint work, independently established an enriched relative monadicity theorem.
\end{remark}

\subsection{Acknowledgements}

The authors thank John Bourke for helpful comments, and thank the anonymous reviewer for their suggestions, which have improved the exposition of the paper.
The second author was supported by Icelandic Research Fund grant \textnumero{}\,228684-052.

\section{Formal categorical prerequisites}
\label{prerequisites}

We shall work throughout the paper in the context of a \ve{} $\X$, adopting the terminology and notation of \cite{arkor2024formal}.
In this section, we shall briefly recall the fundamental concepts of formal category theory in a \ve{}.
A more detailed introduction may be found in \S2~--~\S3 of \cite{arkor2024formal}.
The reader interested primarily in the applications of the relative monadicity theorem is encouraged to skip to \cref{relative-monadicity-in-VCat}, where we shall instantiate the formal theory in the context of the \ve{} $\VCat$ of categories enriched in a monoidal category $\V$, and which may be read without familiarity with formal category theory.

A \ve{} is, in particular, a \vdc{}, which is a two-dimensional structure having \emph{objects}; \emph{tight-cells} ($\to$); \emph{loose-cells} ($\lto$); and \emph{2-cells} ($\tto$) of the following form (\cf~\cite[Definition~2.1]{arkor2024formal}).
\[\begin{tikzcd}
    {A_0} & {A_1} & \cdots & {A_{n - 1}} & {A_n} \\
    {B_0} &&&& {B_n}
    \arrow["q", "\shortmid"{marking}, from=2-5, to=2-1]
    \arrow["{p_n}"', "\shortmid"{marking}, from=1-5, to=1-4]
    \arrow["{p_1}"', "\shortmid"{marking}, from=1-2, to=1-1]
    \arrow["{p_{n - 1}}"', "\shortmid"{marking}, from=1-4, to=1-3]
    \arrow["{p_2}"', "\shortmid"{marking}, from=1-3, to=1-2]
    \arrow[""{name=0, anchor=center, inner sep=0}, "{f_n}", from=1-5, to=2-5]
    \arrow[""{name=1, anchor=center, inner sep=0}, "{f_0}"', from=1-1, to=2-1]
    \arrow["\phi"{description}, draw=none, from=0, to=1]
\end{tikzcd}\]
For instance, the \vdc{} $\VCat$ has as objects the (possibly large) $\V$-categories; as tight-cells the $\V$-functors; as loose-cells the $\V$-distributors; and as 2-cells the $\V$-natural transformations, which are families of morphisms
\[\phi_{x_0, \ldots, x_n} \colon p_1(x_0, x_1) \otimes \cdots \otimes p_n(x_{n - 1}, x_n) \to q(\ob{f_0} x_0, \ob{f_n} x_n)\]
in $\V$ for each $x_0 \in \ob{A_0}, \dots, x_n \in \ob{A_n}$, satisfying several naturality laws~\cite[Definition~8.1]{arkor2024formal}. For tight-cells $f \colon A \to B$ and $g \colon B \to C$, we denote by $(f \d g) \colon A \to C$ or $g f \colon A \to C$ their composite. For loose-cells $p \colon X \lto Y$ and $q \colon Y \lto Z$, we denote by $q \odot p \colon X \lto Z$ and $q \odotl p \colon X \lto Z$ their composite and left-composite (\cite[Definition~2.5]{arkor2024formal}) respectively, when they exist.

A \ve{} (or simply \emph{equipment}) is a \vdc{} for which every object $A$ admits a loose-identity $A(1, 1) \colon A \lto A$ (\cite[Definition~2.4]{arkor2024formal}); and for which, for every loose-cell $p \colon X \lto Y$ and tight-cells $g \colon W \to X$ and $f \colon Z \to Y$, there is a restriction loose-cell $p(f, g) \colon W \lto Z$ (\cite[Definition~2.7]{arkor2024formal}). We denote the restriction $A(1, 1)(f, g)$ along a loose-identity $A(1, 1)$ by $A(f, g)$: in $\VCat$, this restriction is given by the hom-objects of $A$. Virtual equipments admit a string diagram calculus that we shall employ in some proofs, but are not required to understand the definitions or theorem statements.

One notable difference between the formal categorical setting and typical approaches to enriched category theory (\cf{}~\cite{dubuc1970kan,kelly1982basic}) is that our limits and colimits are weighted by $\V$-distributors rather than $\V$-presheaves and $\V$-copresheaves: for a $\V$-distributor $p \colon X \lto Y$ and $\V$-functor $f \colon Y \to Z$, the $p$-weighted colimit of $f$, when it exists, is a $\V$-functor $p \wc f \colon X \to Z$ (\cite[Definition~3.8]{arkor2024formal}); and for a $\V$-distributor $p \colon X \lto Y$ and $\V$-functor $g \colon X \to Z$, the $p$-weighted limit of $g$, when it exists, is a $\V$-functor $p \wl g \colon Y \to Z$ (\cite[Definition~3.9]{arkor2024formal}). In particular, given tight-cells $\jAE$ and $f \colon A \to Z$, the (pointwise) left extension  $j \plx f \colon E \to Z$ is the $E(j, 1)$-weighted colimit of $f$.

\subsection{Relative adjunctions, relative monads, and algebra objects}

We recall the concepts introduced in \cite{arkor2024formal} that are most relevant to our study of relative monadicity. The central concept is that of a relative monad, which is a generalisation of a monad that permits the carrier to be an arbitrary tight-cell, rather than an endo-tight-cell. Accordingly, a relative monad is parameterised by a tight-cell $\jAE$ that plays the same role that the identity tight-cell plays for a non-relative monad.

\begin{definition}[{\cite[Definition~4.1]{arkor2024formal}}]
    \label{relative-monad}
    A \emph{relative monad} comprises
    \begin{enumerate}
        \item a tight-cell $\jAE$, the \emph{root};
        \item a tight-cell $t \colon A \to E$, the \emph{carrier};
        \item a 2-cell $\dag \colon E(j, t) \tto E(t, t)$, the \emph{extension operator};
        \item a 2-cell $\eta \colon j \tto t$, the \emph{unit},
    \end{enumerate}
    satisfying associativity and left- and right-unit laws.
    A \emph{$j$-relative monad} (or simply \emph{$j$-monad}) is a relative monad with root $j$.
\end{definition}

Just as monads are typically induced by adjunctions, relative monads are typically induced by relative adjunctions. A relative adjunction is a generalisation of an adjunction that permits the left adjoint to have domain distinct from the codomain of the right adjoint.

\begin{definition}[{\cite[Definition~5.1]{arkor2024formal}}]
    \label{relative-adjunction}
    A \emph{relative adjunction} comprises
    \[\begin{tikzcd}
    	& C \\
    	A && E
    	\arrow[""{name=0, anchor=center, inner sep=0}, "\ell"{pos=0.4}, from=2-1, to=1-2]
    	\arrow[""{name=1, anchor=center, inner sep=0}, "r"{pos=0.6}, from=1-2, to=2-3]
    	\arrow["j"', from=2-1, to=2-3]
    	\arrow["\dashv"{anchor=center}, shift right=2, draw=none, from=0, to=1]
    \end{tikzcd}\]
    \begin{enumerate}
        \item a tight-cell $\jAE$, the \emph{root};
        \item a tight-cell $\ell \colon A \to C$, the \emph{left (relative) adjoint};
        \item a tight-cell $r \colon C \to E$, the \emph{right (relative) adjoint};
        \item an isomorphism $\sharp \colon C(\ell, 1) \iso E(j, r) \cocolon \flat$, the \emph{(left- and right-) transposition operators}.
    \end{enumerate}
    We denote by $\ljr$ such data, and call $C$ the \emph{apex}. A \emph{$j$-relative adjunction} (or simply \emph{$j$-adjunction}) is a relative adjunction with root $j$.
\end{definition}

Every $j$-adjunction $\ljr$ induces a $j$-monad with carrier $(\ell \d r) \colon A \to E$~\cite[Theorem~5.24]{arkor2024formal}. A $j$-adjunction that induces a $j$-monad $T$ is called a \emph{resolution} of $T$~\cite[Definition~5.25]{arkor2024formal}.

Relative monads admit notions of algebra and algebra morphism analogous to those for non-relative monads. However, it is important to observe that, in the context of a \ve{}, the appropriate notion of morphism of algebras is graded by a chain of loose-cells. This permits the consideration of morphisms between algebras with different domains, in contrast to morphisms of algebras in a 2-category~\cite{street1972formal}.

\begin{definition}[{\cite[Definitions~6.1 \& 6.29]{arkor2024formal}}]
	\label{algebra}
    Let $T$ be a relative monad. An \emph{algebra for $T$} (or simply \emph{$T$-algebra}) comprises
    \begin{enumerate}
		\item an object $D$, the \emph{domain};
        \item a tight-cell $e \colon D \to E$, the \emph{carrier};
        \item a 2-cell $\aop \colon E(j, e) \tto E(t, e)$, the \emph{extension operator},
    \end{enumerate}
    satisfying compatibility laws with the extension operator $\dag$ and unit $\eta$ for $T$.

    Let $(e \colon D \to E, \aop)$ and $(e' \colon D' \to E, \aop')$ be $T$-algebras. A \emph{$(p_1, \ldots, p_n)$-graded $T$-algebra morphism} from $(e, \aop)$ to $(e', \aop')$ is a 2-cell \[\epsilon \colon E(1, e), p_1, \ldots, p_n \tto E(1, e')\] satisfying a compatibility law with the two extension operators $\aop$ and $\aop'$.
\end{definition}

\begin{remark}[{\cite[Remark~6.30]{arkor2024formal}}]
    \label{alternative-graded-morphism}
    A $(p_1, \ldots, p_n)$-graded $T$-algebra morphism from $(e, \aop)$ to $(e', \aop')$ is equivalently a 2-cell \[\epsilon \colon p_1, \ldots, p_n \tto E(e, e')\]
    satisfying a compatibility law with the two extension operators $\aop$ and $\aop'$. We shall use the term \emph{algebra morphism} to refer to both forms of 2-cell interchangeably.
\end{remark}

An algebra object for a relative monad $T$ is a universal $T$-algebra. In $\VCat$, this corresponds to the $\V$-category of $T$-algebras~\cite[\S8.3]{arkor2024formal}. Algebra objects are the central object of interest for the relative monadicity theorem.

\begin{definition}[{\cite[Definition~6.33]{arkor2024formal}}]
	\label{algebra-object}
	Let $T$ be a relative monad. A $T$-algebra $(u_T \colon \Alg(T) \to E, \aop_T)$ is an
	\emph{algebra object} for $T$ when
	\begin{enumerate}
		\item \label{algebra-object-tight-cell-UP} for every $T$-algebra $(e \colon D \to E, \aop)$, there is a
		unique tight-cell $\unit_{(e, \aop)} \colon D \to \Alg(T)$ such that $\unit_{(e, \aop)} \d u_T = e$ and $\aop_T(1, \unit_{(e, \aop)}) = \aop$;
		\item \label{algebra-object-2-cell-UP} for every graded $T$-algebra morphism $\epsilon \colon E(1, e), p_1, \dots, p_n \tto E(1, e')$ there is a unique 2-cell $\unit_\epsilon \colon \Alg(T)(1, \unit_{(e, \aop)}), p_1, \dots, p_n \tto \Alg(T)(1, \unit_{(e', \aop')})$ such that:
		\[
		\epsilon~=~
		\begin{tikzcd}
			E & D & \cdots & {D'} \\
			E &&& {D'}
			\arrow[""{name=0, anchor=center, inner sep=0}, Rightarrow, no head, from=1-4, to=2-4]
			\arrow["{p_n}"', "\shortmid"{marking}, from=1-4, to=1-3]
			\arrow["{p_1}"', "\shortmid"{marking}, from=1-3, to=1-2]
			\arrow["{E(1, e)}"', "\shortmid"{marking}, from=1-2, to=1-1]
			\arrow["{E(1, e')}", "\shortmid"{marking}, from=2-4, to=2-1]
			\arrow[""{name=1, anchor=center, inner sep=0}, Rightarrow, no head, from=1-1, to=2-1]
			\arrow["{E(1, u_T), \unit_\epsilon}"{description}, draw=none, from=0, to=1]
		\end{tikzcd}
		\]
        \qedhere
	\end{enumerate}
\end{definition}

The algebra object for a relative monad $T$ forms a resolution of $T$~\cite[Lemma~6.35]{arkor2024formal}.
\[\begin{tikzcd}
	& {\Alg(T)} \\
	A && E
	\arrow[""{name=0, anchor=center, inner sep=0}, "{f_T}", from=2-1, to=1-2]
	\arrow[""{name=1, anchor=center, inner sep=0}, "{u_T}", from=1-2, to=2-3]
	\arrow["j"', from=2-1, to=2-3]
	\arrow["\dashv"{anchor=center}, shift right=2, draw=none, from=0, to=1]
\end{tikzcd}\]
Furthermore, this resolution is universal in that sense that, for any resolution $\ljr$ of $T$ with apex $C$, there is a unique morphism of resolutions, the \emph{comparison tight-cell}, from $\ljr$ to $f_T \jadj u_T$, \ie a tight-cell $\unit_{\ljr} \colon C \to \Alg(T)$ rendering the following diagram commutative~\cite[Corollary~6.41]{arkor2024formal}.
\[\begin{tikzcd}[row sep=small]
	& C \\
	A && E \\
	& {\Alg(T)}
	\arrow["r", from=1-2, to=2-3]
	\arrow["{\unit_{\ljr}}"{description}, from=1-2, to=3-2]
	\arrow["\ell", from=2-1, to=1-2]
	\arrow["{f_T}"', from=2-1, to=3-2]
	\arrow["{u_T}"', from=3-2, to=2-3]
\end{tikzcd}\]

\section{Creation of limits, colimits, and isomorphisms}
\label{colimits}

We begin with some elementary observations on the creation of limits and certain colimits by forgetful functors from categories of algebras for relative monads.

\subsection{Strict creation}

We first introduce the notion of strict creation of weighted limits and colimits, which will be important in the strict relative monadicity theorem (\cref{strict}). We then introduce creation of isomorphisms, and non-strict creation of limits and colimits in \cref{nonstrict-creation}, which will be important in the non-strict relative monadicity theorem (\cref{non-strict}). For the notions of weighted limit and weighted colimit in an equipment, see \cite[\S3.2]{arkor2024formal}.

\begin{remark}
    Since there is substantial inconsistency in the literature regarding which of strict or non-strict creation (and consequently monadicity) should be taken as fundamental, we shall be explicit about strictness throughout; the unqualified terms are used only informally.
\end{remark}

\begin{definition}
    \label{colimit-creation}
    Let $p \colon Y \lto Z$ be a loose-cell, and let $f \colon Z \to W$ and $g \colon W \to X$ be tight-cells.
    A colimit $(p \wc (f \d g), \lambda)$ in $X$ is \emph{strictly created by $g$} when there exists a $p$-cocone $(w, \lambda')$ for $f$, comprising a tight-cell $w \colon Y \to W$ and a 2-cell $\lambda' \colon p \tto W(f, w)$, such that
    \begin{enumerate}
        \item $(w, \lambda')$ is the unique pair satisfying $p \wc (f \d g) = (w \d g)$ and $\lambda = (\lambda' \d g)$;
        \item $(w, \lambda')$ is the weighted colimit $p \wc f$. \qedhere
    \end{enumerate}
\end{definition}

In particular, strict creation of $p \wc (f \d g)$ by $g$ implies preservation of $p \wc f$ by $g$ (in the sense of \cite[Definition~3.8]{arkor2024formal}). Our motivating example of colimit creation will be the forgetful tight-cell $u_T \colon \Alg(T) \to E$ from an algebra object for a $j$-monad, which strictly creates every colimit that is $j$-absolute, in the sense recalled below (see \cref{j-absolute-is-preservation-by-n_j} for a characterisation of $j$-absoluteness in the enriched context).

\begin{definition}[{\cite[Definition~3.21]{arkor2024formal}}]
	\label{j-absolute}
    Let $p \colon Y \lto Z$ be a loose-cell and let $\jAE$ and $f \colon Z \to E$ be
    tight-cells.
    A colimit $(p \wc f, \lambda)$ is \emph{$j$-absolute} if the left-composite
    $E(j, f) \odotl p$ exists, and the canonical 2-cell $E(j, f) \odotl p \tto E(j, p \wc f)$ is invertible.
\end{definition}

\begin{proposition}
    \label{u_T-creates-j-absolute-colimits}
    Let $\jAE$ be a tight-cell and let $T$ be $j$-monad admitting an algebra object $(u_T, \aop_T)$.
    The tight-cell $u_T \colon \Alg(T) \to E$ strictly creates $j$-absolute colimits.
\end{proposition}

\begin{proof}
    Let $p \colon Y \lto Z$ be a loose-cell and let $f \colon Z \to \Alg(T)$ be a tight-cell. Suppose that $(p \wc (f \d u_T), \lambda)$ is a $j$-absolute colimit. The tight-cell $(f \d u_T)$ forms a $T$-algebra by equipping it with the 2-cell $\aop_T(1, f) \colon E(j, (f; u_T)) \tto E(t, (f; u_T))$.

    We will show first that there is a 2-cell $\aop \colon E(j, p \wc (f \d u_T)) \tto E(t, p \wc (f \d u_T))$ equipping $p \wc (f \d u_T)$ with the structure of a $T$-algebra, and then that this is the unique algebra structure for which $\lambda \colon p \tto E(u_T f, p \wc (f \d u_T))$ is a $(p)$-graded $T$-algebra morphism from $(f \d u_T)$ to $p \wc (f \d u_T)$ in the sense of \cref{alternative-graded-morphism}.

    By definition, $j$-absoluteness of $(p \wc (f \d u_T), \lambda)$ means that the 2-cell
    
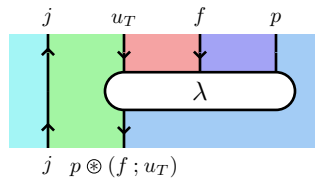
\[
    \begin{tangle}{(4,3)}[trim y]
        \tgBorderA{(0,0)}{\tgColour6}{\tgColour4}{\tgColour4}{\tgColour6}
        \tgBorderA{(1,0)}{\tgColour4}{\tgColour0}{\tgColour0}{\tgColour4}
        \tgBorderA{(2,0)}{\tgColour0}{\tgColour8}{\tgColour8}{\tgColour0}
        \tgBorderA{(3,0)}{\tgColour8}{\tgColour7}{\tgColour7}{\tgColour8}
        \tgBorderA{(0,1)}{\tgColour6}{\tgColour4}{\tgColour4}{\tgColour6}
        \tgBorderA{(1,1)}{\tgColour4}{\tgColour0}{\tgColour7}{\tgColour4}
        \tgBorderA{(2,1)}{\tgColour0}{\tgColour8}{\tgColour7}{\tgColour7}
        \tgBorderA{(3,1)}{\tgColour8}{\tgColour7}{\tgColour7}{\tgColour7}
        \tgBorderA{(0,2)}{\tgColour6}{\tgColour4}{\tgColour4}{\tgColour6}
        \tgBorderA{(1,2)}{\tgColour4}{\tgColour7}{\tgColour7}{\tgColour4}
        \tgBlank{(2,2)}{\tgColour7}
        \tgBlank{(3,2)}{\tgColour7}
        \tgArrow{(1,0.5)}{3}
        \tgCell[(2,0)]{(2,1)}{\lambda}
        \tgArrow{(2,0.5)}{3}
        \tgArrow{(1,1.5)}{3}
        \tgArrow{(0,1.5)}{1}
        \tgArrow{(0,0.5)}{1}
        \tgAxisLabel{(0.5,0.75)}{south}{j}
        \tgAxisLabel{(1.5,0.75)}{south}{u_T}
        \tgAxisLabel{(2.5,0.75)}{south}{f}
        \tgAxisLabel{(3.5,0.75)}{south}{p}
        \tgAxisLabel{(0.5,2.25)}{north}{j}
        \tgAxisLabel{(1.5,2.25)}{north}{p \wc (f \d u_T)}
    \end{tangle}
    \]
    is left-opcartesian, so the 2-cell on the left of the equation below factors uniquely therethrough.
    
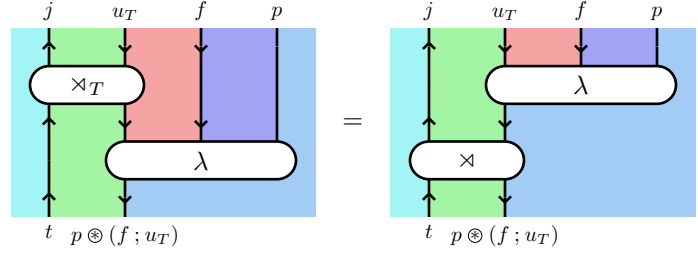
\begin{tangleeqs*}
    \begin{tangle}{(4,4)}[trim y]
        \tgBorderA{(0,0)}{\tgColour6}{\tgColour4}{\tgColour4}{\tgColour6}
        \tgBorderA{(1,0)}{\tgColour4}{\tgColour0}{\tgColour0}{\tgColour4}
        \tgBorderA{(2,0)}{\tgColour0}{\tgColour8}{\tgColour8}{\tgColour0}
        \tgBorderA{(3,0)}{\tgColour8}{\tgColour7}{\tgColour7}{\tgColour8}
        \tgBorderA{(0,1)}{\tgColour6}{\tgColour4}{\tgColour4}{\tgColour6}
        \tgBorder{(0,1)}{0}{1}{0}{0}
        \tgBorderA{(1,1)}{\tgColour4}{\tgColour0}{\tgColour0}{\tgColour4}
        \tgBorder{(1,1)}{0}{0}{0}{1}
        \tgBorderA{(2,1)}{\tgColour0}{\tgColour8}{\tgColour8}{\tgColour0}
        \tgBorderA{(3,1)}{\tgColour8}{\tgColour7}{\tgColour7}{\tgColour8}
        \tgBorderA{(0,2)}{\tgColour6}{\tgColour4}{\tgColour4}{\tgColour6}
        \tgBorderA{(1,2)}{\tgColour4}{\tgColour0}{\tgColour7}{\tgColour4}
        \tgBorderA{(2,2)}{\tgColour0}{\tgColour8}{\tgColour7}{\tgColour7}
        \tgBorderA{(3,2)}{\tgColour8}{\tgColour7}{\tgColour7}{\tgColour7}
        \tgBorderA{(0,3)}{\tgColour6}{\tgColour4}{\tgColour4}{\tgColour6}
        \tgBorderA{(1,3)}{\tgColour4}{\tgColour7}{\tgColour7}{\tgColour4}
        \tgBlank{(2,3)}{\tgColour7}
        \tgBlank{(3,3)}{\tgColour7}
        \tgArrow{(1,1.5)}{3}
        \tgArrow{(1,2.5)}{3}
        \tgCell[(1,0)]{(0.5,1)}{\aop_T}
        \tgArrow{(1,0.5)}{3}
        \tgArrow{(0,1.5)}{1}
        \tgArrow{(0,2.5)}{1}
        \tgArrow{(0,0.5)}{1}
        \tgArrow{(2,0.5)}{3}
        \tgArrow{(2,1.5)}{3}
        \tgCell[(2,0)]{(2,2)}{\lambda}
        \tgAxisLabel{(0.5,0.75)}{south}{j}
        \tgAxisLabel{(1.5,0.75)}{south}{u_T}
        \tgAxisLabel{(2.5,0.75)}{south}{f}
        \tgAxisLabel{(3.5,0.75)}{south}{p}
        \tgAxisLabel{(0.5,3.25)}{north}{t}
        \tgAxisLabel{(1.5,3.25)}{north}{p \wc (f \d u_T)}
    \end{tangle}
    \=
    \begin{tangle}{(4,4)}[trim y]
        \tgBorderA{(0,0)}{\tgColour6}{\tgColour4}{\tgColour4}{\tgColour6}
        \tgBorderA{(1,0)}{\tgColour4}{\tgColour0}{\tgColour0}{\tgColour4}
        \tgBorderA{(2,0)}{\tgColour0}{\tgColour8}{\tgColour8}{\tgColour0}
        \tgBorderA{(3,0)}{\tgColour8}{\tgColour7}{\tgColour7}{\tgColour8}
        \tgBorderA{(0,1)}{\tgColour6}{\tgColour4}{\tgColour4}{\tgColour6}
        \tgBorderA{(1,1)}{\tgColour4}{\tgColour0}{\tgColour7}{\tgColour4}
        \tgBorderA{(2,1)}{\tgColour0}{\tgColour8}{\tgColour7}{\tgColour7}
        \tgBorderA{(3,1)}{\tgColour8}{\tgColour7}{\tgColour7}{\tgColour7}
        \tgBorderA{(0,2)}{\tgColour6}{\tgColour4}{\tgColour4}{\tgColour6}
        \tgBorder{(0,2)}{0}{1}{0}{0}
        \tgBorderA{(1,2)}{\tgColour4}{\tgColour7}{\tgColour7}{\tgColour4}
        \tgBorder{(1,2)}{0}{0}{0}{1}
        \tgBlank{(2,2)}{\tgColour7}
        \tgBlank{(3,2)}{\tgColour7}
        \tgBorderA{(0,3)}{\tgColour6}{\tgColour4}{\tgColour4}{\tgColour6}
        \tgBorderA{(1,3)}{\tgColour4}{\tgColour7}{\tgColour7}{\tgColour4}
        \tgBlank{(2,3)}{\tgColour7}
        \tgBlank{(3,3)}{\tgColour7}
        \tgArrow{(1,0.5)}{3}
        \tgArrow{(1,1.5)}{3}
        \tgCell[(1,0)]{(0.5,2)}{\aop}
        \tgArrow{(0,0.5)}{1}
        \tgArrow{(0,1.5)}{1}
        \tgArrow{(0,2.5)}{1}
        \tgArrow{(1,2.5)}{3}
        \tgCell[(2,0)]{(2,1)}{\lambda}
        \tgArrow{(2,0.5)}{3}
        \tgAxisLabel{(0.5,0.75)}{south}{j}
        \tgAxisLabel{(1.5,0.75)}{south}{u_T}
        \tgAxisLabel{(2.5,0.75)}{south}{f}
        \tgAxisLabel{(3.5,0.75)}{south}{p}
        \tgAxisLabel{(0.5,3.25)}{north}{t}
        \tgAxisLabel{(1.5,3.25)}{north}{p \wc (f \d u_T)}
    \end{tangle}
    \end{tangleeqs*}
    The induced 2-cell $\aop$ makes $p \wc (f \d u_T)$ into a $T$-algebra, the unit and extension operator laws following from those of $\aop_T$ by pasting the left-opcartesian 2-cell.
    Moreover, the above equation is exactly the equation required for $\lambda$ to form a $(p)$-graded $T$-algebra morphism. The 2-cell $\aop$ is therefore the unique algebra structure for which $\lambda$ is such a morphism.

    The universal property of the algebra object for $T$ thus induces a unique tight-cell ${(p \wc f) \colon Y \to \Alg(T)}$ and 2-cell $\lambda' \colon p \tto \Alg(T)(f, p \wc f)$ satisfying the following three equations.
    \[
      p \wc (f \d u_T) = (p \wc f) \d u_T
      \hspace{4em}
      \aop = \aop_T(1, p \wc f)
      \hspace{4em}
      \lambda = \lambda' \d u_T
    \]
    Hence the pair $(p \wc f, \lambda')$ satisfies the existence condition of (1) in \cref{colimit-creation}. Uniqueness of this pair is immediate from uniqueness of $\aop$.

    It remains to show that $(p \wc f, \lambda')$ is the colimit.
    This requires us to show that we have a bijection between 2-cells $q_1, \dots, q_n \tto \Alg(T)(p \wc f, 1)$ and 2-cells $p, q_1, \dots, q_n \tto \Alg(T)(f, 1)$, given by pasting with the following 2-cell.
    
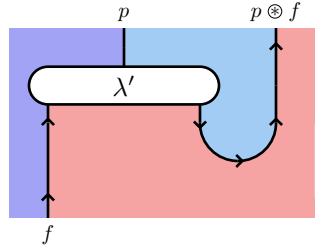
\[
    \begin{tangle}{(4,4)}[trim y]
        \tgBlank{(0,0)}{\tgColour8}
        \tgBorderA{(1,0)}{\tgColour8}{\tgColour7}{\tgColour7}{\tgColour8}
        \tgBlank{(2,0)}{\tgColour7}
        \tgBorderA{(3,0)}{\tgColour7}{\tgColour0}{\tgColour0}{\tgColour7}
        \tgBorderA{(0,1)}{\tgColour8}{\tgColour8}{\tgColour0}{\tgColour8}
        \tgBorderA{(1,1)}{\tgColour8}{\tgColour7}{\tgColour0}{\tgColour0}
        \tgBorderA{(2,1)}{\tgColour7}{\tgColour7}{\tgColour7}{\tgColour0}
        \tgBorderA{(3,1)}{\tgColour7}{\tgColour0}{\tgColour0}{\tgColour7}
        \tgBorderA{(0,2)}{\tgColour8}{\tgColour0}{\tgColour0}{\tgColour8}
        \tgBlank{(1,2)}{\tgColour0}
        \tgBorderC{(2,2)}{0}{\tgColour0}{\tgColour7}
        \tgBorderC{(3,2)}{1}{\tgColour0}{\tgColour7}
        \tgBorderA{(0,3)}{\tgColour8}{\tgColour0}{\tgColour0}{\tgColour8}
        \tgBlank{(1,3)}{\tgColour0}
        \tgBlank{(2,3)}{\tgColour0}
        \tgBlank{(3,3)}{\tgColour0}
        \tgCell[(2,0)]{(1,1)}{\lambda'}
        \tgArrow{(2,1.5)}{3}
        \tgArrow{(2.5,2)}{0}
        \tgArrow{(3,1.5)}{1}
        \tgArrow{(0,1.5)}{1}
        \tgArrow{(3,0.5)}{1}
        \tgArrow{(0,2.5)}{1}
        \tgAxisLabel{(1.5,0.75)}{south}{p}
        \tgAxisLabel{(3.5,0.75)}{south}{p \wc f}
        \tgAxisLabel{(0.5,3.25)}{north}{f}
    \end{tangle}
    \]
    By the universal property of the algebra object, this is equivalent to having a bijection between $T$-algebra morphisms $q_1, \dots, q_n \tto E(p \wc (f \d u_T), u_T)$ and $T$-algebra morphisms ${p, q_1, \dots, q_n \tto E(f, u_T)}$, given by pasting with the following 2-cell.
    
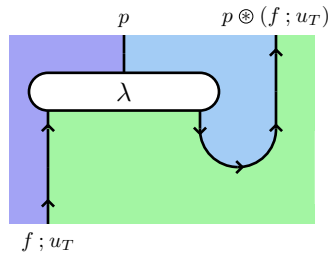
\[
    \begin{tangle}{(4,4)}[trim y]
        \tgBlank{(0,0)}{\tgColour8}
        \tgBorderA{(1,0)}{\tgColour8}{\tgColour7}{\tgColour7}{\tgColour8}
        \tgBlank{(2,0)}{\tgColour7}
        \tgBorderA{(3,0)}{\tgColour7}{\tgColour4}{\tgColour4}{\tgColour7}
        \tgBorderA{(0,1)}{\tgColour8}{\tgColour8}{\tgColour4}{\tgColour8}
        \tgBorderA{(1,1)}{\tgColour8}{\tgColour7}{\tgColour4}{\tgColour4}
        \tgBorderA{(2,1)}{\tgColour7}{\tgColour7}{\tgColour7}{\tgColour4}
        \tgBorderA{(3,1)}{\tgColour7}{\tgColour4}{\tgColour4}{\tgColour7}
        \tgBorderA{(0,2)}{\tgColour8}{\tgColour4}{\tgColour4}{\tgColour8}
        \tgBlank{(1,2)}{\tgColour4}
        \tgBorderC{(2,2)}{0}{\tgColour4}{\tgColour7}
        \tgBorderC{(3,2)}{1}{\tgColour4}{\tgColour7}
        \tgBorderA{(0,3)}{\tgColour8}{\tgColour4}{\tgColour4}{\tgColour8}
        \tgBlank{(1,3)}{\tgColour4}
        \tgBlank{(2,3)}{\tgColour4}
        \tgBlank{(3,3)}{\tgColour4}
        \tgCell[(2,0)]{(1,1)}{\lambda}
        \tgArrow{(2,1.5)}{3}
        \tgArrow{(2.5,2)}{0}
        \tgArrow{(3,1.5)}{1}
        \tgArrow{(0,1.5)}{1}
        \tgArrow{(3,0.5)}{1}
        \tgArrow{(0,2.5)}{1}
        \tgAxisLabel{(1.5,0.75)}{south}{p}
        \tgAxisLabel{(3.5,0.75)}{south}{p \wc (f \d u_T)}
        \tgAxisLabel{(0.5,3.25)}{north}{f \d u_T}
    \end{tangle}
    \]
    Since $(p \wc (f \d u_T), \lambda)$ is a colimit, pasting induces a bijection between 2-cells.
    That this bijection preserves $T$-algebra morphisms is immediate from the defining equation for $\aop$ above.
\end{proof}

Though we shall not make use of it in our proof of relative monadicity, it is nonetheless useful to observe that forgetful tight-cells strictly create all limits.
Strict creation of limits is dual to strict creation of colimits (since a weighted limit $p \wl f$ in an equipment $\X$ is precisely a weighted colimit $p \wc f$ in the dual equipment $\X\co$).
We spell out the definition explicitly for convenience.
\begin{definition}
    \label{limit-creation}
    Let $p \colon Y \lto Z$ be a loose-cell, and let $f \colon Y \to W$ and $g \colon W \to X$ be tight-cells.
    A limit $(p \wl (f \d g), \mu)$ in $X$ is \emph{strictly created by $g$} when there exists a $p$-cone $(w, \mu')$ for $f$, comprising a tight-cell $w \colon Z \to W$ and a 2-cell $\mu' \colon p \tto W(w, f)$, such that
    \begin{enumerate}
        \item $(w, \mu')$ is the unique pair satisfying $p \wl (f \d g) = (w \d g)$ and $\mu = (\mu' \d g)$;
        \item $(w, \mu')$ is the weighted limit $p \wl f$.
        \qedhere
    \end{enumerate}
\end{definition}

In particular, strict creation of limits implies preservation (in the sense of \cite[Definition~3.9]{arkor2024formal}).

\begin{proposition}
    \label{u_T-creates-limits}
    Let $\jAE$ be a tight-cell and let $T$ be $j$-monad admitting an algebra object $(u_T, \aop_T)$.
    The tight-cell $u_T \colon \Alg(T) \to E$ strictly creates limits.
\end{proposition}

\begin{proof}
    Let $p \colon Y \lto Z$ be a loose-cell and let $f \colon Y \to \Alg(T)$ be a tight-cell. Suppose that $(p \wl (f \d u_T), \mu)$ is a limit in $E$.
    The tight-cell $p \wl (f \d u_T)$ forms a $T$-algebra by equipping it with the following 2-cell,
    \[
      \aop \colon E(j, p \wl (f \d u_T))
      \iso p \rx E(j, (f \d u_T))
      \xtto{p \rx \aop_T(1, f)} p \rx E(t, (f \d u_T))
      \iso E(t, p \wl (f \d u_T))
    \]
    where the isomorphisms are induced by the universal property of the limit.
    That $\aop$ satisfies the $T$-algebra laws follows from the fact that $\aop_T$ does.
    This extension operator $\aop$ is the unique algebra structure for which the universal 2-cell $\mu \colon p \tto E(p \wl (f \d u_T), (f \d u_T))$ is a $(p)$-graded algebra morphism from $p \wl (f \d u_T)$ to $(f \d u_T)$. Thus the $p$-cone $(p \wl (f \d u_T), \mu)$ lifts uniquely through $u_T$ as a $p$-cone $(p \wl f, \mu')$.
    The proof that $(p \wl f, \mu')$ is a limit in $\Alg(T)$ is analogous to the proof that the unique lifting in \cref{u_T-creates-j-absolute-colimits} is a colimit.
\end{proof}

\begin{remark}
    From \cref{u_T-creates-limits}, we recover \cites[Proposition~1.2.1]{manes1967triple}[Corollary~5]{frei1969some} and \cite[Theorem~3.2]{wiesler1970remarks} for non-relative monads in $\Cat$ and $\VCat$, and \cite[Theorem~2.1.2]{walters1970categorical} for relative monads in $\Cat$ with \ff{} roots (\cf{}~\cite[Example~8.14(1)]{arkor2024formal}).
\end{remark}

\begin{remark}
    Observe that, since it is not possible to formulate the notions of weighted limit and weighted colimit in an arbitrary 2-category, it is not possible to recover analogues of
    \cref{u_T-creates-j-absolute-colimits,u_T-creates-limits} for algebra objects in a 2-category in the sense of \textcite[\S1]{street1972formal} that specialise appropriately to enriched categories of algebras. In fact, \cref{u_T-creates-j-absolute-colimits,u_T-creates-limits} appear to be new even for non-relative monads in the setting of $\VCat$.\footnotemark{}
    \footnotetext{\textcite[Theorem~II.2.1 \& Proposition~II.4.7 \& II.4.8]{dubuc1970kan} establishes the creation of coequalisers of $u_T$-contractible pairs, powers, conical limits, and ends, but not general weighted limits (the notion of which postdates \citeauthor{dubuc1970kan}'s work~\cite{borceux1973notion,auderset1974adjonctions,borceux1975notion}).}
\end{remark}

\subsection{Non-strict creation}
\label{nonstrict-creation}

Next, we consider the non-strict creation of limits and colimits, which will be important in the non-strict relative monadicity theorem (\cref{non-strict}). It will be useful to observe that non-strict creation of (co)limits is implied by strict creation of (co)limits together with the creation (or, equivalently, the reflection) of isomorphisms.

\begin{definition}
    A tight-cell $g \colon W \to X$ is \emph{conservative} if, for each parallel pair of tight-cells $f_1, f_2 \colon Z \rightrightarrows W$ and each 2-cell $\phi \colon f_1 \tto f_2$, invertibility of $(\phi \d g)$ implies invertibility of $\phi$.
\end{definition}

\begin{lemma}
    \label{u_T-is-conservative}
    Let $\jAE$ be a tight-cell and let $T$ be $j$-monad admitting an algebra object $(u_T, \aop_T)$.
    The tight-cell $u_T \colon \Alg(T) \to E$ is conservative.
\end{lemma}

\begin{proof}
    Let $f_1, f_2 \colon Z \rightrightarrows \Alg(T)$ be a parallel pair of tight-cells and let $\phi \colon f_1 \tto f_2$ be a 2-cell. Suppose that $(\phi \d u_T)$ admits an inverse. Since $(\phi \d u_T)$ is an algebra morphism, so is its inverse, so that the inverse factors uniquely through $u_T$ by the universal property of $\Alg(T)$.
\end{proof}

We shall mark with a prime ($'$) the non-strict variants of definitions and results corresponding to strict variants.

\begin{manualdefinition}{\ref*{colimit-creation}$'$}
    Let $p \colon Y \lto Z$ be a loose-cell, and let $f \colon Z \to W$ and $g \colon W \to X$ be tight-cells, such that the colimit $p \wc (f \d g)$ exists.
    The colimit $p \wc (f \d g)$ is \emph{non-strictly created by $g$} when
    \begin{enumerate}
        \item for every $p$-cocone $(w, \lambda')$ for $f$, the $p$-cocone $((w \d g), (\lambda' \d g))$ is the $p$-colimit of $(f \d g)$ if and only if $(w, \lambda')$ is the $p$-colimit of $f$;
        \item $W$ admits the colimit $p \wc f$.
    \end{enumerate}
\end{manualdefinition}

\begin{manualdefinition}{\ref*{limit-creation}$'$}
    Let $p \colon Y \lto Z$ be a loose-cell, and let $f \colon Y \to W$ and $g \colon W \to X$ be tight-cells, such that the limit $p \wl (f \d g)$ exists.
    The limit $p \wl (f \d g)$ is \emph{non-strictly created by $g$} when
    \begin{enumerate}
        \item for every $p$-cone $(w, \mu')$ for $f$, the $p$-cone $((w \d g), (\mu' \d g))$ is the $p$-limit of $(f \d g)$ if and only if $(w, \mu')$ is the $p$-limit of $f$;
        \item $W$ admits the limit $p \wl f$.
    \end{enumerate}
\end{manualdefinition}

As with strict creation, non-strict creation of (co)limits implies preservation. Conversely, preservation of (co)limits implies non-strict creation assuming conservativity.

\begin{lemma}
    \label{preservation-and-conservativity-implies-creation}
    Let $p \colon Y \lto Z$ be a loose-cell and $g \colon W \to X$ be a conservative tight-cell. For any tight-cell $f \colon Z \to W$, if $W$ admits a colimit $p \wc f$ and $g$ preserves it, then $g$ non-strictly creates the colimit $p \wc (f \d g)$. Similarly, for any tight-cell $f \colon Y \to W$, if $W$ admits a limit $p \wl f$ and $g$ preserves it, then $g$ non-strictly creates the limit $p \wl (f \d g)$.
\end{lemma}

\begin{proof}
    Since $g$ is conservative, a 2-cell $\phi \colon p \wc f \tto w$ is invertible if and only if the 2-cell $(\phi \d g) \colon (p \wc f) \d g \iso p \wc (f \d g) \tto (w \d g)$ is invertible. Therefore, the $p$-cocone corresponding to $\phi$ is colimiting if and only if the $p$-cocone corresponding to $(\phi \d g)$ is colimiting. The limit case follows analogously.
\end{proof}

The non-strict analogues of \cref{u_T-creates-j-absolute-colimits,u_T-creates-limits} follow.

\begin{proposition}
    \label{u_T-nonstrict-creation}
    Let $\jAE$ be a tight-cell and let $T$ be $j$-monad admitting an algebra object $(u_T, \aop_T)$.
    The tight-cell $u_T \colon \Alg(T) \to E$ non-strictly creates limits and $j$-absolute colimits.
\end{proposition}

\begin{proof}
    Follows directly from \cref{u_T-creates-j-absolute-colimits,u_T-creates-limits,u_T-is-conservative,preservation-and-conservativity-implies-creation}.
\end{proof}

\begin{remark}
    From \cref{u_T-nonstrict-creation}, we recover \cite[Theorem~7.2]{pare1971absolute} for non-relative monads in~$\Cat$.
\end{remark}

\section{Relative monadicity}
\label{monadicity}

In this section, we establish a relative monadicity theorem (\cref{relative-monadicity-theorem}(\hyperref[relative-monadicity-theorem']{$'$})) for relative monads with dense roots. We start by reviewing our approach to strictness.

The universal property of an algebra object (\cref{algebra-object}) identifies it up to isomorphism. Consequently, every object isomorphic to an algebra object is itself an algebra object for the same relative monad. This aligns with the general philosophy of \cite{arkor2024formal}, where we are concerned with \emph{categories} of relative monads, in which the appropriate notion of indistinguishability is isomorphism (rather than with 2-categories or \vdcs{} of relative monads, in which the appropriate notion of indistinguishability would be equivalence). Our primary interest in relative monadicity thus concerns \emph{strict} relative monadicity, characterising when a tight-cell $r \colon D \to E$ is isomorphic in $\tX/E$ to the forgetful tight-cell for a $j$-monad.

We therefore proceed to first establish a strict relative monadicity theorem, as it requires more care than the non-strict variant. Since a non-strict relative monadicity theorem is also of interest, we provide one in \cref{non-strict}. However, as the proof strategy is identical to the strict case, up to appropriate replacement of identities with isomorphisms, our treatment of the non-strict case is succinct.

We shall now outline our proof strategy, and discuss the subtleties that arise in the formal setting. Fix a tight-cell $\jAE$. Given a tight-cell $r \colon D \to E$, we should like to characterise when it exhibits the forgetful tight-cell for an algebra object for a $j$-monad. Assuming that algebra objects exist, we may reformulate this condition, using the universal property, into the following.
\begin{definition}
    \label{strictly-j-monadic}
    A tight-cell $r \colon D \to E$ is \emph{strictly $j$-relatively monadic} (alternatively \emph{strictly monadic relative to $j$}, or simply \emph{strictly $j$-monadic}) if it admits a left $j$-adjoint, the induced $j$-monad $T$ admits an algebra object, and the comparison tight-cell $\unit_{\ljr} \colon D \to \Alg(T)$ is an isomorphism.
    \[\begin{tikzcd}
        D && {\Alg(T)} \\
        & E
        \arrow["r"', from=1-1, to=2-2]
        \arrow["{\unit_{\ljr}}", from=1-1, to=1-3]
        \arrow["{u_T}", from=1-3, to=2-2]
    \end{tikzcd}\]
\end{definition}
Explicitly, this means that, to exhibit $r$ as the forgetful tight-cell for an algebra object, we must construct a tight-cell $\Alg(T) \to D$ that is inverse to the comparison tight-cell $\unit_{\ljr}$. To do so, we shall observe that $u_T \colon \Alg(T) \to E$ can be exhibited as a certain weighted colimit (\cref{right-morphism-exhibits-left-extension}). Supposing that $r$ creates this colimit, the lifting then provides the desired inverse. Furthermore, the weighted colimit in question will be shown to be $j$-absolute, and so creation of $j$-absolute colimits provides both a necessary and sufficient condition for $j$-monadicity, by \cref{u_T-creates-j-absolute-colimits}.

Note, however, that when the $j$-monad induced by $\ljr$ does not admit an algebra object, such an approach is not possible, as there is no comparison tight-cell that we may ask to be invertible.
The relative monadicity theorem we prove in this section is therefore a \emph{detection} result, rather than a \emph{construction} result: it characterises when a tight-cell is relatively monadic \emph{assuming that an algebra object exists}. This is not a trivial assumption even for enriched (non-relative) monadicity. However, it is also not a particularly restrictive assumption, as algebra objects typically exist in settings of interest (for instance, when the enriching category $\V$ is closed and admits enough limits \cite[Corollary~8.20]{arkor2024formal}).

\subsection{Strict relative monadicity}
\label{strict}

The first step towards a relative monadicity theorem is to identify the appropriate colimits for creation. To do so, we observe in the following proposition that, for dense $j$, a right-morphism of relative adjunctions $(\ljr) \to (\ljrp)$, exhibits $r'$ as a left extension of $r$ (\cf{}~\cite[Proposition~5.10]{arkor2024formal}).
In particular, this is true of morphisms of resolutions, in which the 2-cell $\rho$ below is the identity (\cite[Definition~5.23]{arkor2024formal}).

\begin{remark}
    \label{uniqueness-of-rho}
    Recall that, in general, a right-morphism $(c, \rho) \colon (\ljr) \to (\ljrp)$ comprises a tight-cell $c \colon C \to C'$ commuting with the left relative adjoints, and a 2-cell $\rho \colon r \tto (c \d r')$ compatible with the transposition operators~\cite[Definition~5.18]{arkor2024formal}. However, when $j$ is dense, the 2-cell $\rho$ is uniquely determined by the transposition operators~\cite[Lemma~5.21]{arkor2024formal}. This is the case in the following proposition.
\end{remark}

\begin{proposition}
    \label{right-morphism-exhibits-left-extension}
    Let $\jAE$ be a dense tight-cell and let $\ljr$ be a $j$-adjunction. Consider tight-cells $c \colon C \to C'$ and $r' \colon C' \to E$, and a 2-cell $\rho \colon r \tto (c \d r')$. Define $\ell' \defeq (\ell \d c)$.
	\[\begin{tikzcd}
		& {C'} \\
		A & C & E
		\arrow["\ell"', from=2-1, to=2-2]
		\arrow["r"', from=2-2, to=2-3]
		\arrow["{\ell'}", from=2-1, to=1-2]
		\arrow["c"{description}, from=2-2, to=1-2]
		\arrow[""{name=0, anchor=center, inner sep=0}, "{r'}", from=1-2, to=2-3]
		\arrow["\rho"', shorten >=3pt, Rightarrow, from=2-2, to=0]
	\end{tikzcd}\]
    The following are equivalent.
    \begin{enumerate}
        \item $\ljrp$, and $(c, \rho) \colon (\ljr) \to (\ljrp)$ is a right-morphism of $j$-adjunctions.
        \item $\rho$ exhibits $r'$ as the $j$-absolute left extension $c \plx r$.
    \end{enumerate}
\end{proposition}

\begin{proof}
    Observe that, since $\ljr$, the loose-composite $E(j, r) \odot C'(c, 1)$ is isomorphic to $C'(\ell', 1)$:
    \[C'(\ell', 1) = C'(c \ell, 1) \iso C(\ell, 1) \odot C'(c, 1) \iso E(j, r) \odot C'(c, 1)\]

    (1) $\implies$ (2). Since $\ljrp$, we have
    \[E(j, r') \iso C'(\ell', 1) \iso E(j, r) \odot C'(c, 1)\]
    from which the result follows by \cite[Lemma~3.23]{arkor2024formal} using that $j$ is dense.

    (2) $\implies$ (1). We have
    \[C'(\ell', 1) \iso E(j, r) \odot C'(c, 1) \iso E(j, c \plx r)\]
    using $j$-absoluteness and \cite[Lemma~3.23]{arkor2024formal}, so that $\ljrp$. That $(c, \rho)$ forms a right-morphism follows from the definition of $\ljrp$.
\end{proof}

Since the left extensions of \cref{right-morphism-exhibits-left-extension} are $j$-absolute, they are created by the forgetful tight-cell $u_T \colon \Alg(T) \to E$ of any algebra object for a $j$-monad. We shall show shortly that creation of such colimits is sufficient to imply $j$-monadicity. The observation above motivates the following definition.

\begin{definition}
    Let $r \colon D \to E$ be tight-cell. An \emph{$r$-extension} is a left extension $c \plx r$, for some tight-cell $c$ with domain $D$.
\end{definition}

The core technical lemma in the proof of the relative monadicity theorem is the following.

\begin{lemma}
    \label{split-resolution-morphisms}
    Let $\jAE$ be a dense tight-cell, and let $\ell \jadj r$ be a resolution of a $j$-monad $T$. If $r$ strictly creates $j$-absolute $r$-extensions, then every morphism of resolutions with domain $(\ell \jadj r)$ admits a retraction $c \plx 1_C$.
\end{lemma}

\begin{proof}
  Let $c \colon (\ell \jadj r) \to (\ell' \jadj r')$ be a morphism of resolutions of $T$.
\[
\begin{tikzcd}
	& {C'} \\
	A & C & E
	\arrow["\ell"', from=2-1, to=2-2]
	\arrow["r"', from=2-2, to=2-3]
	\arrow["{\ell'}", from=2-1, to=1-2]
	\arrow["c"{description}, from=2-2, to=1-2]
	\arrow["{r'}", from=1-2, to=2-3]
\end{tikzcd}
\]
  By \cref{right-morphism-exhibits-left-extension}, using density of $j$, we have $r' \iso c \plx r$.
  This left extension is a $j$-absolute $r$-extension, so is strictly created by $r$.
  Hence there is a unique pair of a tight-cell $d \colon C' \to C$ and 2-cell
  $\eta \colon 1_C \tto c \d d$ such that $d \d r = r'$ and
  $\eta \d r = 1_r$. Furthermore, $\eta$ exhibits $d$ as the left extension $c \plx 1_C$.
  We shall prove that $c \d d = 1_C$: this implies that $d$ is a morphism of resolutions, and hence is a retraction of $c$ in the category of resolutions of $T$, because $\ell \d c = \ell'$.

  The identity 2-cell $1_r$ exhibits $r$ as the $j$-absolute $r$-extension $1_{C} \plx r$ trivially.
  This left extension is strictly created by $r$, so there is a unique pair $(x, \chi)$ of a tight-cell $x \colon C \to C$ and 2-cell $\chi \colon 1_C \tto x$ such that $x \d r = r$ and $\chi \d r = 1_r$.
  Clearly $(1_C, 1_{1_C})$ is such a pair, but, by the above, and the fact that $c \d r' = r$, so is $((c \d d), \eta)$.
  Hence $c \d d = 1_C$ as required.
\end{proof}

In passing, we note that, when algebra objects exist, a relative adjunction is a terminal resolution if and only if its right relative adjoint exhibits an algebra object~\cite[Corollary~6.41]{arkor2024formal}. However, if algebra objects are not known to exist, terminality is a weaker condition than exhibiting an algebra object. The following corollary of \cref{split-resolution-morphisms} is useful when terminality is sufficient. One could use this fact to give a proof of the relative monadicity theorem; however, we shall give a different proof that generalises more easily to the non-strict setting.

\begin{corollary}
    \label{creates-j-absolute-colimits-implies-strictly-j-monadic}
    Let $\jAE$ be a dense tight-cell, let $T$ be a $j$-monad admitting a terminal resolution $f \jadj u$, and let $\ljr$ be a resolution of $T$. If $r$ strictly creates $j$-absolute $r$-extensions, then $\ljr$ is a terminal resolution of $T$.
    \[\begin{tikzcd}
    	D && M \\
    	& E
    	\arrow["r"', from=1-1, to=2-2]
    	\arrow["u", from=1-3, to=2-2]
    	\arrow["{\unit_{\ljr}}", from=1-1, to=1-3]
    \end{tikzcd}\]
\end{corollary}

\begin{proof}
    Since $(f \jadj u)$ is terminal, there is a unique morphism $(\ell \jadj r) \to (f \jadj u)$. By \cref{split-resolution-morphisms}, this morphism is a split monomorphism, because $r$ strictly creates $j$-absolute $r$-extensions. However, every split monomorphism into a terminal object is an isomorphism, so $(\ell \jadj r) \iso (f \jadj u)$.
\end{proof}

The relative monadicity theorem follows essentially directly from \cref{u_T-creates-j-absolute-colimits,split-resolution-morphisms}.

\begin{theorem}[Relative monadicity]
    \label{relative-monadicity-theorem}
    Let $\jAE$ be a dense tight-cell. A tight-cell $r \colon D \to E$ is strictly $j$-monadic if and only if $r$ admits a left $j$-adjoint, the induced $j$-monad admits an algebra object, and $r$ strictly creates $j$-absolute $r$-extensions.
\end{theorem}

\begin{proof}
    If $r$ is strictly $j$-monadic, then it admits a left $j$-adjoint for which the induced $j$-monad has an algebra object, by definition. Furthermore, $r$ is the composite of an isomorphism with $u_T$, hence strictly creates $j$-absolute $r$-extensions by \cref{u_T-creates-j-absolute-colimits}. For the converse, observe that $r$ and $u_T$ both strictly create $j$-absolute $r$-extensions, so that $\unit_{\ljr}$ and $\unit_{\ljr} \plx 1_D$ exhibit $D$ and $\Alg(T)$ as retracts of one another by \cref{split-resolution-morphisms}, and hence $\unit_{\ljr}$ is invertible.
\end{proof}

We may additionally relax the class of colimits that $r$ need create in order to be $j$-monadic, so that the class is independent of $r$.

\begin{corollary}
    \label{relative-monadicity-theorem-absolute}
    Let $\jAE$ be a dense tight-cell. A tight-cell $r \colon D \to E$ is strictly $j$-monadic if and only if $r$ admits a left $j$-adjoint, the induced $j$-monad admits an algebra object, and $r$ strictly creates $j$-absolute colimits.
\end{corollary}

\begin{proof}
    Follows directly from \cref{relative-monadicity-theorem}, since $u_T$ strictly creates all $j$-absolute colimits by \cref{u_T-creates-j-absolute-colimits}.
\end{proof}

In practice, \cref{relative-monadicity-theorem-absolute} is often more convenient: for instance, in \cref{relative-monadicity-in-VCat}, we shall give a simple characterisation of the $j$-absolute colimits in $\VCat$ for a well-behaved monoidal category $\V$. However, in \cref{monadic-iff-left-adjoint} we shall give an example of a situation in which the sharper result of \cref{relative-monadicity-theorem} is useful.

\begin{remark}
    The assumption that $j$ be dense in the statement of \cref{relative-monadicity-theorem} is necessary. For instance, denoting by $0$ the empty category, every functor $r \colon D \to E$ (for arbitrary categories $D$ and $E$) is right adjoint to the unique functor $[]_D \colon 0 \to D$ relative to the unique functor $[]_E \colon 0 \to E$, which is dense only when $E$ is indiscrete.
    \[
    \begin{tikzcd}
        & D \\
        0 && E
        \arrow[""{name=0, anchor=center, inner sep=0}, "{[]_D}", from=2-1, to=1-2]
        \arrow[""{name=1, anchor=center, inner sep=0}, "r", from=1-2, to=2-3]
        \arrow["{[]_E}"', from=2-1, to=2-3]
        \arrow["\dashv"{anchor=center}, shift right=2, draw=none, from=0, to=1]
    \end{tikzcd}
  \]
  The trivial $[]_E$-monad is the unique $[]_E$-monad. Consequently, $r \colon D \to E$ is $[]_E$-monadic if and only if it is an isomorphism \cite[Proposition~6.42]{arkor2024formal}. However, every colimit in $E$ is $[]_E$-absolute~(by \cref{j-absolute-is-preservation-by-n_j} below), and so any non-invertible functor $r \colon D \to E$ that strictly creates colimits is a counterexample to the statement of \cref{relative-monadicity-theorem}. We leave as an open question whether there are sufficient conditions for relative monadicity in the absence of density.
\end{remark}

\begin{remark}
    For each tight-cell $\jAE$, the category of $j$-monads is equipped with a \ff{} partial functor $u_{\ph} \colon \RMnd(j)\op \ffto \tX/E$ to the category of slices over $E$, sending each relative monad $T$ that admits an algebra object to the forgetful tight-cell $u_T \colon \Alg(T) \to E$ from its algebra object~\cite[Corollary~6.40]{arkor2024formal}. Abstractly, the relative monadicity theorem may be seen as a characterisation of the essential image of this partial functor, consequently exhibiting an equivalence between the full subcategory of $\RMnd(j)\op$ spanned by $j$-monads admitting algebra objects, and the full subcategory of $\tX/E$ spanned by $j$-monadic tight-cells.
\end{remark}

\begin{remark}
    Algebra objects for monads in a 2-category (which are weaker than algebra objects for monads in an equipment~\cite[61]{arkor2024formal}) can be characterised representably in terms of algebra objects for monads in the 2-category $\b{Cat}$~\cite[Theorem~8]{street1972formal}. Consequently, monadicity for monads in a 2-category~$\K$ may be characterised representably in terms of monadicity in $\b{Cat}$~\cite[Corollary~8.1]{street1972formal}. \textcite[Proposition~22]{wood1985proarrows} makes use of this characterisation to give a formal monadicity theorem, which is, in essence, an objectwise version of \citeauthor{beck1966untitled}'s monadicity theorem~\cite{beck1966untitled}. \citeauthor{wood1985proarrows}'s monadicity theorem is thus of an entirely different nature to \cref{relative-monadicity-theorem}, which is a characterisation internal to the equipment $\X$.
    A consequence is that \cref{relative-monadicity-theorem}, in contrast to \citeauthor{wood1985proarrows}'s monadicity theorem, directly specialises to the monadicity theorem for enriched monads.
\end{remark}

A particularly useful consequence of the relative monadicity theorem is the following proposition, which relates algebra objects for relative monads that have different roots. Observe that the proof follows easily from \cref{relative-monadicity-theorem}, whereas a proof based on \cref{relative-monadicity-theorem-absolute} is less straightforward.

\begin{proposition}
    \label{monadic-iff-left-adjoint}
    Let $\jAE$ and $\jEI$ be tight-cells, and let $T$ be a $(j \d j')$-monad admitting an algebra object. Suppose that $j'$ is dense. Then $u_T \colon \Alg(T) \to E'$ is strictly $j'$-monadic if and only if it admits a left $j'$-adjoint and the induced $j'$-monad admits an algebra object.
\end{proposition}

\begin{proof}
    The only if direction is trivial. For the other direction, observe that every $j'$-absolute $r$-extension is a $(j \d j')$-absolute $r$-extension, since every right-morphism of $j'$-adjunctions induces a right-morphism of $(j \d j')$-adjunctions by precomposing $j$~\cite[Proposition~5.29]{arkor2024formal}. Therefore, since $u_T$ strictly creates $(j \d j')$-absolute $r$-extensions, it strictly creates, in particular, $j'$-absolute $r$-extensions, from which the result follows by \cref{relative-monadicity-theorem}.
\end{proof}

In particular, if every monad admits an algebra object, we obtain that the algebra object for a relative monad $T$ is the algebra object for a monad whenever the forgetful tight-cell $u_T \colon \Alg(T) \to E$ admits a left adjoint.

\subsection{Non-strict relative monadicity}
\label{non-strict}

We now briefly discuss the case of non-strict relative monadicity. While strict relative monadicity is often the appropriate notion, there are situations in which it is convenient to consider a property that is invariant under equivalence. This motivates the following definition.

\begin{manualdefinition}{\ref*{strictly-j-monadic}$'$}
    \label{strictly-j-monadic'}
    Let $\jAE$ be a tight-cell. A tight-cell $r \colon D \to E$ is \emph{non-strictly $j$-relatively monadic} (alternatively \emph{non-strictly monadic relative to $j$}, or simply \emph{non-strictly $j$-monadic}) if it admits a left $j$-adjoint, the induced $j$-monad $T$ admits an algebra object, and the comparison tight-cell $\unit_{\ljr} \colon D \to \Alg(T)$ is an equivalence.
\end{manualdefinition}

Informally, a characterisation of non-strict relative monadicity is obtained by relaxing \cref{relative-monadicity-theorem} through the replacement of the equality of tight-cells by isomorphism, and thereby asking for non-strict creation of colimits rather than strict creation. With respect to the proof strategy, there is one notable difference to the strict case. While, just as in the strict setting, the comparison tight-cell $\unit_{\ljr} \colon D \to \Alg(T)$ is a morphism of resolutions, the tight-cell $\Alg(T) \to D$ constructed via the non-strict creation of colimits only commutes with the left and right adjoints up to isomorphism. However, once this subtlety is taken into account, the proof then proceeds in essentially the same way as the strict case.

Below, our proof follows the structure of the strict case. Each statement is numbered according to its non-strict pair and marked with a prime ($'$) to denote non-strictness.

\begin{manuallemma}{\ref*{split-resolution-morphisms}$'$}
    \label{split-resolution-morphisms'}
    Let $\jAE$ be a dense tight-cell, and let $\ell \jadj r$ be a resolution of a $j$-monad $T$. If $r$ non-strictly creates $j$-absolute $r$-extensions, then every tight-cell $c \colon C \to C'$ rendering pseudo-commutative the diagram on the left below admits a pseudo-retraction $c \plx 1_C$ rendering pseudo-commutative the diagram on the right below.
    \[
    \begin{tikzcd}[row sep=small]
        & C \\
        A && E \\
        & {C'}
        \arrow["r", from=1-2, to=2-3]
        \arrow[""{name=0, anchor=center, inner sep=0}, "c"{description}, from=1-2, to=3-2]
        \arrow["\ell", from=2-1, to=1-2]
        \arrow["{\ell'}"', from=2-1, to=3-2]
        \arrow["{r'}"', from=3-2, to=2-3]
        \arrow["\iso"{description}, draw=none, from=0, to=2-3]
        \arrow["\iso"{description}, draw=none, from=2-1, to=0]
    \end{tikzcd}
    \hspace{4em}
    \begin{tikzcd}[row sep=small]
        & C \\
        A && E \\
        & {C'}
        \arrow["r", from=1-2, to=2-3]
        \arrow["\ell", from=2-1, to=1-2]
        \arrow["{\ell'}"', from=2-1, to=3-2]
        \arrow[""{name=0, anchor=center, inner sep=0}, "{c \plx 1_C}"{description}, from=3-2, to=1-2]
        \arrow["{r'}"', from=3-2, to=2-3]
        \arrow["\iso"{description, pos=0.45}, draw=none, from=2-1, to=0]
        \arrow["\iso"{description, pos=0.55}, draw=none, from=0, to=2-3]
    \end{tikzcd}
    \]
\end{manuallemma}

\begin{sketch}
    We have a $j$-adjunction $(\ell \d c) \jadj r'$ using $\ell' \iso (\ell \d c)$.
    Since $j$ is dense, there is a unique right-morphism $(c, \rho) \colon (\ljr) \to ((\ell \d c) \jadj r')$ of $j$-adjunctions (\cf~\cref{uniqueness-of-rho}).
    Therefore, by \cref{right-morphism-exhibits-left-extension}, $\rho$ exhibits $r'$ as the $j$-absolute $r$-extension $c \plx r$.
    The proof of \cref{split-resolution-morphisms} then carries through with respect to non-strict creation of colimits after replacing the equalities of tight-cells with isomorphisms.
\end{sketch}

\begin{manualtheorem}{\ref*{relative-monadicity-theorem}$'$}
    \label{relative-monadicity-theorem'}
    Let $\jAE$ be a dense tight-cell. A tight-cell $r \colon D \to E$ is non-strictly $j$-monadic if and only if $r$ admits a left $j$-adjoint, the induced $j$-monad admits an algebra object, and $r$ non-strictly creates $j$-absolute $r$-extensions.
\end{manualtheorem}

\begin{proof}
    If $r$ is non-strictly $j$-monadic, then it admits a left $j$-adjoint for which the induced $j$-monad has an algebra object, by definition. Furthermore, $r$ is the composite of an equivalence with $u_T$, hence non-strictly creates $j$-absolute $r$-extensions by \cref{u_T-nonstrict-creation}. For the converse, observe that $r$ and $u_T$ both non-strictly create $j$-absolute $r$-extensions, so that $\unit_{\ljr}$ and $\unit_{\ljr} \plx 1_D$ exhibit $D$ and $\Alg(T)$ as pseudo-retracts of one another by \cref{split-resolution-morphisms'}, and hence that $\unit_{\ljr}$ is an equivalence.
\end{proof}

\begin{manualcorollary}{\ref*{relative-monadicity-theorem-absolute}$'$}
    \label{relative-monadicity-theorem-absolute'}
    Let $\jAE$ be a dense tight-cell. A tight-cell $r \colon D \to E$ is non-strictly $j$-monadic if and only if $r$ admits a left $j$-adjoint, the induced $j$-monad admits an algebra object, and $r$ non-strictly creates $j$-absolute colimits.
\end{manualcorollary}

\begin{proof}
    Follows directly from \cref{relative-monadicity-theorem'} together with \cref{u_T-nonstrict-creation}.
\end{proof}

Equivalently, in the statements of \cref{relative-monadicity-theorem',relative-monadicity-theorem-absolute'}, rather than ask that $r$ non-strictly create the requisite colimits, we could ask for $r$ to be conservative and preserve the requisite colimits (using \cref{preservation-and-conservativity-implies-creation} and that conservative tight-cells are closed under equivalences and composition), which is often more convenient in practice.

\begin{manualproposition}{\ref*{monadic-iff-left-adjoint}$'$}
    \label{monadic-iff-left-adjoint'}
    Let $\jAE$ and $\jEI$ be tight-cells and suppose that $j'$ is dense. A non-strictly $(j \d j')$-monadic tight-cell $r \colon D \to E'$ is non-strictly $j'$-monadic if and only if it admits a left $j'$-adjoint and the induced $j'$-monad admits an algebra object.
\end{manualproposition}

\begin{proof}
    The proof of \cref{monadic-iff-left-adjoint} carries through with respect to non-strict creation of colimits.
\end{proof}

\section{Enriched relative monadicity}
\label{relative-monadicity-in-VCat}

Instantiating \cref{relative-monadicity-theorem}(\hyperref[relative-monadicity-theorem']{$'$}) and its corollaries in the equipment $\VCat$ of categories enriched in a monoidal category~$\V$ (\cite[Definition~8.1]{arkor2024formal}), we immediately obtain a relative monadicity theorem for enriched relative monads. Under some additional assumptions on $\V$, we may make some simplifications.
We first make note of the following characterisation of $j$-absolute colimits, which is a special case of \cite[Lemma~8.10]{arkor2024formal}.
Below, when they exist, we write $\P A$ for the $\V$-category of presheaves on a $\V$-category $A$, and write $n_j \colon E \to \P A$ for the nerve of a $\V$-functor $\jAE$, which is defined by $n_j e \defeq E(j{-}, e)$. Dually, we write $\cl Q Z$ for the $\V$-category of copresheaves on a $\V$-category $Z$, and write $m_i \colon U \to \cl Q Z$ for the co-nerve of a $\V$-functor $i \colon Z \to U$, which is defined by $m_i u \defeq U(u, i{-})$.

\begin{lemma}
    \label{j-absolute-is-preservation-by-n_j}
    Let $\V$ be a complete left- and right-closed monoidal category and let $\jAE$ be a $\V$-functor with small domain.
    Given a $\V$-distributor $p \colon X \lto Y$ and $\V$-functor $f \colon Y \to E$, a colimit $p \wc f$ in $E$ is $j$-absolute exactly when it is preserved by the nerve $n_j \colon E \to \P A$.
\end{lemma}

\begin{proof}
    Since $A$ is small, $\P A$ exists by \cite[\S3]{gordon1999gabriel}, and the result follows by \cite[Lemma~8.10]{arkor2024formal}.
\end{proof}

We spell out \cref{relative-monadicity-theorem-absolute}(\hyperref[relative-monadicity-theorem-absolute']{$'$}) in the equipments $\VCat$ and $\VCat\co$, where $\V$ is a well-behaved monoidal category, using the characterisation of $j$-absolute colimits above.

\begin{theorem}[Enriched relative monadicity]
    \label{enriched-relative-monadicity-theorem}
    Let $\V$ be a complete left- and right-closed monoidal category and let $\jAE$ be a dense $\V$-functor with small domain.
    A $\V$-functor $r \colon D \to E$ is (non)strictly $j$-monadic if and only if $r$ admits a left $j$-adjoint and (non)strictly creates those colimits that are preserved by the nerve $n_j \colon E \to \P A$.
\end{theorem}

\begin{proof}
    Since $A$ is small, $\V$-categories of algebras for $j$-monads exist by \cite[Corollary~8.20]{arkor2024formal}.
    The result then follows directly from \cref{relative-monadicity-theorem-absolute}(\hyperref[relative-monadicity-theorem-absolute']{$'$}) together with \cref{j-absolute-is-preservation-by-n_j}.
\end{proof}

\begin{manualtheorem}{\ref*{enriched-relative-monadicity-theorem}$\co$}[Enriched relative comonadicity]
    \label{enriched-relative-comonadicity-theorem}
    Let $\V$ be a complete left- and right-closed monoidal category and let $i \colon Z \to U$ be a codense $\V$-functor with small domain.
    A $\V$-functor $\ell \colon W \to U$ is (non)strictly $i$-comonadic if and only if $\ell$ admits a right $i$-coadjoint and (non)strictly creates those limits that are preserved by the co-nerve $m_i \colon U \to \cl Q Z$.
\end{manualtheorem}

\begin{proof}
    Since $A$ is small, copresheaf $\V$-categories exist, as do $\V$-categories of coalgebras for $i$-comonads by \cite[Theorem~8.23]{arkor2024formal}. The result then follows as for \cref{enriched-relative-monadicity-theorem}.
\end{proof}

\begin{remark}
    When $A$ is large, presheaf $\V$-categories can no longer be assumed to exist, and so $j$-absoluteness cannot be characterised in terms of the nerve without a change of enrichment base. In this case \cref{relative-monadicity-theorem}(\hyperref[relative-monadicity-theorem']{$'$}) and \cref{relative-monadicity-theorem-absolute}(\hyperref[relative-monadicity-theorem-absolute']{$'$}) may be used directly.
\end{remark}

From \cref{enriched-relative-monadicity-theorem}, taking $j = 1$, we recover the unenriched monadicity theorem of \citeauthor{pare1969absolute}~\cites[Theorem~3.5]{pare1969absoluteness}[Theorem]{pare1969absolute}[Theorem~7.3]{pare1971absolute}, which is a reformulation of the unenriched monadicity theorem of \textcite{beck1966untitled} in terms of absolute colimits; as well as an analogous characterisation of enriched monadicity (\cf{}~\cites[Theorem~2.11]{bunge1969relative}[Theorem~II.2.1]{dubuc1970kan}). Taking $j$ to be dense and \ff{}, we recover the unenriched relative monadicity theorems of \citeauthor{diers1975jmonades}~\cites[Th\'eorem\`e~2.5]{diers1975jmonades}[Theoreme~5.1]{diers1976foncteur} and \textcite[Corollary~2.8]{lee1977relative}.

\begin{remark}
    It does not appear to be possible to obtain a characterisation of relative monadicity in terms of contractible coequalisers without much stronger assumptions on the root $j$ that seldom hold in examples of interest other than $j = 1$; we do not pursue such a characterisation.
\end{remark}

\begin{remark}
    The formal relative monadicity theorem \cref{relative-monadicity-theorem-absolute}(\hyperref[relative-monadicity-theorem-absolute']{$'$}) applies to a relative monad $T$ with a dense root as soon as $T$ admits an algebra object. Therefore, establishing a relative monadicity theorem in a given \ve{} $\X$ reduces to establishing the existence of algebra objects in $\X$. In \cite{arkor2024formal}, the existence of algebra objects was studied only for the \ve{} $\VCat$. However, as we shall show in future work, the proof given there carries through with minor modifications in the more general context of the \ve{} $\L\dc{Mnd}(\X)$ of loose-monads and modules in a \vdc{} $\X$ with restrictions (\cf~\cite[Proposition~7.4]{cruttwell2010unified}). In this way, we shall obtain relative monadicity theorems for other notions of category, such as internal categories and generalised multicategories, as well as for more general notions of enriched category~\cite{leinster1999generalized}.
\end{remark}

\subsection{Examples}

We now demonstrate the usefulness of the relative monadicity theorem by presenting several prototypical situations in which the theorem may be applied.

In the first situation we present, one has a notion of theory and an accompanying notion of algebra. For instance, we may show that the category of algebras for a finitary algebraic theory is monadic relative to the inclusion of finite sets into small sets (\cf{}~\cite{linton1966triples,linton1969outline,diers1975jmonades,diers1976foncteur}).

\begin{example}[Algebraic theories induce monads]
    \label{algebraic-theories-induce-monads}
    Denote by $\F$ the free category with strict finite coproducts on a single object $1$. The inclusion $j \colon \F \equiv \FinSet \ffto \Set$ of finite ordinals into small sets exhibits a cocompletion under sifted colimits, and is hence dense and \ff{}.

    Now consider a finitary algebraic theory $k \colon \F \to K$ in the sense of \cite[Chapter~2]{lawvere1963functorial}, \ie{} an \ioo{} functor preserving finite coproducts. The category of algebras for $k$ is, up to equivalence, the category $\Cart[K\op, \Set]$ of finite-product-preserving functors from $K\op$ to $\Set$. The forgetful functor, which we denote by $u_k \colon \Cart[K\op, \Set] \to \Set$, is the composite of $\Cart[k\op, \Set]$ with the equivalence $\Cart[\F\op, \Set] \equiv \Set$. Since $K$ has finite coproducts, $u_k$ is a continuous and sifted-cocontinuous functor between locally strongly finitely presentable categories~\cite{adamek2001sifted}, and consequently admits a left adjoint $f_k \colon \Set \to \Cart[K\op, \Set]$. Denoting by $\sigma_L \colon K \ffto \Cart[K\op, \Set]$ the Yoneda embedding, we therefore have the following diagram in $\CAT$, in which the rightmost square is a pseudopullback.
    \[\begin{tikzcd}
        K & {\Cart[K\op, \Set]} & {[K\op, \Set]} \\
        \F & \Set & {[\F\op, \Set]}
        \arrow[""{name=0, anchor=center, inner sep=0}, "{u_k}", curve={height=-12pt}, from=1-2, to=2-2]
        \arrow["{n_{\sigma_K}}", hook, from=1-2, to=1-3]
        \arrow["{n_j}"', hook, from=2-2, to=2-3]
        \arrow[""{name=1, anchor=center, inner sep=0}, "{[k\op, \Set]}", from=1-3, to=2-3]
        \arrow[""{name=2, anchor=center, inner sep=0}, "k", from=2-1, to=1-1]
        \arrow[""{name=3, anchor=center, inner sep=0}, "{f_k}", curve={height=-12pt}, from=2-2, to=1-2]
        \arrow["{\sigma_K}", hook, from=1-1, to=1-2]
        \arrow["j"', hook, from=2-1, to=2-2]
        \arrow["\dashv"{anchor=center}, draw=none, from=3, to=0]
        \arrow["\iso"{description}, draw=none, from=2, to=3]
        \arrow["{\it{ps.\;pb}}"{description}, draw=none, from=0, to=1]
    \end{tikzcd}\]
    Since $k$ is identity-on-objects, $[k\op, \Set]$ is an amnestic isofibration and strictly creates colimits. By the former property, the pseudopullback is equivalent to the strict pullback~\cite[Corollary~1]{joyal1993pullbacks}.
    The forgetful functor $u_k$ therefore non-strictly creates those colimits that $n_j$ preserves \cite[Proposition~21.7.2(c)]{schubert1972categories}, which are precisely the $j$-absolute colimits by \cref{j-absolute-is-preservation-by-n_j}. By \cite[Proposition~5.29]{arkor2024formal}, we have $(j \d f_k) \jadj u_k$ and so $u_k$ is $j$-monadic by \cref{enriched-relative-monadicity-theorem}. Furthermore, by \cref{monadic-iff-left-adjoint'}, it is additionally non-strictly monadic. Every finitary algebraic theory therefore induces both a $j$-monad, and a sifted-cocontinuous monad on $\Set$ -- the former obtained by precomposing the latter by $j$~\cite[Proposition~5.36]{arkor2024formal} -- whose categories of algebras are concretely equivalent.
\end{example}

More generally, the methods of \cref{algebraic-theories-induce-monads} may be used to show that, for an arity class $\kappa$ in the sense of \cite[Definition~2.2]{shulman2012exact} (for instance, any regular cardinal), the category of algebras for a $\kappa$-ary algebraic theory is monadic relative to $\Set_\kappa \ffto \Set$, the full subcategory inclusion of the $\kappa$-small sets. When $\kappa = \N$, we recover finitary algebraic theories; and when $\kappa$ is the cardinality of the universe, we recover infinitary algebraic theories. Another useful case is given by taking $\kappa = \{ 1 \}$, as in the following example.

\begin{example}
  \label{unary-alebraic-theories}
  A unary algebraic theory is an \ioo{} functor with domain $1$, and is understood syntactically to be a theory presented solely by unary operations. Concretely, each unary algebraic theory is equivalent to a monoid $M$, viewed as a one-object category, and its category of algebras is the presheaf category $\Set^{M\op}$ of right-actions of $M$ (sometimes called \emph{right $M$-sets}). Consequently, a functor is $(1 \ffto \Set)$-monadic precisely when its domain is equivalent to a presheaf category on a single object, and its action is given by evaluation at that unique object.

  Alternatively, we may give a characterisation via the relative monadicity theorem. Observe that the nerve of $1 \ffto \Set$ is isomorphic to the identity functor on $\Set$. Hence every colimit is $(1 \ffto \Set)$-absolute. A functor $u \colon D \to \Set$ is therefore $(1 \ffto \Set)$-monadic if and only if it creates colimits and admits a left $(1 \ffto \Set)$-adjoint. This latter condition holds exactly when $u$ is corepresentable, \ie{} when there exists an object $d \in \ob D$ such that $u \iso D(d, {-})$.
\end{example}

\Cref{algebraic-theories-induce-monads,unary-alebraic-theories} demonstrate the typical relationship between monadicity relative to different functors. Suppose that we have functors $\jAE$ and $\jEI$. In general, $(j \d j')$-monadicity is neither stronger nor weaker than $j'$-monadicity: $(j \d j')$-adjointness is a weaker property than $j'$-adjointness, but creation of the requisite $(j \d j')$-absolute colimits is a stronger property than creation of the requisite $j'$-absolute colimits. However, in many cases of interest, a $(j \d j')$-monadic functor will admit a left $j'$-adjoint, in which case \cref{monadic-iff-left-adjoint}(\hyperref[monadic-iff-left-adjoint']{$'$}) will apply.

\begin{example}
  Consider the inclusions $j \colon 1 \ffto \F$ and $j' \colon \F \ffto \Set$, so that $(j \d j')$-monadicity is the $(1 \ffto \Set)$-monadicity of \cref{unary-alebraic-theories} and $j'$-monadicity is the $(\F \ffto \Set)$-monadicity of \cref{algebraic-theories-induce-monads}. By the adjoint functor theorem, the forgetful functor from the category of algebras for any unary algebraic theory admits a left adjoint, and is thus monadic, and hence also $j'$-monadic. Thus $(j \d j')$-monadicity implies $j'$-monadicity. The converse is not true in general: although the forgetful functor from the category of algebras for a finitary algebraic theory admits a left adjoint, and thus also a $(j \d j')$-adjoint, it will not create all colimits in general. For instance, the underlying set of a coproduct of monoids is not the coproduct of their underlying sets.

  More generally, for arity classes $\kappa \subseteq \kappa'$, the methods above may be used to show that the category of algebras for a $\kappa$-ary algebraic theory is always $(\Set_{\kappa'} \ffto \Set)$-monadic. Taking $\kappa'$ to be the cardinality of the universe, every $\kappa$-ary algebraic theory is seen to induce a monad on $\Set$ of rank $\kappa$.
\end{example}

In the second situation we present, one has a class of weights $\Phi$, and a monad $T$ on a free $\Phi$-cocompletion $\Phi(A)$, for which $T$ preserves $\Phi$-weighted colimits. In this case, the category of algebras for $T$ is monadic relative to the cocompletion $A \to \Phi(A)$.

\begin{example}
    \label{cocontinuous-monads}
    Let $\Phi$ be a class of weights, let $A$ be a small $\V$-category admitting a free cocompletion $A \to \Phi(A)$ under $\Phi$-weighted colimits, and let $T$ be a $\Phi$-cocontinuous monad on $\Phi(A)$. Since $T$ is $\Phi$-cocontinuous, it preserves $(A \to \Phi(A))$-absolute colimits~\cite[Theorems~5.29 \& 5.35]{kelly1982basic}. Hence, the forgetful functor $u_T \colon \Alg(T) \to \Phi(A)$ strictly creates $(A \to \Phi(A))$-absolute colimits (\cf{}~\cite[Proposition~4.3.2]{borceux1994handbook2}). Consequently, since every cocompletion is dense, and $u_T$ admits a left $(A \to \Phi(A))$-adjoint~\cite[Proposition~5.29]{arkor2024formal}, \cref{enriched-relative-monadicity-theorem} implies that the forgetful functor $u_T$ is strictly $(A \to \Phi(A))$-monadic. For instance, the category of algebras for a finitary monad $T$ on a locally finitely presentable category $E$ is the category of algebras for the $(E_\tx{fp} \ffto E)$-monad induced by precomposing $T$ with the inclusion of the full subcategory $E_\tx{fp}$ of finitely presentable objects in $E$~\cite[Proposition~5.36]{arkor2024formal}.
\end{example}

\Cref{cocontinuous-monads} gives an alternative method to \cref{algebraic-theories-induce-monads} for proving that finitary algebraic theories induce $(\F \ffto \Set)$-relative monads, at least supposing that one already knows that every algebraic theory induces a sifted-cocontinuous monad on $\Set$ (which is the cocompletion of $\F$ under sifted colimits).

In the third situation, one has a presentation of a structure by operations and equations, and consequently an induced notion of algebra for the presentation. For instance, we may show that the category of algebras for any quantitative equational theory in the sense of \cite[Definition~2.2]{mardare2016quantitative} is relatively monadic. We give the specific example of quantitative semigroups below; the general case follows analogously.

\begin{example}
  \label{metric-spaces}
  Denote by $\Met$ the category of extended metric spaces and nonexpanding maps, and by $j \colon \FinMet \ffto \Met$ the full subcategory inclusion of the finite metric spaces. For a metric space $X$, denote by $d_X \colon X \times X \to [0, \infty]$ its distance function. Since $\Met$ is a cartesian-monoidal category, we may consider the category $\Semigrp(\Met)$ of semigroups internal to $\Met$, \ie{} metric spaces $X$ equipped with an associative function $\ph \monprod \ph \colon X \times X \to X$ such that $\max\{d_X(x_1, x_1'), d_X(x_2, x_2')\} \geq d_X((x_1 \monprod x_2), (x_1' \monprod x_2'))$.
  The forgetful functor $u \colon \Semigrp(\Met) \to \Met$ admits a left adjoint (an explicit construction is given in \cite[\S6]{mardare2016quantitative}), and hence a left $j$-adjoint~\cite[Proposition~5.29]{arkor2024formal}.

  Moreover, we may prove that $u$ strictly creates every $j$-absolute colimit $C \defeq \colim_i (u(D_i))$, where $D$ is a diagram in $\Semigrp(\Met)$.
  To do so, we shall show that each such $C$ admits a unique semigroup structure such that the coprojections $\copi_i$ are semigroup homomorphisms.
  Since the nerve of $j$ preserves the colimit $C$ by assumption, every nonexpansive map $X \to C$ with finite domain factors through some coprojection $\copi_i \colon D_i \to C$.
  In particular, every finite subobject of $C$ factors thus, and consequently every finite subset of $C$ is equal to the image $\copi_i(S)$, for some $i \in \ob D$, of a subset $S \subseteq D_i$, such that $d_{D_i}(x, x') = d_C(\copi_i(x), \copi_i(x'))$ for each pair $x, x' \in S$.
  Uniqueness of the binary operation $\monprod$ on $C$ follows by considering two-element subsets, and associativity by considering three-element subsets.
  Nonexpansiveness of $\monprod$ follows by considering a four-element subset, and using the condition on $d_{D_i}$ above.

  It thereby follows that $u$ is strictly $(\FinMet \ffto \Met)$-monadic.
\end{example}

\Cref{metric-spaces} gives another alternative method to \cref{algebraic-theories-induce-monads}, since every finitary algebraic theory admits a presentation~\cite{birkhoff1935structure}.
In fact, the three situations we have described -- theories, cocontinuous monads, and presentations -- are all facets of a more general relationship. To illustrate this, we give one final example: the relative monadicity of categories of algebras for a $j$-theory in the sense of \cites[Definition~3.1.13]{arkor2022monadic}[Definition~3.1]{lucyshyn2023enriched}.

\begin{example}
    \label{j-theory}
    Let $\V$ be a locally small closed symmetric monoidal category with equalisers, embedding into a complete closed symmetric monoidal category $\V'$~\cite[\S3.11 \& \S3.12]{kelly1982basic}. Let $\jAE$ be a dense and \ff{} $\V$-functor, and let $\ell \colon A \to B$ be a $j$-theory, \ie{} an \ioo{} $\V$-functor admitting a right $j$-adjoint $r \colon B \to E$. The category of \emph{(concrete) algebras} for $\ell$ is defined to be the following pullback in $\V'\h\Cat$~\cite[Definition~3.3]{lucyshyn2023enriched}; we denote by $\yo_B$ the Yoneda embedding.
    \[\begin{tikzcd}
        B & \ell\h\Alg & {[B\op, \V]} \\
        & E & {[A\op, \V]}
        \arrow[""{name=0, anchor=center, inner sep=0}, "{u_\ell}"{description}, from=1-2, to=2-2]
        \arrow["n", hook, from=1-2, to=1-3]
        \arrow["{n_j}"', hook, from=2-2, to=2-3]
        \arrow[""{name=1, anchor=center, inner sep=0}, "{[\ell\op, \V]}", from=1-3, to=2-3]
        \arrow["{\yo_B}", curve={height=-30pt}, hook, from=1-1, to=1-3]
        \arrow["r"', curve={height=12pt}, from=1-1, to=2-2]
        \arrow["i"{description}, dashed, hook, from=1-1, to=1-2]
        \arrow["{\it{pb}}"{description}, draw=none, from=0, to=1]
    \end{tikzcd}\]
    Just as in \cref{algebraic-theories-induce-monads}, it follows that the pullback is equivalent to the pseudopullback, and that $u_\ell$ strictly creates $j$-absolute colimits. The projection $\V'$-functor $n \colon \ell\h\Alg \to [B\op, \V]$ is \ff{}, since $j$ is dense and \ff{} functors are closed under pullback. The relative adjunction $B(\ell, 1) \iso E(j, r)$ consequently induces a mediating $\V'$-functor $i \colon B \to \ell\h\Alg$, which is \ff{}, since $i \d n \iso \yo_B$ and $n$ and $\yo_B$ are \ff{}. By \cite[Proposition~5.16]{kelly1982basic}, we therefore have $n \iso n_i$. Pseudocommutativity of the pseudopullback square thus asserts that $(\ell \d i) \colon A \to \ell\h\Alg$ is left-$j$-adjoint to $u_\ell$.
    Consequently, by \cref{enriched-relative-monadicity-theorem}, the category of algebras for $\ell$ is strictly $j$-monadic in $\V'\h\Cat$. When $u_\ell$ furthermore admits a left adjoint (\ie{} $\ell$ is \emph{admissible} in the sense of \cite[Definition~3.12]{lucyshyn2023enriched}), it is strictly monadic in $\V'\h\Cat$ by \cref{monadic-iff-left-adjoint}, recovering \cite[Proposition~4.5]{lucyshyn2023enriched}.
\end{example}

\Cref{j-theory} turns out to subsume the previous examples. However, explicating this connection is beyond the scope of the present paper. In future work, we shall develop this connection fully, establishing a formal correspondence between these notions (\cf{}~\cites[Chapters~3 \& 7]{arkor2022monadic}[]{arkor2022relative}[\S5.1]{lucyshyn2023enriched}[]{lucyshyn2022diagrammatic}).

\section{A pasting law for relatively monadic adjunctions}
\label{composition-and-monadicity}

In this section, we recall the pasting law for relative adjunctions and show that it respects relative monadicity in an appropriate sense. As a consequence, we derive necessary and sufficient conditions for the composite $(r \d r')$ of a tight-cell $r$ with a $j$-monadic tight-cell $r'$ to be $j$-monadic. We continue to work in the context of a virtual equipment $\X$.

\begin{lemma}[{\cite[Proposition~5.30]{arkor2024formal}}]
    \label{relative-adjunction-pasting}
    Consider the following diagram.
    \[\begin{tikzcd}[sep=small]
        && C \\
        &&& D \\
        A &&&& E
        \arrow["r", from=1-3, to=2-4]
        \arrow["j"', from=3-1, to=3-5]
        \arrow[""{name=0, anchor=center, inner sep=0}, "{r'}", from=2-4, to=3-5]
        \arrow[""{name=1, anchor=center, inner sep=0}, "{\ell'}"{description}, from=3-1, to=2-4]
        \arrow["\ell", from=3-1, to=1-3]
        \arrow["\dashv"{anchor=center}, shift right=1, draw=none, from=1, to=0]
    \end{tikzcd}\]
    The left triangle is a relative adjunction ($\ell \radj{\ell'} r$) if and only if the outer triangle is a relative adjunction ($\ell \jadj r \d r'$).
\end{lemma}

We shall specialise \cref{relative-adjunction-pasting} to the case in which $r'$ is strictly $j$-monadic, and study the algebras for the relative monads induced by the corresponding relative adjunctions.

\begin{lemma}
    \label{pasting-law-algebras}
    Consider the following situation, and let $T$ be the $f_{T'}$-monad induced by $\ell \radj{f_{T'}} r$.
    \[\begin{tikzcd}[sep=small]
        && C \\
        &&& {\Alg(T')} \\
        A &&&& E
        \arrow[""{name=0, anchor=center, inner sep=0}, "r", from=1-3, to=2-4]
        \arrow["j"', from=3-1, to=3-5]
        \arrow[""{name=1, anchor=center, inner sep=0}, "{u_{T'}}", from=2-4, to=3-5]
        \arrow[""{name=2, anchor=center, inner sep=0}, "{f_{T'}}"{description}, from=3-1, to=2-4]
        \arrow[""{name=3, anchor=center, inner sep=0}, "\ell", from=3-1, to=1-3]
        \arrow["\dashv"{anchor=center, pos=.6}, draw=none, from=2, to=1]
        \arrow["\dashv"{anchor=center, rotate=19}, shift right=2, draw=none, from=3, to=0]
    \end{tikzcd}\]
    Postcomposition by $u_{T'}$ induces a bijection between $T$-algebras and $(T \d u_{T'})$-algebras, and their graded morphisms.
\end{lemma}

\begin{proof}
  We denote the carrier of the $f_{T'}$-monad $T$ by $t = (\ell \d r)$, the unit by $\eta \colon f_{T'} \tto t$, and the extension operator by $\dagger \colon \Alg(T')(f_{T'}, t) \tto \Alg(T')(t, t)$;
  and denote the transposition operators of the relative adjunction $f_{T'} \jadj u_{T'}$ by
  $
    \sharp' \colon \Alg(T')(f_{T'}, 1) \iso E(j, r) \cocolon \flat'
  $.
  Recall that every $T$-algebra (morphism) induces a $(T \d u_{T'})$-algebra (morphism) by postcomposing $u_{T'}$ (\cite[Proposition~6.27]{arkor2024formal}), and that every $(T \d u_{T'})$-algebra (morphism) induces a $T'$-algebra (morphism) via the left-morphism $(r, \eta)$.

  Given a $(T \d u_{T'})$-algebra $(e, \aop')$, we shall construct a $T$-algebra. The carrier $x \colon X \to \Alg(T')$ is induced by the universal property of $\Alg(T')$ from the $T'$-algebra induced by $(r, \eta)$. For the algebra structure, observe that the 2-cell $\aop' \colon E(j, e) \tto E(u_{T'} t, e)$ forms an $(E(j, e))$-graded $T'$-algebra morphism in the sense of \cref{alternative-graded-morphism}, from $(T \d u_{T'})$ to the induced $T'$-algebra on $e$, following \cite[Example~6.31]{arkor2024formal}. Hence, the universal property of $\Alg(T')$ induces a 2-cell $\unit_{\aop'} \colon E(j, e) \tto \Alg(T')(t, x)$ such that postcomposition by $u_{T'}$ yields $\aop'$. Let $\aop$ be as follows.
    \[\aop \colon \Alg(T')(f_{T'}, x) \xtto{\sharp'(1, x)} E(j, u_{T'} x) = E(j, e) \xtto{\unit_{\aop'}} \Alg(T')(t, x)\]
    The pair $(x, \aop)$ forms a $T$-algebra. The unit law follows: \eqstepref{1.1} by the $\sharp'$--$\flat'$ isomorphism; \eqstepref{1.2} by bending $u_{T'}$; \eqstepref{1.3} by the definition of $\unit_{\aop'}$; \eqstepref{1.4} by the unit law for $\aop'$; and \eqstepref{1.5} by the $\sharp'$--$\flat'$ isomorphism. To show the extension law, it suffices to show the equality of the corresponding 2-cells under postcomposition by $u_{T'}$, using the universal property of $\Alg(T')$. This equality follows: \eqstepref{2.1} by the definition of $\unit_{\aop'}$; \eqstepref{2.2} by the $\sharp'$--$\flat'$ isomorphism; \eqstepref{2.3} by the extension operator law for $\aop'$; and \eqstepref{2.4} by the definition of $\unit_{\aop'}$ twice.
    
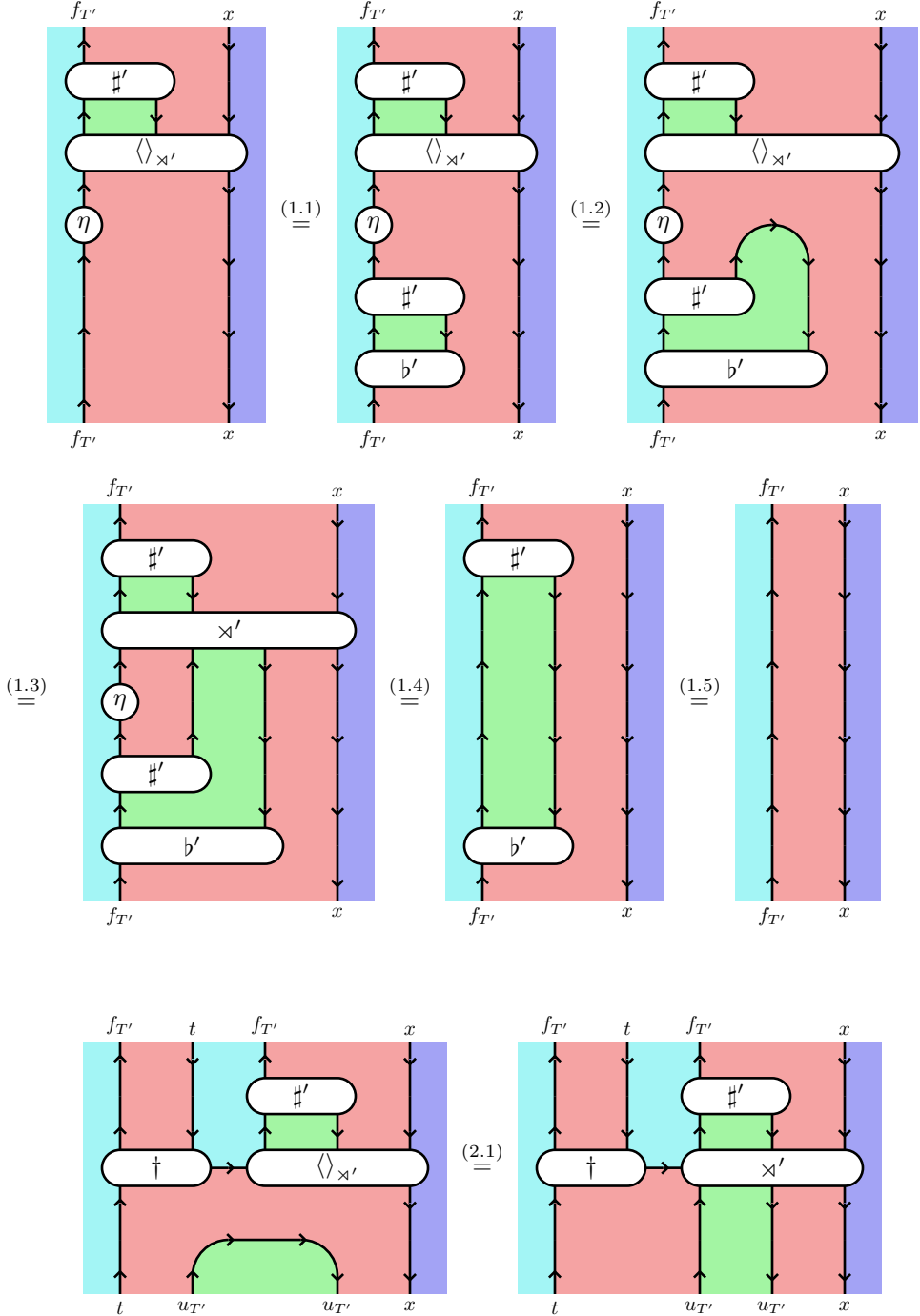
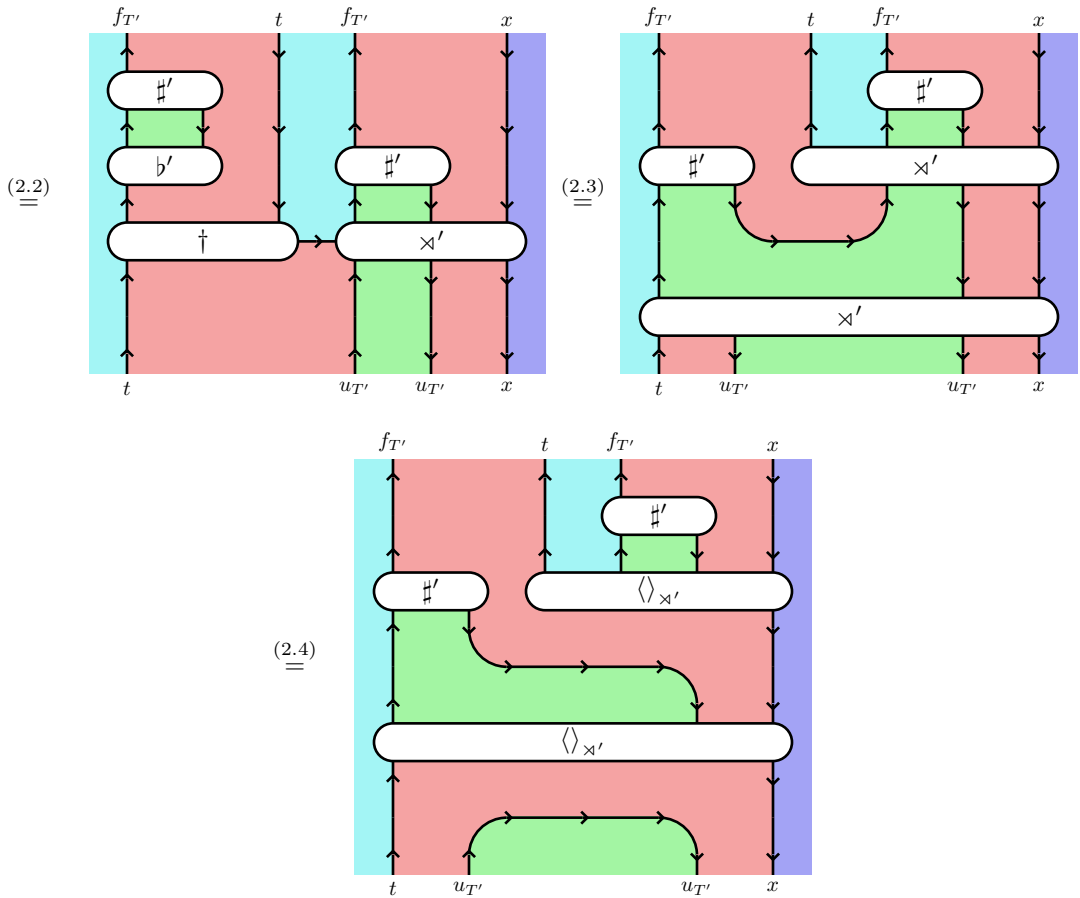
\begin{tangleeqs}
    \begin{tangle}{(3,7)}[trim y]
        \tgBorderA{(0,0)}{\tgColour6}{\tgColour0}{\tgColour0}{\tgColour6}
        \tgBlank{(1,0)}{\tgColour0}
        \tgBorderA{(2,0)}{\tgColour0}{\tgColour8}{\tgColour8}{\tgColour0}
        \tgBorderA{(0,1)}{\tgColour6}{\tgColour0}{\tgColour4}{\tgColour6}
        \tgBorderA{(1,1)}{\tgColour0}{\tgColour0}{\tgColour0}{\tgColour4}
        \tgBorderA{(2,1)}{\tgColour0}{\tgColour8}{\tgColour8}{\tgColour0}
        \tgBorderA{(0,2)}{\tgColour6}{\tgColour4}{\tgColour0}{\tgColour6}
        \tgBorderA{(1,2)}{\tgColour4}{\tgColour0}{\tgColour0}{\tgColour0}
        \tgBorder{(1,2)}{0}{1}{0}{0}
        \tgBorderA{(2,2)}{\tgColour0}{\tgColour8}{\tgColour8}{\tgColour0}
        \tgBorder{(2,2)}{0}{0}{0}{1}
        \tgBorderA{(0,3)}{\tgColour6}{\tgColour0}{\tgColour0}{\tgColour6}
        \tgBlank{(1,3)}{\tgColour0}
        \tgBorderA{(2,3)}{\tgColour0}{\tgColour8}{\tgColour8}{\tgColour0}
        \tgBorderA{(0,4)}{\tgColour6}{\tgColour0}{\tgColour0}{\tgColour6}
        \tgBlank{(1,4)}{\tgColour0}
        \tgBorderA{(2,4)}{\tgColour0}{\tgColour8}{\tgColour8}{\tgColour0}
        \tgBorderA{(0,5)}{\tgColour6}{\tgColour0}{\tgColour0}{\tgColour6}
        \tgBlank{(1,5)}{\tgColour0}
        \tgBorderA{(2,5)}{\tgColour0}{\tgColour8}{\tgColour8}{\tgColour0}
        \tgBorderA{(0,6)}{\tgColour6}{\tgColour0}{\tgColour0}{\tgColour6}
        \tgBlank{(1,6)}{\tgColour0}
        \tgBorderA{(2,6)}{\tgColour0}{\tgColour8}{\tgColour8}{\tgColour0}
        \tgCell[(1,0)]{(0.5,1)}{\sharp'}
        \tgCell[(2,0)]{(1,2)}{\unit_{\aop'}}
        \tgCell{(0,3)}{\eta}
        \tgArrow{(0,3.5)}{1}
        \tgArrow{(2,0.5)}{3}
        \tgArrow{(2,1.5)}{3}
        \tgArrow{(2,2.5)}{3}
        \tgArrow{(2,3.5)}{3}
        \tgArrow{(1,1.5)}{3}
        \tgArrow{(0,0.5)}{1}
        \tgArrow{(0,1.5)}{1}
        \tgArrow{(0,2.5)}{1}
        \tgArrow{(2,4.5)}{3}
        \tgArrow{(2,5.5)}{3}
        \tgArrow{(0,4.5)}{1}
        \tgArrow{(0,5.5)}{1}
        \tgAxisLabel{(0.5,0.75)}{south}{f_{T'}}
        \tgAxisLabel{(2.5,0.75)}{south}{x}
        \tgAxisLabel{(0.5,6.25)}{north}{f_{T'}}
        \tgAxisLabel{(2.5,6.25)}{north}{x}
    \end{tangle}
    \=
    \begin{tangle}{(3,7)}[trim y]
        \tgBorderA{(0,0)}{\tgColour6}{\tgColour0}{\tgColour0}{\tgColour6}
        \tgBlank{(1,0)}{\tgColour0}
        \tgBorderA{(2,0)}{\tgColour0}{\tgColour8}{\tgColour8}{\tgColour0}
        \tgBorderA{(0,1)}{\tgColour6}{\tgColour0}{\tgColour4}{\tgColour6}
        \tgBorderA{(1,1)}{\tgColour0}{\tgColour0}{\tgColour0}{\tgColour4}
        \tgBorderA{(2,1)}{\tgColour0}{\tgColour8}{\tgColour8}{\tgColour0}
        \tgBorderA{(0,2)}{\tgColour6}{\tgColour4}{\tgColour0}{\tgColour6}
        \tgBorderA{(1,2)}{\tgColour4}{\tgColour0}{\tgColour0}{\tgColour0}
        \tgBorder{(1,2)}{0}{1}{0}{0}
        \tgBorderA{(2,2)}{\tgColour0}{\tgColour8}{\tgColour8}{\tgColour0}
        \tgBorder{(2,2)}{0}{0}{0}{1}
        \tgBorderA{(0,3)}{\tgColour6}{\tgColour0}{\tgColour0}{\tgColour6}
        \tgBlank{(1,3)}{\tgColour0}
        \tgBorderA{(2,3)}{\tgColour0}{\tgColour8}{\tgColour8}{\tgColour0}
        \tgBorderA{(0,4)}{\tgColour6}{\tgColour0}{\tgColour4}{\tgColour6}
        \tgBorderA{(1,4)}{\tgColour0}{\tgColour0}{\tgColour0}{\tgColour4}
        \tgBorderA{(2,4)}{\tgColour0}{\tgColour8}{\tgColour8}{\tgColour0}
        \tgBorderA{(0,5)}{\tgColour6}{\tgColour4}{\tgColour0}{\tgColour6}
        \tgBorderA{(1,5)}{\tgColour4}{\tgColour0}{\tgColour0}{\tgColour0}
        \tgBorderA{(2,5)}{\tgColour0}{\tgColour8}{\tgColour8}{\tgColour0}
        \tgBorderA{(0,6)}{\tgColour6}{\tgColour0}{\tgColour0}{\tgColour6}
        \tgBlank{(1,6)}{\tgColour0}
        \tgBorderA{(2,6)}{\tgColour0}{\tgColour8}{\tgColour8}{\tgColour0}
        \tgCell[(1,0)]{(0.5,1)}{\sharp'}
        \tgCell[(2,0)]{(1,2)}{\unit_{\aop'}}
        \tgCell{(0,3)}{\eta}
        \tgArrow{(0,3.5)}{1}
        \tgArrow{(2,0.5)}{3}
        \tgArrow{(2,1.5)}{3}
        \tgArrow{(2,2.5)}{3}
        \tgArrow{(2,3.5)}{3}
        \tgArrow{(1,1.5)}{3}
        \tgArrow{(0,0.5)}{1}
        \tgArrow{(0,1.5)}{1}
        \tgArrow{(0,2.5)}{1}
        \tgCell[(1,0)]{(0.5,4)}{\sharp'}
        \tgCell[(1,0)]{(0.5,5)}{\flat'}
        \tgArrow{(1,4.5)}{3}
        \tgArrow{(0,4.5)}{1}
        \tgArrow{(0,5.5)}{1}
        \tgArrow{(2,4.5)}{3}
        \tgArrow{(2,5.5)}{3}
        \tgAxisLabel{(0.5,0.75)}{south}{f_{T'}}
        \tgAxisLabel{(2.5,0.75)}{south}{x}
        \tgAxisLabel{(0.5,6.25)}{north}{f_{T'}}
        \tgAxisLabel{(2.5,6.25)}{north}{x}
    \end{tangle}
    \=
    \begin{tangle}{(4,7)}[trim y]
        \tgBorderA{(0,0)}{\tgColour6}{\tgColour0}{\tgColour0}{\tgColour6}
        \tgBlank{(1,0)}{\tgColour0}
        \tgBlank{(2,0)}{\tgColour0}
        \tgBorderA{(3,0)}{\tgColour0}{\tgColour8}{\tgColour8}{\tgColour0}
        \tgBorderA{(0,1)}{\tgColour6}{\tgColour0}{\tgColour4}{\tgColour6}
        \tgBorderA{(1,1)}{\tgColour0}{\tgColour0}{\tgColour0}{\tgColour4}
        \tgBlank{(2,1)}{\tgColour0}
        \tgBorderA{(3,1)}{\tgColour0}{\tgColour8}{\tgColour8}{\tgColour0}
        \tgBorderA{(0,2)}{\tgColour6}{\tgColour4}{\tgColour0}{\tgColour6}
        \tgBorderA{(1,2)}{\tgColour4}{\tgColour0}{\tgColour0}{\tgColour0}
        \tgBorder{(1,2)}{0}{1}{0}{0}
        \tgBorderA{(2,2)}{\tgColour0}{\tgColour0}{\tgColour0}{\tgColour0}
        \tgBorder{(2,2)}{0}{1}{0}{1}
        \tgBorderA{(3,2)}{\tgColour0}{\tgColour8}{\tgColour8}{\tgColour0}
        \tgBorder{(3,2)}{0}{0}{0}{1}
        \tgBorderA{(0,3)}{\tgColour6}{\tgColour0}{\tgColour0}{\tgColour6}
        \tgBorderC{(1,3)}{3}{\tgColour0}{\tgColour4}
        \tgBorderC{(2,3)}{2}{\tgColour0}{\tgColour4}
        \tgBorderA{(3,3)}{\tgColour0}{\tgColour8}{\tgColour8}{\tgColour0}
        \tgBorderA{(0,4)}{\tgColour6}{\tgColour0}{\tgColour4}{\tgColour6}
        \tgBorderA{(1,4)}{\tgColour0}{\tgColour4}{\tgColour4}{\tgColour4}
        \tgBorderA{(2,4)}{\tgColour4}{\tgColour0}{\tgColour0}{\tgColour4}
        \tgBorderA{(3,4)}{\tgColour0}{\tgColour8}{\tgColour8}{\tgColour0}
        \tgBorderA{(0,5)}{\tgColour6}{\tgColour4}{\tgColour0}{\tgColour6}
        \tgBorderA{(1,5)}{\tgColour4}{\tgColour4}{\tgColour0}{\tgColour0}
        \tgBorderA{(2,5)}{\tgColour4}{\tgColour0}{\tgColour0}{\tgColour0}
        \tgBorderA{(3,5)}{\tgColour0}{\tgColour8}{\tgColour8}{\tgColour0}
        \tgBorderA{(0,6)}{\tgColour6}{\tgColour0}{\tgColour0}{\tgColour6}
        \tgBlank{(1,6)}{\tgColour0}
        \tgBlank{(2,6)}{\tgColour0}
        \tgBorderA{(3,6)}{\tgColour0}{\tgColour8}{\tgColour8}{\tgColour0}
        \tgCell{(0,3)}{\eta}
        \tgCell[(2,0)]{(1,5)}{\flat'}
        \tgCell[(1,0)]{(0.5,4)}{\sharp'}
        \tgCell[(3,0)]{(1.5,2)}{\unit_{\aop'}}
        \tgCell[(1,0)]{(0.5,1)}{\sharp'}
        \tgArrow{(3,0.5)}{3}
        \tgArrow{(0,0.5)}{1}
        \tgArrow{(0,1.5)}{1}
        \tgArrow{(0,2.5)}{1}
        \tgArrow{(0,3.5)}{1}
        \tgArrow{(0,4.5)}{1}
        \tgArrow{(0,5.5)}{1}
        \tgArrow{(1,3.5)}{1}
        \tgArrow{(1.5,3)}{0}
        \tgArrow{(2,3.5)}{3}
        \tgArrow{(2,4.5)}{3}
        \tgArrow{(3,2.5)}{3}
        \tgArrow{(3,3.5)}{3}
        \tgArrow{(3,4.5)}{3}
        \tgArrow{(3,5.5)}{3}
        \tgArrow{(3,1.5)}{3}
        \tgArrow{(1,1.5)}{3}
        \tgAxisLabel{(0.5,0.75)}{south}{f_{T'}}
        \tgAxisLabel{(3.5,0.75)}{south}{x}
        \tgAxisLabel{(0.5,6.25)}{north}{f_{T'}}
        \tgAxisLabel{(3.5,6.25)}{north}{x}
    \end{tangle}
    \\
    \begin{tangle}{(4,7)}[trim y]
        \tgBorderA{(0,0)}{\tgColour6}{\tgColour0}{\tgColour0}{\tgColour6}
        \tgBlank{(1,0)}{\tgColour0}
        \tgBlank{(2,0)}{\tgColour0}
        \tgBorderA{(3,0)}{\tgColour0}{\tgColour8}{\tgColour8}{\tgColour0}
        \tgBorderA{(0,1)}{\tgColour6}{\tgColour0}{\tgColour4}{\tgColour6}
        \tgBorderA{(1,1)}{\tgColour0}{\tgColour0}{\tgColour0}{\tgColour4}
        \tgBlank{(2,1)}{\tgColour0}
        \tgBorderA{(3,1)}{\tgColour0}{\tgColour8}{\tgColour8}{\tgColour0}
        \tgBorderA{(0,2)}{\tgColour6}{\tgColour4}{\tgColour0}{\tgColour6}
        \tgBorderA{(1,2)}{\tgColour4}{\tgColour0}{\tgColour4}{\tgColour0}
        \tgBorderA{(2,2)}{\tgColour0}{\tgColour0}{\tgColour0}{\tgColour4}
        \tgBorder{(2,2)}{0}{1}{0}{0}
        \tgBorderA{(3,2)}{\tgColour0}{\tgColour8}{\tgColour8}{\tgColour0}
        \tgBorder{(3,2)}{0}{0}{0}{1}
        \tgBorderA{(0,3)}{\tgColour6}{\tgColour0}{\tgColour0}{\tgColour6}
        \tgBorderA{(1,3)}{\tgColour0}{\tgColour4}{\tgColour4}{\tgColour0}
        \tgBorderA{(2,3)}{\tgColour4}{\tgColour0}{\tgColour0}{\tgColour4}
        \tgBorderA{(3,3)}{\tgColour0}{\tgColour8}{\tgColour8}{\tgColour0}
        \tgBorderA{(0,4)}{\tgColour6}{\tgColour0}{\tgColour4}{\tgColour6}
        \tgBorderA{(1,4)}{\tgColour0}{\tgColour4}{\tgColour4}{\tgColour4}
        \tgBorderA{(2,4)}{\tgColour4}{\tgColour0}{\tgColour0}{\tgColour4}
        \tgBorderA{(3,4)}{\tgColour0}{\tgColour8}{\tgColour8}{\tgColour0}
        \tgBorderA{(0,5)}{\tgColour6}{\tgColour4}{\tgColour0}{\tgColour6}
        \tgBorderA{(1,5)}{\tgColour4}{\tgColour4}{\tgColour0}{\tgColour0}
        \tgBorderA{(2,5)}{\tgColour4}{\tgColour0}{\tgColour0}{\tgColour0}
        \tgBorderA{(3,5)}{\tgColour0}{\tgColour8}{\tgColour8}{\tgColour0}
        \tgBorderA{(0,6)}{\tgColour6}{\tgColour0}{\tgColour0}{\tgColour6}
        \tgBlank{(1,6)}{\tgColour0}
        \tgBlank{(2,6)}{\tgColour0}
        \tgBorderA{(3,6)}{\tgColour0}{\tgColour8}{\tgColour8}{\tgColour0}
        \tgCell{(0,3)}{\eta}
        \tgCell[(2,0)]{(1,5)}{\flat'}
        \tgCell[(1,0)]{(0.5,4)}{\sharp'}
        \tgCell[(3,0)]{(1.5,2)}{\aop'}
        \tgCell[(1,0)]{(0.5,1)}{\sharp'}
        \tgArrow{(3,0.5)}{3}
        \tgArrow{(0,0.5)}{1}
        \tgArrow{(0,1.5)}{1}
        \tgArrow{(0,2.5)}{1}
        \tgArrow{(0,3.5)}{1}
        \tgArrow{(0,4.5)}{1}
        \tgArrow{(0,5.5)}{1}
        \tgArrow{(1,3.5)}{1}
        \tgArrow{(2,3.5)}{3}
        \tgArrow{(2,4.5)}{3}
        \tgArrow{(3,2.5)}{3}
        \tgArrow{(3,3.5)}{3}
        \tgArrow{(3,4.5)}{3}
        \tgArrow{(3,5.5)}{3}
        \tgArrow{(3,1.5)}{3}
        \tgArrow{(1,1.5)}{3}
        \tgArrow{(1,2.5)}{1}
        \tgArrow{(2,2.5)}{3}
        \tgAxisLabel{(0.5,0.75)}{south}{f_{T'}}
        \tgAxisLabel{(3.5,0.75)}{south}{x}
        \tgAxisLabel{(0.5,6.25)}{north}{f_{T'}}
        \tgAxisLabel{(3.5,6.25)}{north}{x}
    \end{tangle}
    \=
    \begin{tangle}{(3,7)}[trim y]
        \tgBorderA{(0,0)}{\tgColour6}{\tgColour0}{\tgColour0}{\tgColour6}
        \tgBlank{(1,0)}{\tgColour0}
        \tgBorderA{(2,0)}{\tgColour0}{\tgColour8}{\tgColour8}{\tgColour0}
        \tgBorderA{(0,1)}{\tgColour6}{\tgColour0}{\tgColour4}{\tgColour6}
        \tgBorderA{(1,1)}{\tgColour0}{\tgColour0}{\tgColour0}{\tgColour4}
        \tgBorderA{(2,1)}{\tgColour0}{\tgColour8}{\tgColour8}{\tgColour0}
        \tgBorderA{(0,2)}{\tgColour6}{\tgColour4}{\tgColour4}{\tgColour6}
        \tgBorderA{(1,2)}{\tgColour4}{\tgColour0}{\tgColour0}{\tgColour4}
        \tgBorderA{(2,2)}{\tgColour0}{\tgColour8}{\tgColour8}{\tgColour0}
        \tgBorderA{(0,3)}{\tgColour6}{\tgColour4}{\tgColour4}{\tgColour6}
        \tgBorderA{(1,3)}{\tgColour4}{\tgColour0}{\tgColour0}{\tgColour4}
        \tgBorderA{(2,3)}{\tgColour0}{\tgColour8}{\tgColour8}{\tgColour0}
        \tgBorderA{(0,4)}{\tgColour6}{\tgColour4}{\tgColour4}{\tgColour6}
        \tgBorderA{(1,4)}{\tgColour4}{\tgColour0}{\tgColour0}{\tgColour4}
        \tgBorderA{(2,4)}{\tgColour0}{\tgColour8}{\tgColour8}{\tgColour0}
        \tgBorderA{(0,5)}{\tgColour6}{\tgColour4}{\tgColour0}{\tgColour6}
        \tgBorderA{(1,5)}{\tgColour4}{\tgColour0}{\tgColour0}{\tgColour0}
        \tgBorderA{(2,5)}{\tgColour0}{\tgColour8}{\tgColour8}{\tgColour0}
        \tgBorderA{(0,6)}{\tgColour6}{\tgColour0}{\tgColour0}{\tgColour6}
        \tgBlank{(1,6)}{\tgColour0}
        \tgBorderA{(2,6)}{\tgColour0}{\tgColour8}{\tgColour8}{\tgColour0}
        \tgCell[(1,0)]{(0.5,1)}{\sharp'}
        \tgArrow{(0,0.5)}{1}
        \tgArrow{(0,1.5)}{1}
        \tgArrow{(0,2.5)}{1}
        \tgArrow{(0,3.5)}{1}
        \tgArrow{(0,4.5)}{1}
        \tgArrow{(0,5.5)}{1}
        \tgArrow{(1,3.5)}{3}
        \tgArrow{(1,1.5)}{3}
        \tgArrow{(1,2.5)}{3}
        \tgCell[(1,0)]{(0.5,5)}{\flat'}
        \tgArrow{(1,4.5)}{3}
        \tgArrow{(2,0.5)}{3}
        \tgArrow{(2,1.5)}{3}
        \tgArrow{(2,2.5)}{3}
        \tgArrow{(2,3.5)}{3}
        \tgArrow{(2,4.5)}{3}
        \tgArrow{(2,5.5)}{3}
        \tgAxisLabel{(0.5,0.75)}{south}{f_{T'}}
        \tgAxisLabel{(2.5,0.75)}{south}{x}
        \tgAxisLabel{(0.5,6.25)}{north}{f_{T'}}
        \tgAxisLabel{(2.5,6.25)}{north}{x}
    \end{tangle}
    \=
    \begin{tangle}{(2,7)}[trim y]
        \tgBorderA{(0,0)}{\tgColour6}{\tgColour0}{\tgColour0}{\tgColour6}
        \tgBorderA{(1,0)}{\tgColour0}{\tgColour8}{\tgColour8}{\tgColour0}
        \tgBorderA{(0,1)}{\tgColour6}{\tgColour0}{\tgColour0}{\tgColour6}
        \tgBorderA{(1,1)}{\tgColour0}{\tgColour8}{\tgColour8}{\tgColour0}
        \tgBorderA{(0,2)}{\tgColour6}{\tgColour0}{\tgColour0}{\tgColour6}
        \tgBorderA{(1,2)}{\tgColour0}{\tgColour8}{\tgColour8}{\tgColour0}
        \tgBorderA{(0,3)}{\tgColour6}{\tgColour0}{\tgColour0}{\tgColour6}
        \tgBorderA{(1,3)}{\tgColour0}{\tgColour8}{\tgColour8}{\tgColour0}
        \tgBorderA{(0,4)}{\tgColour6}{\tgColour0}{\tgColour0}{\tgColour6}
        \tgBorderA{(1,4)}{\tgColour0}{\tgColour8}{\tgColour8}{\tgColour0}
        \tgBorderA{(0,5)}{\tgColour6}{\tgColour0}{\tgColour0}{\tgColour6}
        \tgBorderA{(1,5)}{\tgColour0}{\tgColour8}{\tgColour8}{\tgColour0}
        \tgBorderA{(0,6)}{\tgColour6}{\tgColour0}{\tgColour0}{\tgColour6}
        \tgBorderA{(1,6)}{\tgColour0}{\tgColour8}{\tgColour8}{\tgColour0}
        \tgArrow{(0,0.5)}{1}
        \tgArrow{(0,1.5)}{1}
        \tgArrow{(1,0.5)}{3}
        \tgArrow{(1,1.5)}{3}
        \tgArrow{(1,2.5)}{3}
        \tgArrow{(1,3.5)}{3}
        \tgArrow{(1,4.5)}{3}
        \tgArrow{(1,5.5)}{3}
        \tgArrow{(0,5.5)}{1}
        \tgArrow{(0,4.5)}{1}
        \tgArrow{(0,3.5)}{1}
        \tgArrow{(0,2.5)}{1}
        \tgAxisLabel{(0.5,0.75)}{south}{f_{T'}}
        \tgAxisLabel{(1.5,0.75)}{south}{x}
        \tgAxisLabel{(0.5,6.25)}{north}{f_{T'}}
        \tgAxisLabel{(1.5,6.25)}{north}{x}
    \end{tangle}
    \end{tangleeqs}
    \begin{tangleeqs}
    \begin{tangle}{(5,5)}[trim y]
        \tgBorderA{(0,0)}{\tgColour6}{\tgColour0}{\tgColour0}{\tgColour6}
        \tgBorderA{(1,0)}{\tgColour0}{\tgColour6}{\tgColour6}{\tgColour0}
        \tgBorderA{(2,0)}{\tgColour6}{\tgColour0}{\tgColour0}{\tgColour6}
        \tgBlank{(3,0)}{\tgColour0}
        \tgBorderA{(4,0)}{\tgColour0}{\tgColour8}{\tgColour8}{\tgColour0}
        \tgBorderA{(0,1)}{\tgColour6}{\tgColour0}{\tgColour0}{\tgColour6}
        \tgBorderA{(1,1)}{\tgColour0}{\tgColour6}{\tgColour6}{\tgColour0}
        \tgBorderA{(2,1)}{\tgColour6}{\tgColour0}{\tgColour4}{\tgColour6}
        \tgBorderA{(3,1)}{\tgColour0}{\tgColour0}{\tgColour0}{\tgColour4}
        \tgBorderA{(4,1)}{\tgColour0}{\tgColour8}{\tgColour8}{\tgColour0}
        \tgBorderA{(0,2)}{\tgColour6}{\tgColour0}{\tgColour0}{\tgColour6}
        \tgBorder{(0,2)}{0}{1}{0}{0}
        \tgBorderA{(1,2)}{\tgColour0}{\tgColour6}{\tgColour0}{\tgColour0}
        \tgBorder{(1,2)}{0}{0}{0}{1}
        \tgBorderA{(2,2)}{\tgColour6}{\tgColour4}{\tgColour0}{\tgColour0}
        \tgBorderA{(3,2)}{\tgColour4}{\tgColour0}{\tgColour0}{\tgColour0}
        \tgBorder{(3,2)}{0}{1}{0}{0}
        \tgBorderA{(4,2)}{\tgColour0}{\tgColour8}{\tgColour8}{\tgColour0}
        \tgBorder{(4,2)}{0}{0}{0}{1}
        \tgBorderA{(0,3)}{\tgColour6}{\tgColour0}{\tgColour0}{\tgColour6}
        \tgBorderC{(1,3)}{3}{\tgColour0}{\tgColour4}
        \tgBorderA{(2,3)}{\tgColour0}{\tgColour0}{\tgColour4}{\tgColour4}
        \tgBorderC{(3,3)}{2}{\tgColour0}{\tgColour4}
        \tgBorderA{(4,3)}{\tgColour0}{\tgColour8}{\tgColour8}{\tgColour0}
        \tgBorderA{(0,4)}{\tgColour6}{\tgColour0}{\tgColour0}{\tgColour6}
        \tgBorderA{(1,4)}{\tgColour0}{\tgColour4}{\tgColour4}{\tgColour0}
        \tgBlank{(2,4)}{\tgColour4}
        \tgBorderA{(3,4)}{\tgColour4}{\tgColour0}{\tgColour0}{\tgColour4}
        \tgBorderA{(4,4)}{\tgColour0}{\tgColour8}{\tgColour8}{\tgColour0}
        \tgCell[(1,0)]{(0.5,2)}{\dag}
        \tgCell[(2,0)]{(3,2)}{\unit_{\aop'}}
        \tgCell[(1,0)]{(2.5,1)}{\sharp'}
        \tgArrow{(0,2.5)}{1}
        \tgArrow{(1,1.5)}{3}
        \tgArrow{(1,0.5)}{3}
        \tgArrow{(3,1.5)}{3}
        \tgArrow{(4,1.5)}{3}
        \tgArrow{(4,0.5)}{3}
        \tgArrow{(4,2.5)}{3}
        \tgArrow{(0,1.5)}{1}
        \tgArrow{(0,0.5)}{1}
        \tgArrow{(1.5,2)}{0}
        \tgArrow{(2,1.5)}{1}
        \tgArrow{(2,0.5)}{1}
        \tgArrow{(1.5,3)}{0}
        \tgArrow{(2.5,3)}{0}
        \tgArrow{(4,3.5)}{3}
        \tgArrow{(3,3.5)}{3}
        \tgArrow{(0,3.5)}{1}
        \tgArrow{(1,3.5)}{1}
        \tgAxisLabel{(0.5,0.75)}{south}{f_{T'}}
        \tgAxisLabel{(1.5,0.75)}{south}{t}
        \tgAxisLabel{(2.5,0.75)}{south}{f_{T'}}
        \tgAxisLabel{(4.5,0.75)}{south}{x}
        \tgAxisLabel{(0.5,4.25)}{north}{t}
        \tgAxisLabel{(1.5,4.25)}{north}{u_{T'}}
        \tgAxisLabel{(3.5,4.25)}{north}{u_{T'}}
        \tgAxisLabel{(4.5,4.25)}{north}{x}
    \end{tangle}
    \=
    \begin{tangle}{(5,5)}[trim y]
        \tgBorderA{(0,0)}{\tgColour6}{\tgColour0}{\tgColour0}{\tgColour6}
        \tgBorderA{(1,0)}{\tgColour0}{\tgColour6}{\tgColour6}{\tgColour0}
        \tgBorderA{(2,0)}{\tgColour6}{\tgColour0}{\tgColour0}{\tgColour6}
        \tgBlank{(3,0)}{\tgColour0}
        \tgBorderA{(4,0)}{\tgColour0}{\tgColour8}{\tgColour8}{\tgColour0}
        \tgBorderA{(0,1)}{\tgColour6}{\tgColour0}{\tgColour0}{\tgColour6}
        \tgBorderA{(1,1)}{\tgColour0}{\tgColour6}{\tgColour6}{\tgColour0}
        \tgBorderA{(2,1)}{\tgColour6}{\tgColour0}{\tgColour4}{\tgColour6}
        \tgBorderA{(3,1)}{\tgColour0}{\tgColour0}{\tgColour0}{\tgColour4}
        \tgBorderA{(4,1)}{\tgColour0}{\tgColour8}{\tgColour8}{\tgColour0}
        \tgBorderA{(0,2)}{\tgColour6}{\tgColour0}{\tgColour0}{\tgColour6}
        \tgBorder{(0,2)}{0}{1}{0}{0}
        \tgBorderA{(1,2)}{\tgColour0}{\tgColour6}{\tgColour0}{\tgColour0}
        \tgBorder{(1,2)}{0}{0}{0}{1}
        \tgBorderA{(2,2)}{\tgColour6}{\tgColour4}{\tgColour4}{\tgColour0}
        \tgBorder{(2,2)}{0}{1}{0}{0}
        \tgBorderA{(3,2)}{\tgColour4}{\tgColour0}{\tgColour0}{\tgColour4}
        \tgBorder{(3,2)}{0}{1}{0}{1}
        \tgBorderA{(4,2)}{\tgColour0}{\tgColour8}{\tgColour8}{\tgColour0}
        \tgBorder{(4,2)}{0}{0}{0}{1}
        \tgBorderA{(0,3)}{\tgColour6}{\tgColour0}{\tgColour0}{\tgColour6}
        \tgBlank{(1,3)}{\tgColour0}
        \tgBorderA{(2,3)}{\tgColour0}{\tgColour4}{\tgColour4}{\tgColour0}
        \tgBorderA{(3,3)}{\tgColour4}{\tgColour0}{\tgColour0}{\tgColour4}
        \tgBorderA{(4,3)}{\tgColour0}{\tgColour8}{\tgColour8}{\tgColour0}
        \tgBorderA{(0,4)}{\tgColour6}{\tgColour0}{\tgColour0}{\tgColour6}
        \tgBlank{(1,4)}{\tgColour0}
        \tgBorderA{(2,4)}{\tgColour0}{\tgColour4}{\tgColour4}{\tgColour0}
        \tgBorderA{(3,4)}{\tgColour4}{\tgColour0}{\tgColour0}{\tgColour4}
        \tgBorderA{(4,4)}{\tgColour0}{\tgColour8}{\tgColour8}{\tgColour0}
        \tgCell[(1,0)]{(0.5,2)}{\dag}
        \tgCell[(2,0)]{(3,2)}{\aop'}
        \tgCell[(1,0)]{(2.5,1)}{\sharp'}
        \tgArrow{(0,2.5)}{1}
        \tgArrow{(1,1.5)}{3}
        \tgArrow{(1,0.5)}{3}
        \tgArrow{(3,1.5)}{3}
        \tgArrow{(4,1.5)}{3}
        \tgArrow{(4,0.5)}{3}
        \tgArrow{(4,2.5)}{3}
        \tgArrow{(0,1.5)}{1}
        \tgArrow{(0,0.5)}{1}
        \tgArrow{(1.5,2)}{0}
        \tgArrow{(2,1.5)}{1}
        \tgArrow{(2,0.5)}{1}
        \tgArrow{(4,3.5)}{3}
        \tgArrow{(3,3.5)}{3}
        \tgArrow{(0,3.5)}{1}
        \tgArrow{(3,2.5)}{3}
        \tgArrow{(2,3.5)}{1}
        \tgArrow{(2,2.5)}{1}
        \tgAxisLabel{(0.5,0.75)}{south}{f_{T'}}
        \tgAxisLabel{(1.5,0.75)}{south}{t}
        \tgAxisLabel{(2.5,0.75)}{south}{f_{T'}}
        \tgAxisLabel{(4.5,0.75)}{south}{x}
        \tgAxisLabel{(0.5,4.25)}{north}{t}
        \tgAxisLabel{(2.5,4.25)}{north}{u_{T'}}
        \tgAxisLabel{(3.5,4.25)}{north}{u_{T'}}
        \tgAxisLabel{(4.5,4.25)}{north}{x}
    \end{tangle}
    \\
    \begin{tangle}{(6,6)}[trim y]
        \tgBorderA{(0,0)}{\tgColour6}{\tgColour0}{\tgColour0}{\tgColour6}
        \tgBlank{(1,0)}{\tgColour0}
        \tgBorderA{(2,0)}{\tgColour0}{\tgColour6}{\tgColour6}{\tgColour0}
        \tgBorderA{(3,0)}{\tgColour6}{\tgColour0}{\tgColour0}{\tgColour6}
        \tgBlank{(4,0)}{\tgColour0}
        \tgBorderA{(5,0)}{\tgColour0}{\tgColour8}{\tgColour8}{\tgColour0}
        \tgBorderA{(0,1)}{\tgColour6}{\tgColour0}{\tgColour4}{\tgColour6}
        \tgBorderA{(1,1)}{\tgColour0}{\tgColour0}{\tgColour0}{\tgColour4}
        \tgBorderA{(2,1)}{\tgColour0}{\tgColour6}{\tgColour6}{\tgColour0}
        \tgBorderA{(3,1)}{\tgColour6}{\tgColour0}{\tgColour0}{\tgColour6}
        \tgBlank{(4,1)}{\tgColour0}
        \tgBorderA{(5,1)}{\tgColour0}{\tgColour8}{\tgColour8}{\tgColour0}
        \tgBorderA{(0,2)}{\tgColour6}{\tgColour4}{\tgColour0}{\tgColour6}
        \tgBorderA{(1,2)}{\tgColour4}{\tgColour0}{\tgColour0}{\tgColour0}
        \tgBorderA{(2,2)}{\tgColour0}{\tgColour6}{\tgColour6}{\tgColour0}
        \tgBorderA{(3,2)}{\tgColour6}{\tgColour0}{\tgColour4}{\tgColour6}
        \tgBorderA{(4,2)}{\tgColour0}{\tgColour0}{\tgColour0}{\tgColour4}
        \tgBorderA{(5,2)}{\tgColour0}{\tgColour8}{\tgColour8}{\tgColour0}
        \tgBorderA{(0,3)}{\tgColour6}{\tgColour0}{\tgColour0}{\tgColour6}
        \tgBorder{(0,3)}{0}{1}{0}{0}
        \tgBorderA{(1,3)}{\tgColour0}{\tgColour0}{\tgColour0}{\tgColour0}
        \tgBorder{(1,3)}{0}{1}{0}{1}
        \tgBorderA{(2,3)}{\tgColour0}{\tgColour6}{\tgColour0}{\tgColour0}
        \tgBorder{(2,3)}{0}{0}{0}{1}
        \tgBorderA{(3,3)}{\tgColour6}{\tgColour4}{\tgColour4}{\tgColour0}
        \tgBorder{(3,3)}{0}{1}{0}{0}
        \tgBorderA{(4,3)}{\tgColour4}{\tgColour0}{\tgColour0}{\tgColour4}
        \tgBorder{(4,3)}{0}{1}{0}{1}
        \tgBorderA{(5,3)}{\tgColour0}{\tgColour8}{\tgColour8}{\tgColour0}
        \tgBorder{(5,3)}{0}{0}{0}{1}
        \tgBorderA{(0,4)}{\tgColour6}{\tgColour0}{\tgColour0}{\tgColour6}
        \tgBlank{(1,4)}{\tgColour0}
        \tgBlank{(2,4)}{\tgColour0}
        \tgBorderA{(3,4)}{\tgColour0}{\tgColour4}{\tgColour4}{\tgColour0}
        \tgBorderA{(4,4)}{\tgColour4}{\tgColour0}{\tgColour0}{\tgColour4}
        \tgBorderA{(5,4)}{\tgColour0}{\tgColour8}{\tgColour8}{\tgColour0}
        \tgBorderA{(0,5)}{\tgColour6}{\tgColour0}{\tgColour0}{\tgColour6}
        \tgBlank{(1,5)}{\tgColour0}
        \tgBlank{(2,5)}{\tgColour0}
        \tgBorderA{(3,5)}{\tgColour0}{\tgColour4}{\tgColour4}{\tgColour0}
        \tgBorderA{(4,5)}{\tgColour4}{\tgColour0}{\tgColour0}{\tgColour4}
        \tgBorderA{(5,5)}{\tgColour0}{\tgColour8}{\tgColour8}{\tgColour0}
        \tgCell[(2,0)]{(4,3)}{\aop'}
        \tgCell[(1,0)]{(3.5,2)}{\sharp'}
        \tgArrow{(2,2.5)}{3}
        \tgArrow{(2,1.5)}{3}
        \tgArrow{(4,2.5)}{3}
        \tgArrow{(5,2.5)}{3}
        \tgArrow{(5,1.5)}{3}
        \tgArrow{(5,3.5)}{3}
        \tgArrow{(2.5,3)}{0}
        \tgArrow{(3,2.5)}{1}
        \tgArrow{(3,1.5)}{1}
        \tgArrow{(5,4.5)}{3}
        \tgArrow{(4,4.5)}{3}
        \tgArrow{(4,3.5)}{3}
        \tgArrow{(3,4.5)}{1}
        \tgArrow{(3,3.5)}{1}
        \tgCell[(2,0)]{(1,3)}{\dag}
        \tgCell[(1,0)]{(0.5,2)}{\flat'}
        \tgCell[(1,0)]{(0.5,1)}{\sharp'}
        \tgArrow{(5,0.5)}{3}
        \tgArrow{(2,0.5)}{3}
        \tgArrow{(1,1.5)}{3}
        \tgArrow{(0,0.5)}{1}
        \tgArrow{(0,1.5)}{1}
        \tgArrow{(0,2.5)}{1}
        \tgArrow{(0,3.5)}{1}
        \tgArrow{(0,4.5)}{1}
        \tgArrow{(3,0.5)}{1}
        \tgAxisLabel{(0.5,0.75)}{south}{f_{T'}}
        \tgAxisLabel{(2.5,0.75)}{south}{t}
        \tgAxisLabel{(3.5,0.75)}{south}{f_{T'}}
        \tgAxisLabel{(5.5,0.75)}{south}{x}
        \tgAxisLabel{(0.5,5.25)}{north}{t}
        \tgAxisLabel{(3.5,5.25)}{north}{u_{T'}}
        \tgAxisLabel{(4.5,5.25)}{north}{u_{T'}}
        \tgAxisLabel{(5.5,5.25)}{north}{x}
    \end{tangle}
    \=
    \begin{tangle}{(6,6)}[trim y]
        \tgBorderA{(0,0)}{\tgColour6}{\tgColour0}{\tgColour0}{\tgColour6}
        \tgBlank{(1,0)}{\tgColour0}
        \tgBorderA{(2,0)}{\tgColour0}{\tgColour6}{\tgColour6}{\tgColour0}
        \tgBorderA{(3,0)}{\tgColour6}{\tgColour0}{\tgColour0}{\tgColour6}
        \tgBlank{(4,0)}{\tgColour0}
        \tgBorderA{(5,0)}{\tgColour0}{\tgColour8}{\tgColour8}{\tgColour0}
        \tgBorderA{(0,1)}{\tgColour6}{\tgColour0}{\tgColour0}{\tgColour6}
        \tgBlank{(1,1)}{\tgColour0}
        \tgBorderA{(2,1)}{\tgColour0}{\tgColour6}{\tgColour6}{\tgColour0}
        \tgBorderA{(3,1)}{\tgColour6}{\tgColour0}{\tgColour4}{\tgColour6}
        \tgBorderA{(4,1)}{\tgColour0}{\tgColour0}{\tgColour0}{\tgColour4}
        \tgBorderA{(5,1)}{\tgColour0}{\tgColour8}{\tgColour8}{\tgColour0}
        \tgBorderA{(0,2)}{\tgColour6}{\tgColour0}{\tgColour4}{\tgColour6}
        \tgBorderA{(1,2)}{\tgColour0}{\tgColour0}{\tgColour0}{\tgColour4}
        \tgBorderA{(2,2)}{\tgColour0}{\tgColour6}{\tgColour0}{\tgColour0}
        \tgBorderA{(3,2)}{\tgColour6}{\tgColour4}{\tgColour4}{\tgColour0}
        \tgBorder{(3,2)}{0}{1}{0}{0}
        \tgBorderA{(4,2)}{\tgColour4}{\tgColour0}{\tgColour0}{\tgColour4}
        \tgBorder{(4,2)}{0}{1}{0}{1}
        \tgBorderA{(5,2)}{\tgColour0}{\tgColour8}{\tgColour8}{\tgColour0}
        \tgBorder{(5,2)}{0}{0}{0}{1}
        \tgBorderA{(0,3)}{\tgColour6}{\tgColour4}{\tgColour4}{\tgColour6}
        \tgBorderC{(1,3)}{0}{\tgColour4}{\tgColour0}
        \tgBorderA{(2,3)}{\tgColour0}{\tgColour0}{\tgColour4}{\tgColour4}
        \tgBorderC{(3,3)}{1}{\tgColour4}{\tgColour0}
        \tgBorderA{(4,3)}{\tgColour4}{\tgColour0}{\tgColour0}{\tgColour4}
        \tgBorderA{(5,3)}{\tgColour0}{\tgColour8}{\tgColour8}{\tgColour0}
        \tgBorderA{(0,4)}{\tgColour6}{\tgColour4}{\tgColour0}{\tgColour6}
        \tgBorderA{(1,4)}{\tgColour4}{\tgColour4}{\tgColour4}{\tgColour0}
        \tgBorder{(1,4)}{0}{1}{0}{0}
        \tgBorderA{(2,4)}{\tgColour4}{\tgColour4}{\tgColour4}{\tgColour4}
        \tgBorder{(2,4)}{0}{1}{0}{1}
        \tgBorderA{(3,4)}{\tgColour4}{\tgColour4}{\tgColour4}{\tgColour4}
        \tgBorder{(3,4)}{0}{1}{0}{1}
        \tgBorderA{(4,4)}{\tgColour4}{\tgColour0}{\tgColour0}{\tgColour4}
        \tgBorder{(4,4)}{0}{1}{0}{1}
        \tgBorderA{(5,4)}{\tgColour0}{\tgColour8}{\tgColour8}{\tgColour0}
        \tgBorder{(5,4)}{0}{0}{0}{1}
        \tgBorderA{(0,5)}{\tgColour6}{\tgColour0}{\tgColour0}{\tgColour6}
        \tgBorderA{(1,5)}{\tgColour0}{\tgColour4}{\tgColour4}{\tgColour0}
        \tgBlank{(2,5)}{\tgColour4}
        \tgBlank{(3,5)}{\tgColour4}
        \tgBorderA{(4,5)}{\tgColour4}{\tgColour0}{\tgColour0}{\tgColour4}
        \tgBorderA{(5,5)}{\tgColour0}{\tgColour8}{\tgColour8}{\tgColour0}
        \tgCell[(1,0)]{(3.5,1)}{\sharp'}
        \tgArrow{(4,1.5)}{3}
        \tgArrow{(5,1.5)}{3}
        \tgArrow{(5,2.5)}{3}
        \tgArrow{(3,1.5)}{1}
        \tgArrow{(5,3.5)}{3}
        \tgArrow{(4,3.5)}{3}
        \tgArrow{(4,2.5)}{3}
        \tgArrow{(3,2.5)}{1}
        \tgCell[(3,0)]{(3.5,2)}{\aop'}
        \tgCell[(1,0)]{(0.5,2)}{\sharp'}
        \tgArrow{(5,0.5)}{3}
        \tgArrow{(1,2.5)}{3}
        \tgArrow{(4,4.5)}{3}
        \tgArrow{(5,4.5)}{3}
        \tgArrow{(3,0.5)}{1}
        \tgArrow{(2,0.5)}{1}
        \tgArrow{(2,1.5)}{1}
        \tgArrow{(0,0.5)}{1}
        \tgArrow{(0,1.5)}{1}
        \tgArrow{(1.5,3)}{0}
        \tgArrow{(2.5,3)}{0}
        \tgArrow{(0,2.5)}{1}
        \tgArrow{(0,3.5)}{1}
        \tgArrow{(0,4.5)}{1}
        \tgCell[(5,0)]{(2.5,4)}{\aop'}
        \tgArrow{(1,4.5)}{3}
        \tgAxisLabel{(0.5,0.75)}{south}{f_{T'}}
        \tgAxisLabel{(2.5,0.75)}{south}{t}
        \tgAxisLabel{(3.5,0.75)}{south}{f_{T'}}
        \tgAxisLabel{(5.5,0.75)}{south}{x}
        \tgAxisLabel{(0.5,5.25)}{north}{t}
        \tgAxisLabel{(1.5,5.25)}{north}{u_{T'}}
        \tgAxisLabel{(4.5,5.25)}{north}{u_{T'}}
        \tgAxisLabel{(5.5,5.25)}{north}{x}
    \end{tangle}
    \\
    \begin{tangle}{(6,7)}[trim y]
        \tgBorderA{(0,0)}{\tgColour6}{\tgColour0}{\tgColour0}{\tgColour6}
        \tgBlank{(1,0)}{\tgColour0}
        \tgBorderA{(2,0)}{\tgColour0}{\tgColour6}{\tgColour6}{\tgColour0}
        \tgBorderA{(3,0)}{\tgColour6}{\tgColour0}{\tgColour0}{\tgColour6}
        \tgBlank{(4,0)}{\tgColour0}
        \tgBorderA{(5,0)}{\tgColour0}{\tgColour8}{\tgColour8}{\tgColour0}
        \tgBorderA{(0,1)}{\tgColour6}{\tgColour0}{\tgColour0}{\tgColour6}
        \tgBlank{(1,1)}{\tgColour0}
        \tgBorderA{(2,1)}{\tgColour0}{\tgColour6}{\tgColour6}{\tgColour0}
        \tgBorderA{(3,1)}{\tgColour6}{\tgColour0}{\tgColour4}{\tgColour6}
        \tgBorderA{(4,1)}{\tgColour0}{\tgColour0}{\tgColour0}{\tgColour4}
        \tgBorderA{(5,1)}{\tgColour0}{\tgColour8}{\tgColour8}{\tgColour0}
        \tgBorderA{(0,2)}{\tgColour6}{\tgColour0}{\tgColour4}{\tgColour6}
        \tgBorderA{(1,2)}{\tgColour0}{\tgColour0}{\tgColour0}{\tgColour4}
        \tgBorderA{(2,2)}{\tgColour0}{\tgColour6}{\tgColour0}{\tgColour0}
        \tgBorderA{(3,2)}{\tgColour6}{\tgColour4}{\tgColour0}{\tgColour0}
        \tgBorderA{(4,2)}{\tgColour4}{\tgColour0}{\tgColour0}{\tgColour0}
        \tgBorder{(4,2)}{0}{1}{0}{0}
        \tgBorderA{(5,2)}{\tgColour0}{\tgColour8}{\tgColour8}{\tgColour0}
        \tgBorder{(5,2)}{0}{0}{0}{1}
        \tgBorderA{(0,3)}{\tgColour6}{\tgColour4}{\tgColour4}{\tgColour6}
        \tgBorderC{(1,3)}{0}{\tgColour4}{\tgColour0}
        \tgBorderA{(2,3)}{\tgColour0}{\tgColour0}{\tgColour4}{\tgColour4}
        \tgBorderA{(3,3)}{\tgColour0}{\tgColour0}{\tgColour4}{\tgColour4}
        \tgBorderC{(4,3)}{2}{\tgColour0}{\tgColour4}
        \tgBorderA{(5,3)}{\tgColour0}{\tgColour8}{\tgColour8}{\tgColour0}
        \tgBorderA{(0,4)}{\tgColour6}{\tgColour4}{\tgColour0}{\tgColour6}
        \tgBorderA{(1,4)}{\tgColour4}{\tgColour4}{\tgColour0}{\tgColour0}
        \tgBorderA{(2,4)}{\tgColour4}{\tgColour4}{\tgColour0}{\tgColour0}
        \tgBorderA{(3,4)}{\tgColour4}{\tgColour4}{\tgColour0}{\tgColour0}
        \tgBorderA{(4,4)}{\tgColour4}{\tgColour0}{\tgColour0}{\tgColour0}
        \tgBorder{(4,4)}{0}{1}{0}{0}
        \tgBorderA{(5,4)}{\tgColour0}{\tgColour8}{\tgColour8}{\tgColour0}
        \tgBorder{(5,4)}{0}{0}{0}{1}
        \tgBorderA{(0,5)}{\tgColour6}{\tgColour0}{\tgColour0}{\tgColour6}
        \tgBorderC{(1,5)}{3}{\tgColour0}{\tgColour4}
        \tgBorderA{(2,5)}{\tgColour0}{\tgColour0}{\tgColour4}{\tgColour4}
        \tgBorderA{(3,5)}{\tgColour0}{\tgColour0}{\tgColour4}{\tgColour4}
        \tgBorderC{(4,5)}{2}{\tgColour0}{\tgColour4}
        \tgBorderA{(5,5)}{\tgColour0}{\tgColour8}{\tgColour8}{\tgColour0}
        \tgBorderA{(0,6)}{\tgColour6}{\tgColour0}{\tgColour0}{\tgColour6}
        \tgBorderA{(1,6)}{\tgColour0}{\tgColour4}{\tgColour4}{\tgColour0}
        \tgBlank{(2,6)}{\tgColour4}
        \tgBlank{(3,6)}{\tgColour4}
        \tgBorderA{(4,6)}{\tgColour4}{\tgColour0}{\tgColour0}{\tgColour4}
        \tgBorderA{(5,6)}{\tgColour0}{\tgColour8}{\tgColour8}{\tgColour0}
        \tgCell[(1,0)]{(3.5,1)}{\sharp'}
        \tgArrow{(4,1.5)}{3}
        \tgArrow{(5,1.5)}{3}
        \tgArrow{(5,2.5)}{3}
        \tgArrow{(3,1.5)}{1}
        \tgArrow{(5,3.5)}{3}
        \tgArrow{(4,3.5)}{3}
        \tgCell[(3,0)]{(3.5,2)}{\unit_{\aop'}}
        \tgCell[(1,0)]{(0.5,2)}{\sharp'}
        \tgArrow{(5,0.5)}{3}
        \tgArrow{(1,2.5)}{3}
        \tgArrow{(5,4.5)}{3}
        \tgArrow{(3,0.5)}{1}
        \tgArrow{(2,0.5)}{1}
        \tgArrow{(2,1.5)}{1}
        \tgArrow{(0,0.5)}{1}
        \tgArrow{(0,1.5)}{1}
        \tgArrow{(1.5,3)}{0}
        \tgArrow{(2.5,3)}{0}
        \tgArrow{(0,2.5)}{1}
        \tgArrow{(0,3.5)}{1}
        \tgArrow{(0,4.5)}{1}
        \tgCell[(5,0)]{(2.5,4)}{\unit_{\aop'}}
        \tgArrow{(1.5,5)}{0}
        \tgArrow{(3.5,5)}{0}
        \tgArrow{(2.5,5)}{0}
        \tgArrow{(4,5.5)}{3}
        \tgArrow{(1,5.5)}{1}
        \tgArrow{(0,5.5)}{1}
        \tgArrow{(5,5.5)}{3}
        \tgArrow{(3.5,3)}{0}
        \tgAxisLabel{(0.5,0.75)}{south}{f_{T'}}
        \tgAxisLabel{(2.5,0.75)}{south}{t}
        \tgAxisLabel{(3.5,0.75)}{south}{f_{T'}}
        \tgAxisLabel{(5.5,0.75)}{south}{x}
        \tgAxisLabel{(0.5,6.25)}{north}{t}
        \tgAxisLabel{(1.5,6.25)}{north}{u_{T'}}
        \tgAxisLabel{(4.5,6.25)}{north}{u_{T'}}
        \tgAxisLabel{(5.5,6.25)}{north}{x}
    \end{tangle}
    \end{tangleeqs}

    Every $(p_1, \ldots, p_n)$-graded $(T \d u_{T'})$-algebra morphism $\epsilon \colon (e_1, \aop_1') \to (e_2, \aop_2')$, is also a $(p_1, \ldots, p_n)$-graded $T'$-algebra morphism between the induced $T'$-algebras. By the universal property of $\Alg(T')$, $\epsilon$ thus induces a 2-cell $\unit_\epsilon \colon \Alg(T')(1, x_1), p_1, \ldots, p_n \tto \Alg(T')(1, x_2)$. This is a $(p_1, \ldots, p_n)$-graded $T$-algebra morphism $(x_1, \aop_1) \to (x_2, \aop_2)$ since, by postcomposing $u_{T'}$, we have:
    
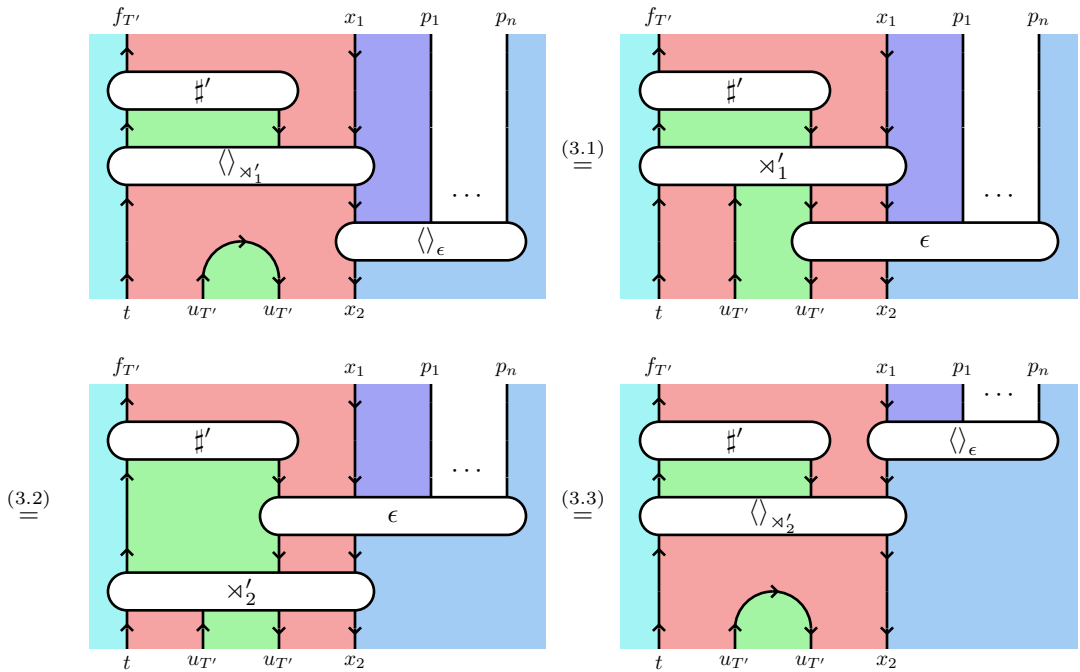
\begin{tangleeqs}
    \begin{tangle}{(6,5)}[trim y]
        \tgBorderA{(0,0)}{\tgColour6}{\tgColour0}{\tgColour0}{\tgColour6}
        \tgBlank{(1,0)}{\tgColour0}
        \tgBlank{(2,0)}{\tgColour0}
        \tgBorderA{(3,0)}{\tgColour0}{\tgColour8}{\tgColour8}{\tgColour0}
        \tgBorderA{(4,0)}{\tgColour8}{white}{white}{\tgColour8}
        \tgBorderA{(5,0)}{white}{\tgColour7}{\tgColour7}{white}
        \tgBorderA{(0,1)}{\tgColour6}{\tgColour0}{\tgColour4}{\tgColour6}
        \tgBorderA{(1,1)}{\tgColour0}{\tgColour0}{\tgColour4}{\tgColour4}
        \tgBorderA{(2,1)}{\tgColour0}{\tgColour0}{\tgColour0}{\tgColour4}
        \tgBorderA{(3,1)}{\tgColour0}{\tgColour8}{\tgColour8}{\tgColour0}
        \tgBorderA{(4,1)}{\tgColour8}{white}{white}{\tgColour8}
        \tgBorderA{(5,1)}{white}{\tgColour7}{\tgColour7}{white}
        \tgBorderA{(0,2)}{\tgColour6}{\tgColour4}{\tgColour0}{\tgColour6}
        \tgBorderA{(1,2)}{\tgColour4}{\tgColour4}{\tgColour0}{\tgColour0}
        \tgBorderA{(2,2)}{\tgColour4}{\tgColour0}{\tgColour0}{\tgColour0}
        \tgBorder{(2,2)}{0}{1}{0}{0}
        \tgBorderA{(3,2)}{\tgColour0}{\tgColour8}{\tgColour8}{\tgColour0}
        \tgBorder{(3,2)}{0}{0}{0}{1}
        \tgBorderA{(4,2)}{\tgColour8}{white}{white}{\tgColour8}
        \tgBorderA{(5,2)}{white}{\tgColour7}{\tgColour7}{white}
        \tgBorderA{(0,3)}{\tgColour6}{\tgColour0}{\tgColour0}{\tgColour6}
        \tgBorderC{(1,3)}{3}{\tgColour0}{\tgColour4}
        \tgBorderC{(2,3)}{2}{\tgColour0}{\tgColour4}
        \tgBorderA{(3,3)}{\tgColour0}{\tgColour8}{\tgColour7}{\tgColour0}
        \tgBorderA{(4,3)}{\tgColour8}{white}{\tgColour7}{\tgColour7}
        \tgBorderA{(5,3)}{white}{\tgColour7}{\tgColour7}{\tgColour7}
        \tgBorderA{(0,4)}{\tgColour6}{\tgColour0}{\tgColour0}{\tgColour6}
        \tgBorderA{(1,4)}{\tgColour0}{\tgColour4}{\tgColour4}{\tgColour0}
        \tgBorderA{(2,4)}{\tgColour4}{\tgColour0}{\tgColour0}{\tgColour4}
        \tgBorderA{(3,4)}{\tgColour0}{\tgColour7}{\tgColour7}{\tgColour0}
        \tgBlank{(4,4)}{\tgColour7}
        \tgBlank{(5,4)}{\tgColour7}
        \tgCell[(3,0)]{(1.5,2)}{\unit_{\aop_1'}}
        \tgArrow{(3,0.5)}{3}
        \tgArrow{(0,0.5)}{1}
        \tgArrow{(0,1.5)}{1}
        \tgArrow{(0,2.5)}{1}
        \tgArrow{(0,3.5)}{1}
        \tgArrow{(1,3.5)}{1}
        \tgArrow{(2,3.5)}{3}
        \tgArrow{(3,2.5)}{3}
        \tgArrow{(3,3.5)}{3}
        \tgArrow{(3,1.5)}{3}
        \tgArrow{(1.5,3)}{0}
        \tgCell[(2,0)]{(1,1)}{\sharp'}
        \tgArrow{(2,1.5)}{3}
        \tgCell[(2,0)]{(4,3)}{\unit_\epsilon}
        \tgAxisLabel{(0.5,0.75)}{south}{f_{T'}}
        \tgAxisLabel{(3.5,0.75)}{south}{x_1}
        \tgAxisLabel{(4.5,0.75)}{south}{p_1}
        \tgAxisLabel{(5.5,0.75)}{south}{p_n}
        \tgAxisLabel{(0.5,4.25)}{north}{t}
        \tgAxisLabel{(1.5,4.25)}{north}{u_{T'}}
        \tgAxisLabel{(2.5,4.25)}{north}{u_{T'}}
        \tgAxisLabel{(3.5,4.25)}{north}{x_2}
        \node at (5,2.9) {$\cdots$};
    \end{tangle}
    \=
    \begin{tangle}{(6,5)}[trim y]
        \tgBorderA{(0,0)}{\tgColour6}{\tgColour0}{\tgColour0}{\tgColour6}
        \tgBlank{(1,0)}{\tgColour0}
        \tgBlank{(2,0)}{\tgColour0}
        \tgBorderA{(3,0)}{\tgColour0}{\tgColour8}{\tgColour8}{\tgColour0}
        \tgBorderA{(4,0)}{\tgColour8}{white}{white}{\tgColour8}
        \tgBorderA{(5,0)}{white}{\tgColour7}{\tgColour7}{white}
        \tgBorderA{(0,1)}{\tgColour6}{\tgColour0}{\tgColour4}{\tgColour6}
        \tgBorderA{(1,1)}{\tgColour0}{\tgColour0}{\tgColour4}{\tgColour4}
        \tgBorderA{(2,1)}{\tgColour0}{\tgColour0}{\tgColour0}{\tgColour4}
        \tgBorderA{(3,1)}{\tgColour0}{\tgColour8}{\tgColour8}{\tgColour0}
        \tgBorderA{(4,1)}{\tgColour8}{white}{white}{\tgColour8}
        \tgBorderA{(5,1)}{white}{\tgColour7}{\tgColour7}{white}
        \tgBorderA{(0,2)}{\tgColour6}{\tgColour4}{\tgColour0}{\tgColour6}
        \tgBorderA{(1,2)}{\tgColour4}{\tgColour4}{\tgColour4}{\tgColour0}
        \tgBorder{(1,2)}{0}{1}{0}{0}
        \tgBorderA{(2,2)}{\tgColour4}{\tgColour0}{\tgColour0}{\tgColour4}
        \tgBorder{(2,2)}{0}{1}{0}{1}
        \tgBorderA{(3,2)}{\tgColour0}{\tgColour8}{\tgColour8}{\tgColour0}
        \tgBorder{(3,2)}{0}{0}{0}{1}
        \tgBorderA{(4,2)}{\tgColour8}{white}{white}{\tgColour8}
        \tgBorderA{(5,2)}{white}{\tgColour7}{\tgColour7}{white}
        \tgBorderA{(0,3)}{\tgColour6}{\tgColour0}{\tgColour0}{\tgColour6}
        \tgBorderA{(1,3)}{\tgColour0}{\tgColour4}{\tgColour4}{\tgColour0}
        \tgBorderA{(2,3)}{\tgColour4}{\tgColour0}{\tgColour0}{\tgColour4}
        \tgBorder{(2,3)}{0}{1}{0}{0}
        \tgBorderA{(3,3)}{\tgColour0}{\tgColour8}{\tgColour7}{\tgColour0}
        \tgBorder{(3,3)}{0}{0}{0}{1}
        \tgBorderA{(4,3)}{\tgColour8}{white}{\tgColour7}{\tgColour7}
        \tgBorderA{(5,3)}{white}{\tgColour7}{\tgColour7}{\tgColour7}
        \tgBorderA{(0,4)}{\tgColour6}{\tgColour0}{\tgColour0}{\tgColour6}
        \tgBorderA{(1,4)}{\tgColour0}{\tgColour4}{\tgColour4}{\tgColour0}
        \tgBorderA{(2,4)}{\tgColour4}{\tgColour0}{\tgColour0}{\tgColour4}
        \tgBorderA{(3,4)}{\tgColour0}{\tgColour7}{\tgColour7}{\tgColour0}
        \tgBlank{(4,4)}{\tgColour7}
        \tgBlank{(5,4)}{\tgColour7}
        \tgCell[(3,0)]{(1.5,2)}{\aop'_1}
        \tgArrow{(3,0.5)}{3}
        \tgArrow{(0,0.5)}{1}
        \tgArrow{(0,1.5)}{1}
        \tgArrow{(0,2.5)}{1}
        \tgArrow{(0,3.5)}{1}
        \tgArrow{(1,3.5)}{1}
        \tgArrow{(2,3.5)}{3}
        \tgArrow{(3,2.5)}{3}
        \tgArrow{(3,3.5)}{3}
        \tgArrow{(3,1.5)}{3}
        \tgArrow{(2,1.5)}{3}
        \tgCell[(2,0)]{(1,1)}{\sharp'}
        \tgArrow{(2,2.5)}{3}
        \tgArrow{(1,2.5)}{1}
        \tgCell[(3,0)]{(3.5,3)}{\epsilon}
        \tgAxisLabel{(0.5,0.75)}{south}{f_{T'}}
        \tgAxisLabel{(3.5,0.75)}{south}{x_1}
        \tgAxisLabel{(4.5,0.75)}{south}{p_1}
        \tgAxisLabel{(5.5,0.75)}{south}{p_n}
        \tgAxisLabel{(0.5,4.25)}{north}{t}
        \tgAxisLabel{(1.5,4.25)}{north}{u_{T'}}
        \tgAxisLabel{(2.5,4.25)}{north}{u_{T'}}
        \tgAxisLabel{(3.5,4.25)}{north}{x_2}
        \node at (5,2.9) {$\cdots$};
    \end{tangle}
    \\
    \begin{tangle}{(6,5)}[trim y]
        \tgBorderA{(0,0)}{\tgColour6}{\tgColour0}{\tgColour0}{\tgColour6}
        \tgBlank{(1,0)}{\tgColour0}
        \tgBlank{(2,0)}{\tgColour0}
        \tgBorderA{(3,0)}{\tgColour0}{\tgColour8}{\tgColour8}{\tgColour0}
        \tgBorderA{(4,0)}{\tgColour8}{white}{white}{\tgColour8}
        \tgBorderA{(5,0)}{white}{\tgColour7}{\tgColour7}{white}
        \tgBorderA{(0,1)}{\tgColour6}{\tgColour0}{\tgColour4}{\tgColour6}
        \tgBorderA{(1,1)}{\tgColour0}{\tgColour0}{\tgColour4}{\tgColour4}
        \tgBorderA{(2,1)}{\tgColour0}{\tgColour0}{\tgColour0}{\tgColour4}
        \tgBorderA{(3,1)}{\tgColour0}{\tgColour8}{\tgColour8}{\tgColour0}
        \tgBorderA{(4,1)}{\tgColour8}{white}{white}{\tgColour8}
        \tgBorderA{(5,1)}{white}{\tgColour7}{\tgColour7}{white}
        \tgBorderA{(0,2)}{\tgColour6}{\tgColour4}{\tgColour4}{\tgColour6}
        \tgBlank{(1,2)}{\tgColour4}
        \tgBorderA{(2,2)}{\tgColour4}{\tgColour0}{\tgColour0}{\tgColour4}
        \tgBorder{(2,2)}{0}{1}{0}{0}
        \tgBorderA{(3,2)}{\tgColour0}{\tgColour8}{\tgColour7}{\tgColour0}
        \tgBorder{(3,2)}{0}{0}{0}{1}
        \tgBorderA{(4,2)}{\tgColour8}{white}{\tgColour7}{\tgColour7}
        \tgBorderA{(5,2)}{white}{\tgColour7}{\tgColour7}{\tgColour7}
        \tgBorderA{(0,3)}{\tgColour6}{\tgColour4}{\tgColour0}{\tgColour6}
        \tgBorderA{(1,3)}{\tgColour4}{\tgColour4}{\tgColour4}{\tgColour0}
        \tgBorder{(1,3)}{0}{1}{0}{0}
        \tgBorderA{(2,3)}{\tgColour4}{\tgColour0}{\tgColour0}{\tgColour4}
        \tgBorder{(2,3)}{0}{1}{0}{1}
        \tgBorderA{(3,3)}{\tgColour0}{\tgColour7}{\tgColour7}{\tgColour0}
        \tgBorder{(3,3)}{0}{0}{0}{1}
        \tgBlank{(4,3)}{\tgColour7}
        \tgBlank{(5,3)}{\tgColour7}
        \tgBorderA{(0,4)}{\tgColour6}{\tgColour0}{\tgColour0}{\tgColour6}
        \tgBorderA{(1,4)}{\tgColour0}{\tgColour4}{\tgColour4}{\tgColour0}
        \tgBorderA{(2,4)}{\tgColour4}{\tgColour0}{\tgColour0}{\tgColour4}
        \tgBorderA{(3,4)}{\tgColour0}{\tgColour7}{\tgColour7}{\tgColour0}
        \tgBlank{(4,4)}{\tgColour7}
        \tgBlank{(5,4)}{\tgColour7}
        \tgArrow{(3,0.5)}{3}
        \tgArrow{(0,0.5)}{1}
        \tgArrow{(0,1.5)}{1}
        \tgArrow{(0,2.5)}{1}
        \tgArrow{(0,3.5)}{1}
        \tgArrow{(1,3.5)}{1}
        \tgArrow{(2,3.5)}{3}
        \tgArrow{(3,2.5)}{3}
        \tgArrow{(3,3.5)}{3}
        \tgArrow{(3,1.5)}{3}
        \tgArrow{(2,1.5)}{3}
        \tgCell[(2,0)]{(1,1)}{\sharp'}
        \tgCell[(3,0)]{(1.5,3)}{\aop_2'}
        \tgArrow{(2,2.5)}{3}
        \tgCell[(3,0)]{(3.5,2)}{\epsilon}
        \tgAxisLabel{(0.5,0.75)}{south}{f_{T'}}
        \tgAxisLabel{(3.5,0.75)}{south}{x_1}
        \tgAxisLabel{(4.5,0.75)}{south}{p_1}
        \tgAxisLabel{(5.5,0.75)}{south}{p_n}
        \tgAxisLabel{(0.5,4.25)}{north}{t}
        \tgAxisLabel{(1.5,4.25)}{north}{u_{T'}}
        \tgAxisLabel{(2.5,4.25)}{north}{u_{T'}}
        \tgAxisLabel{(3.5,4.25)}{north}{x_2}
        \node at (5,1.9) {$\cdots$};
    \end{tangle}
    \=
    \begin{tangle}{(6,5)}[trim y]
        \tgBorderA{(0,0)}{\tgColour6}{\tgColour0}{\tgColour0}{\tgColour6}
        \tgBlank{(1,0)}{\tgColour0}
        \tgBlank{(2,0)}{\tgColour0}
        \tgBorderA{(3,0)}{\tgColour0}{\tgColour8}{\tgColour8}{\tgColour0}
        \tgBorderA{(4,0)}{\tgColour8}{white}{white}{\tgColour8}
        \tgBorderA{(5,0)}{white}{\tgColour7}{\tgColour7}{white}
        \tgBorderA{(0,1)}{\tgColour6}{\tgColour0}{\tgColour4}{\tgColour6}
        \tgBorderA{(1,1)}{\tgColour0}{\tgColour0}{\tgColour4}{\tgColour4}
        \tgBorderA{(2,1)}{\tgColour0}{\tgColour0}{\tgColour0}{\tgColour4}
        \tgBorderA{(3,1)}{\tgColour0}{\tgColour8}{\tgColour7}{\tgColour0}
        \tgBorderA{(4,1)}{\tgColour8}{white}{\tgColour7}{\tgColour7}
        \tgBorderA{(5,1)}{white}{\tgColour7}{\tgColour7}{\tgColour7}
        \tgBorderA{(0,2)}{\tgColour6}{\tgColour4}{\tgColour0}{\tgColour6}
        \tgBorderA{(1,2)}{\tgColour4}{\tgColour4}{\tgColour0}{\tgColour0}
        \tgBorderA{(2,2)}{\tgColour4}{\tgColour0}{\tgColour0}{\tgColour0}
        \tgBorder{(2,2)}{0}{1}{0}{0}
        \tgBorderA{(3,2)}{\tgColour0}{\tgColour7}{\tgColour7}{\tgColour0}
        \tgBorder{(3,2)}{0}{0}{0}{1}
        \tgBlank{(4,2)}{\tgColour7}
        \tgBlank{(5,2)}{\tgColour7}
        \tgBorderA{(0,3)}{\tgColour6}{\tgColour0}{\tgColour0}{\tgColour6}
        \tgBorderC{(1,3)}{3}{\tgColour0}{\tgColour4}
        \tgBorderC{(2,3)}{2}{\tgColour0}{\tgColour4}
        \tgBorderA{(3,3)}{\tgColour0}{\tgColour7}{\tgColour7}{\tgColour0}
        \tgBlank{(4,3)}{\tgColour7}
        \tgBlank{(5,3)}{\tgColour7}
        \tgBorderA{(0,4)}{\tgColour6}{\tgColour0}{\tgColour0}{\tgColour6}
        \tgBorderA{(1,4)}{\tgColour0}{\tgColour4}{\tgColour4}{\tgColour0}
        \tgBorderA{(2,4)}{\tgColour4}{\tgColour0}{\tgColour0}{\tgColour4}
        \tgBorderA{(3,4)}{\tgColour0}{\tgColour7}{\tgColour7}{\tgColour0}
        \tgBlank{(4,4)}{\tgColour7}
        \tgBlank{(5,4)}{\tgColour7}
        \tgArrow{(0,1.5)}{1}
        \tgArrow{(0,2.5)}{1}
        \tgArrow{(3,1.5)}{3}
        \tgArrow{(3,2.5)}{3}
        \tgCell[(3,0)]{(1.5,2)}{\unit_{\aop_2'}}
        \tgArrow{(2,1.5)}{3}
        \tgCell[(2,0)]{(1,1)}{\sharp'}
        \tgArrow{(1.5,3)}{0}
        \tgArrow{(2,3.5)}{3}
        \tgArrow{(3,3.5)}{3}
        \tgArrow{(3,0.5)}{3}
        \tgArrow{(0,0.5)}{1}
        \tgArrow{(0,3.5)}{1}
        \tgArrow{(1,3.5)}{1}
        \tgCell[(2,0)]{(4,1)}{\unit_\epsilon}
        \tgAxisLabel{(0.5,0.75)}{south}{f_{T'}}
        \tgAxisLabel{(3.5,0.75)}{south}{x_1}
        \tgAxisLabel{(4.5,0.75)}{south}{p_1}
        \tgAxisLabel{(5.5,0.75)}{south}{p_n}
        \tgAxisLabel{(0.5,4.25)}{north}{t}
        \tgAxisLabel{(1.5,4.25)}{north}{u_{T'}}
        \tgAxisLabel{(2.5,4.25)}{north}{u_{T'}}
        \tgAxisLabel{(3.5,4.25)}{north}{x_2}
        \node at (5,0.9) {$\cdots$};
    \end{tangle}
    \end{tangleeqs}
    \eqstepref{3.1} by the definitions of $\unit_{\aop_1'}$ and $\unit_\epsilon$; \eqstepref{3.2} by the compatibility law for $\epsilon$; and \eqstepref{3.3} by the definitions of $\unit_{\aop_2'}$ and $\unit_\epsilon$.

    We must now show that these assignments are inverse to one another. Suppose we have a $T$-algebra $(x, \aop)$, inducing a $(T \d u_{T'})$-algebra $((e \d u_{T'}), \aop')$. We must show that the $T'$-algebra induced by $x$ coincides with that induced by the $(T \d u_{T'})$-algebra. By definition $(x \d u_{T'})$ is the carrier of both $T'$-algebras. Using the unit law for $\aop$, the latter $T'$-algebra structure is given by
    
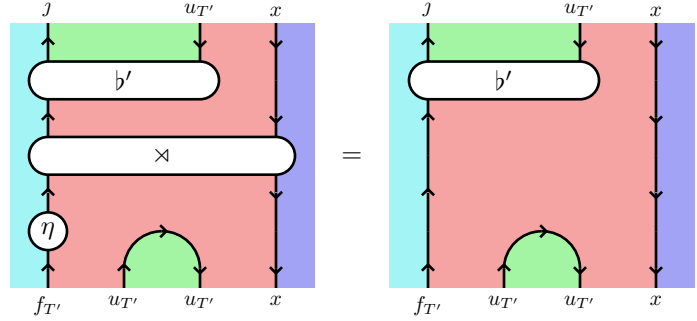
\begin{tangleeqs*}
    \begin{tangle}{(4,5)}[trim y]
        \tgBorderA{(0,0)}{\tgColour6}{\tgColour4}{\tgColour4}{\tgColour6}
        \tgBlank{(1,0)}{\tgColour4}
        \tgBorderA{(2,0)}{\tgColour4}{\tgColour0}{\tgColour0}{\tgColour4}
        \tgBorderA{(3,0)}{\tgColour0}{\tgColour8}{\tgColour8}{\tgColour0}
        \tgBorderA{(0,1)}{\tgColour6}{\tgColour4}{\tgColour0}{\tgColour6}
        \tgBorderA{(1,1)}{\tgColour4}{\tgColour4}{\tgColour0}{\tgColour0}
        \tgBorderA{(2,1)}{\tgColour4}{\tgColour0}{\tgColour0}{\tgColour0}
        \tgBorderA{(3,1)}{\tgColour0}{\tgColour8}{\tgColour8}{\tgColour0}
        \tgBorderA{(0,2)}{\tgColour6}{\tgColour0}{\tgColour0}{\tgColour6}
        \tgBorder{(0,2)}{0}{1}{0}{0}
        \tgBorderA{(1,2)}{\tgColour0}{\tgColour0}{\tgColour0}{\tgColour0}
        \tgBorder{(1,2)}{0}{1}{0}{1}
        \tgBorderA{(2,2)}{\tgColour0}{\tgColour0}{\tgColour0}{\tgColour0}
        \tgBorder{(2,2)}{0}{1}{0}{1}
        \tgBorderA{(3,2)}{\tgColour0}{\tgColour8}{\tgColour8}{\tgColour0}
        \tgBorder{(3,2)}{0}{0}{0}{1}
        \tgBorderA{(0,3)}{\tgColour6}{\tgColour0}{\tgColour0}{\tgColour6}
        \tgBorderC{(1,3)}{3}{\tgColour0}{\tgColour4}
        \tgBorderC{(2,3)}{2}{\tgColour0}{\tgColour4}
        \tgBorderA{(3,3)}{\tgColour0}{\tgColour8}{\tgColour8}{\tgColour0}
        \tgBorderA{(0,4)}{\tgColour6}{\tgColour0}{\tgColour0}{\tgColour6}
        \tgBorderA{(1,4)}{\tgColour0}{\tgColour4}{\tgColour4}{\tgColour0}
        \tgBorderA{(2,4)}{\tgColour4}{\tgColour0}{\tgColour0}{\tgColour4}
        \tgBorderA{(3,4)}{\tgColour0}{\tgColour8}{\tgColour8}{\tgColour0}
        \tgCell[(2,0)]{(1,1)}{\flat'}
        \tgCell[(3,0)]{(1.5,2)}{\aop}
        \tgCell{(0,3)}{\eta}
        \tgArrow{(3,0.5)}{3}
        \tgArrow{(3,1.5)}{3}
        \tgArrow{(3,2.5)}{3}
        \tgArrow{(3,3.5)}{3}
        \tgArrow{(1.5,3)}{0}
        \tgArrow{(2,3.5)}{3}
        \tgArrow{(1,3.5)}{1}
        \tgArrow{(0,1.5)}{1}
        \tgArrow{(0,2.5)}{1}
        \tgArrow{(0,3.5)}{1}
        \tgArrow{(0,0.5)}{1}
        \tgArrow{(2,0.5)}{3}
        \tgAxisLabel{(0.5,0.75)}{south}{j}
        \tgAxisLabel{(2.5,0.75)}{south}{u_{T'}}
        \tgAxisLabel{(3.5,0.75)}{south}{x}
        \tgAxisLabel{(0.5,4.25)}{north}{f_{T'}}
        \tgAxisLabel{(1.5,4.25)}{north}{u_{T'}}
        \tgAxisLabel{(2.5,4.25)}{north}{u_{T'}}
        \tgAxisLabel{(3.5,4.25)}{north}{x}
    \end{tangle}
    \=
    \begin{tangle}{(4,5)}[trim y]
        \tgBorderA{(0,0)}{\tgColour6}{\tgColour4}{\tgColour4}{\tgColour6}
        \tgBlank{(1,0)}{\tgColour4}
        \tgBorderA{(2,0)}{\tgColour4}{\tgColour0}{\tgColour0}{\tgColour4}
        \tgBorderA{(3,0)}{\tgColour0}{\tgColour8}{\tgColour8}{\tgColour0}
        \tgBorderA{(0,1)}{\tgColour6}{\tgColour4}{\tgColour0}{\tgColour6}
        \tgBorderA{(1,1)}{\tgColour4}{\tgColour4}{\tgColour0}{\tgColour0}
        \tgBorderA{(2,1)}{\tgColour4}{\tgColour0}{\tgColour0}{\tgColour0}
        \tgBorderA{(3,1)}{\tgColour0}{\tgColour8}{\tgColour8}{\tgColour0}
        \tgBorderA{(0,2)}{\tgColour6}{\tgColour0}{\tgColour0}{\tgColour6}
        \tgBlank{(1,2)}{\tgColour0}
        \tgBlank{(2,2)}{\tgColour0}
        \tgBorderA{(3,2)}{\tgColour0}{\tgColour8}{\tgColour8}{\tgColour0}
        \tgBorderA{(0,3)}{\tgColour6}{\tgColour0}{\tgColour0}{\tgColour6}
        \tgBorderC{(1,3)}{3}{\tgColour0}{\tgColour4}
        \tgBorderC{(2,3)}{2}{\tgColour0}{\tgColour4}
        \tgBorderA{(3,3)}{\tgColour0}{\tgColour8}{\tgColour8}{\tgColour0}
        \tgBorderA{(0,4)}{\tgColour6}{\tgColour0}{\tgColour0}{\tgColour6}
        \tgBorderA{(1,4)}{\tgColour0}{\tgColour4}{\tgColour4}{\tgColour0}
        \tgBorderA{(2,4)}{\tgColour4}{\tgColour0}{\tgColour0}{\tgColour4}
        \tgBorderA{(3,4)}{\tgColour0}{\tgColour8}{\tgColour8}{\tgColour0}
        \tgCell[(2,0)]{(1,1)}{\flat'}
        \tgArrow{(3,0.5)}{3}
        \tgArrow{(3,1.5)}{3}
        \tgArrow{(3,2.5)}{3}
        \tgArrow{(3,3.5)}{3}
        \tgArrow{(1.5,3)}{0}
        \tgArrow{(2,3.5)}{3}
        \tgArrow{(1,3.5)}{1}
        \tgArrow{(0,1.5)}{1}
        \tgArrow{(0,2.5)}{1}
        \tgArrow{(0,3.5)}{1}
        \tgArrow{(0,0.5)}{1}
        \tgArrow{(2,0.5)}{3}
        \tgAxisLabel{(0.5,0.75)}{south}{j}
        \tgAxisLabel{(2.5,0.75)}{south}{u_{T'}}
        \tgAxisLabel{(3.5,0.75)}{south}{x}
        \tgAxisLabel{(0.5,4.25)}{north}{f_{T'}}
        \tgAxisLabel{(1.5,4.25)}{north}{u_{T'}}
        \tgAxisLabel{(2.5,4.25)}{north}{u_{T'}}
        \tgAxisLabel{(3.5,4.25)}{north}{x}
    \end{tangle}
    \end{tangleeqs*}
    which is the former $T'$-algebra structure.
    That the $T$-algebra structure induced by $\aop'$ coincides with $\aop$ follows from the universal property of $\Alg(T')$ with respect to graded algebra morphisms.

    Suppose we have a $(p_1, \ldots, p_n)$-graded $T$-algebra morphism $\chi \colon (x_1, \aop_1) \to (x_2, \aop_2)$. The $T'$-algebra morphism induced directly by $\chi$, and the $T'$-algebra morphism induced from the $(T \d u_{T'})$-algebra morphism are both given by $(\chi \d u_{T'})$ by definition. Thus the universal property of $\Alg(T')$ with respect to graded algebra morphisms implies that the induced $T$-algebra morphism is $\chi$.

    The other direction follows immediately from the universal property of $\Alg(T')$ (with respect both to algebras and graded algebra morphisms).
\end{proof}

\begin{definition}
    A relative adjunction $\ljr$ inducing a relative monad $T$ is \emph{strictly $j$-relatively monadic} (alternatively \emph{strictly monadic relative to $j$}, or simply \emph{strictly $j$-monadic}) if $T$ admits an algebra object and $\ljr$ is isomorphic to $f_T \jadj u_T$ in the category of resolutions of $T$.
\end{definition}

Note that a tight-cell $r$ is strictly $j$-monadic if and only if it admits a left $j$-adjoint and the induced $j$-adjunction is strictly $j$-monadic, since both properties exhibit $r$ as being the right $j$-adjoint of a terminal resolution.

It follows from \cref{pasting-law-algebras} that the pasting law for relative adjunctions (\cref{relative-adjunction-pasting}) respects relative monadicity.

\begin{lemma}
    \label{monadic-relative-adjunction-pasting}
    Consider the following diagram.
    
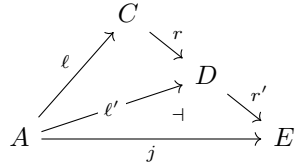
\[\begin{tikzcd}[sep=small]
        && C \\
        &&& D \\
        A &&&& E
        \arrow["r", from=1-3, to=2-4]
        \arrow["j"', from=3-1, to=3-5]
        \arrow[""{name=0, anchor=center, inner sep=0}, "{r'}", from=2-4, to=3-5]
        \arrow[""{name=1, anchor=center, inner sep=0}, "{\ell'}"{description}, from=3-1, to=2-4]
        \arrow["\ell", from=3-1, to=1-3]
        \arrow["\dashv"{anchor=center}, shift right=1, draw=none, from=1, to=0]
    \end{tikzcd}\]
    Suppose that $\ell' \jadj r'$ is strictly $j$-monadic. Then the left triangle exhibits a strictly $\ell'$-monadic relative adjunction if and only if the outer triangle exhibits a strictly $j$-monadic relative adjunction.
\end{lemma}

\begin{proof}
    Follows directly from \cref{relative-adjunction-pasting,pasting-law-algebras}, the latter exhibiting the universal properties of $\Alg(T)$ and $\Alg(T \d r')$ as being equivalent.
\end{proof}

We rephrase \cref{monadic-relative-adjunction-pasting} to concern only the tight-cells $r$ and $r'$, rather than relative adjunctions, which gives necessary and sufficient conditions for the composite $(r \d r')$ to be relatively monadic.

\begin{theorem}[Relative monadicity of composites]
    \label{relative-monadicity-of-composite}
    Let $\jAE$ and $r' \colon D \to E$ be tight-cells, and assume that $r'$ is (non)strictly $j$-monadic, with left $j$-adjoint $\ell'$.
    A tight-cell $r \colon C \to D$ is (non)strictly $\ell'$-monadic if and only if $(r \d r') \colon C \to E$ is (non)strictly $j$-monadic.
\end{theorem}

\begin{proof}
    Let $T'$ be the $j$-monad induced by $\ell' \jadj r'$.
    Without loss of generality, we may assume $\ell' \jadj r'$ is strictly $j$-monadic: the following diagram commutes, so that non-strict $\ell'$-monadicity of $r$ is equivalently non-strict $f_{T'}$-monadicity of $(r \d {\equiv})$, and non-strict $j$-monadicity of $(r \d r')$ is equivalently non-strict $j$-monadicity of $(r \d {\equiv} \d u_{T'})$.
\[\begin{tikzcd}
	& D \\
	& {\Alg(T')} \\
	A && E
	\arrow["j"', from=3-1, to=3-3]
	\arrow[""{name=0, anchor=center, inner sep=0}, "{u_{T'}}"{description}, from=2-2, to=3-3]
	\arrow[""{name=1, anchor=center, inner sep=0}, "{f_{T'}}"{description}, from=3-1, to=2-2]
	\arrow["\equiv"{description}, from=1-2, to=2-2]
	\arrow["{r'}", curve={height=-12pt}, from=1-2, to=3-3]
	\arrow["{\ell'}", curve={height=-12pt}, from=3-1, to=1-2]
	\arrow["\dashv"{anchor=center}, shift right=1, draw=none, from=1, to=0]
\end{tikzcd}\]
    The strict case then follows directly from \cref{monadic-relative-adjunction-pasting}, considering the following triangle, where $\ell \colon A \to C$ is the left relative adjoint of either $r$ or of $(r \d u_{T'})$.
\[\begin{tikzcd}[sep=small]
	&& C \\
	&&& {\Alg(T')} \\
	A &&&& E
	\arrow["r", from=1-3, to=2-4]
	\arrow["j"', from=3-1, to=3-5]
	\arrow[""{name=0, anchor=center, inner sep=0}, "{u_{T'}}", from=2-4, to=3-5]
	\arrow[""{name=1, anchor=center, inner sep=0}, "{f_{T'}}"{description}, from=3-1, to=2-4]
	\arrow["\ell", from=3-1, to=1-3]
	\arrow["\dashv"{anchor=center}, draw=none, from=1, to=0]
\end{tikzcd}\]
    The non-strict case follows therefrom by precomposing $r$ by equivalences.
\end{proof}

\Cref{relative-monadicity-of-composite} may be seen as a significant generalisation of classical sufficiency results for the monadicity of a composite $(r \d r')$ functor (\eg{}~\cite[Proposition~3.5.1]{barr1985toposes}). When $r$ is \ff{}, it may alternatively be seen as a weakening of \citeauthor{birkhoff1935structure}'s \emph{HSP theorem}~\cite[Theorem~10]{birkhoff1935structure}, giving necessary and sufficient conditions for a subcategory of a category of algebras to itself be a category of algebras (the HSP theorem further asks that, in this case, the induced monad morphism be a regular epimorphism, which expresses that the subcategory is formed by imposing additional axioms on the algebras).\footnotemark{}
\footnotetext{While we expect that the HSP theorem may be derived from \cref{relative-monadicity-of-composite}, we shall not do so here.}

To conclude, we mention two useful consequences of \cref{relative-monadicity-of-composite}. The first is the following classical cancellability result for monadic adjunctions (\cf{}~\cite[Proposition~5]{bourn1992low}).

\begin{corollary}[Cancellability]
    Consider the following situation, in which the monad induced by $\ell \adj r$ admits an algebra object.
    \[\begin{tikzcd}
    	C & D & E
    	\arrow[""{name=0, anchor=center, inner sep=0}, "r", shift left=2, from=1-1, to=1-2]
    	\arrow[""{name=1, anchor=center, inner sep=0}, "{r'}", shift left=2, from=1-2, to=1-3]
    	\arrow[""{name=2, anchor=center, inner sep=0}, "{\ell'}", shift left=2, from=1-3, to=1-2]
    	\arrow[""{name=3, anchor=center, inner sep=0}, "\ell", shift left=2, from=1-2, to=1-1]
    	\arrow["\dashv"{anchor=center, rotate=90}, draw=none, from=3, to=0]
    	\arrow["\dashv"{anchor=center, rotate=90}, draw=none, from=2, to=1]
    \end{tikzcd}\]
    If $r'$ and $(r \d r')$ are (non)strictly monadic, then so is $r$.
\end{corollary}

\begin{proof}
    From \cref{relative-monadicity-of-composite}, we have that $r'$ is (non)strictly $\ell'$-monadic. The result then follows from \cref{monadic-iff-left-adjoint}(\hyperref[monadic-iff-left-adjoint']{$'$}), since the monad induced by $\ell \adj r$ admits an algebra object.
\end{proof}

The second establishes that algebraic tight-cells -- that is, those tight-cells between algebra objects that are induced by relative monad morphisms -- create limits and certain colimits, and are themselves monadic as soon as they admit left adjoints.

\begin{corollary}
    \label{algebraic-tight-cells-create-limits-and-j-absolute-colimits}
    Let $\jAE$ be a tight-cell, and consider a commutative triangle as follows, where $r$ and $r'$ are (non)strictly $j$-monadic.
    \[\begin{tikzcd}
        {D'} && D \\
        & E
        \arrow["i", from=1-1, to=1-3]
        \arrow["{r'}"', from=1-1, to=2-2]
        \arrow["r", from=1-3, to=2-2]
    \end{tikzcd}\]
    Then:
    \begin{enumerate}
      \item $i$ (non)strictly creates limits and those colimits in $D$ that are sent by $r$ to $j$-absolute colimits.
      \item $i$ is (non)strictly monadic if and only if it admits a left adjoint.
    \end{enumerate}
\end{corollary}

\begin{proof}
    Denote by $\ell$ the left $j$-adjoint of $r$.
    By \cref{relative-monadicity-of-composite}, $i$ is (non)strictly $\ell$-monadic.

    For (1), $i$ (non)strictly creates limits and $\ell$-absolute colimits by \cref{u_T-creates-j-absolute-colimits,u_T-creates-limits,u_T-nonstrict-creation}. In particular, the (non)strict lifting of any $j$-absolute colimit through $r$ is $\ell$-absolute: for each loose-cell $p \colon Y \lto Z$ and tight-cell $f \colon Z \to D$ admitting a $j$-absolute colimit $p \wc (f \d r)$, we have
    \begin{align*}
        D(\ell, f) \odotl p & \iso E(j, r f) \odotl p \tag{$\ell \jadj r$} \\
            & \iso E(j, p \wc (f \d r)) \tag{$p \wc (f \d r)$ is $j$-absolute} \\
            & \iso E(j, r (p \wc f)) \tag{\cref{u_T-creates-j-absolute-colimits,u_T-nonstrict-creation}} \\
            & \iso D(\ell, p \wc f) \tag{$\ell \jadj r$}
    \end{align*}
    with canonicity of this isomorphism following from that for $j$-absoluteness.
    (Alternatively, we could establish (1) as a consequence of \cref{u_T-creates-j-absolute-colimits,u_T-creates-limits,u_T-nonstrict-creation} together with simple observations about the creation of limits and colimits, but find the approach above to be particularly satisfying.)

    (2) follows immediately from $\ell$-monadicity of $i$ by \cref{monadic-iff-left-adjoint}(\hyperref[monadic-iff-left-adjoint']{$'$}).
\end{proof}

The preservation of limits and certain colimits in \cref{algebraic-tight-cells-create-limits-and-j-absolute-colimits} is often sufficient, via an adjointness theorem, to imply that, for sufficiently nice roots $j$, every algebraic tight-cell admits a right adjoint, and hence is monadic.
We end with an example of this situation.

\begin{example}
     From \cref{algebraic-theories-induce-monads}, we have that the category of algebras $\Cart[L\op, \Set]$ of a finitary algebraic theory $\ell \colon \F \to L$ is $(j \colon \F \ffto \Set)$-monadic. Since $u_\ell$ and $n_j$ preserve sifted colimits, sifted colimits in $\Cart[L\op, \Set]$ are lifted non-strictly by $u_\ell$ via \cref{j-absolute-is-preservation-by-n_j,u_T-nonstrict-creation}. Consequently, every concrete functor $i$ between categories of algebras for finitary algebraic theories preserves limits and sifted colimits.
\[\begin{tikzcd}
	{\Cart[L'\op, \Set]} && {\Cart[L\op, \Set]} \\
	& \Set
	\arrow[""{name=0, anchor=center, inner sep=0}, "i"', shift right=2, from=1-1, to=1-3]
	\arrow["{u_{\ell'}}"', from=1-1, to=2-2]
	\arrow["{u_\ell}", from=1-3, to=2-2]
	\arrow[""{name=1, anchor=center, inner sep=0}, shift right=2, dashed, from=1-3, to=1-1]
	\arrow["\dashv"{anchor=center, rotate=-90}, draw=none, from=1, to=0]
\end{tikzcd}\]
    Since categories of algebras for algebraic theories are locally strongly finitely presentable, and every continuous and sifted-cocontinuous functor therebetween is a right adjoint, every such $i$ is a right adjoint, and hence is non-strictly monadic (\cf{}~\cites[Theorem~IV.2.1]{lawvere1963functorial}[Corollary~1.5.2]{manes1967triple}).
\end{example}

\printbibliography

\end{document}